\documentclass[12pt,twoside,a4paper]{article}

\usepackage{amsmath,amsthm,amssymb,tikz-cd,graphicx,hyperref}

\newtheorem{theorem}{Theorem}[section]					%Theorem environments
\newtheorem{lemma}[theorem]{Lemma}
\newtheorem{corollary}[theorem]{Corollary}
\newtheorem{proposition}[theorem]{Proposition}
\theoremstyle{definition}
\newtheorem{definition}[theorem]{Definition}
\theoremstyle{definition}

\newtheorem{hypotheses}[theorem]{Hypotheses}

\newcommand{\N}{\mathbb{N}}								%Sets of numbers
\newcommand{\Z}{\mathbb{Z}}

\newcommand{\K}{\mathbb{K}}

\newcommand{\OO}{\mathcal{O}}

\newcommand{\catname}[1]{{\normalfont\textbf{#1}}}		%Category names

\newcommand{\bimod}[2]{{}_{#1}\catname{mod}_{#2}}

\newcommand{\lmod}[1]{{}_{#1}\catname{mod}}

\newcommand{\ltriv}[1]{{}_{#1}\catname{triv}}

\newcommand{\Aut}{\operatorname{Aut}}
\newcommand{\BE}{\mathcal{BE}}
\newcommand{\Bl}{\operatorname{Bl}}
\newcommand{\bli}{\operatorname{bli}}
\newcommand{\BP}{\mathcal{BP}}
\newcommand{\br}{\operatorname{br}}

\newcommand{\CentFun}{\operatorname{CentFun}}

\newcommand{\dsum}{\oplus}

\newcommand{\extens}[2]{\overset{#2}{\underset{#1}{\tensor}}}

\newcommand{\gp}[1]{\langle#1\rangle}
\newcommand{\Hom}{\operatorname{Hom}}
\newcommand{\IBr}{\operatorname{IBr}}
\newcommand{\id}{\operatorname{id}}

\newcommand{\Ind}{\operatorname{Ind}}

\newcommand{\into}{\hookrightarrow}
\newcommand{\Irr}{\operatorname{Irr}}
\newcommand{\iso}{\cong}

\newcommand{\isoto}{\overset{\sim}{\to}}

\newcommand{\nor}{\trianglelefteq}

\newcommand{\onto}{\twoheadrightarrow}

\newcommand{\Orb}{\operatorname{Orb}}

\newcommand{\PrInd}{\operatorname{PrInd}}

\newcommand{\Res}{\operatorname{Res}}

\newcommand{\set}[1]{\left\{#1\right\}}

\newcommand{\Stab}{\operatorname{Stab}}
\newcommand{\subgp}{\leq}

\newcommand{\tensor}{\otimes}

\newcommand{\tr}{\operatorname{tr}}

\makeatletter
\newcommand{\tpitchfork}{%
	\vbox{
		\baselineskip\z@skip
		\lineskip-.52ex
		\lineskiplimit\maxdimen
		\m@th
		\ialign{##\crcr\hidewidth\smash{$-$}\hidewidth\crcr$\pitchfork$\crcr}
	}%
}
\makeatother

\title{Coherence conditions for the characters of trivial source modules and strong isotypies}
\author{John Revere McHugh}

\begin{document}
	
	\maketitle
	
	\begin{abstract}
		We introduce a new type of equivalence between blocks of finite group algebras called a \textit{strong isotypy}. A strong isotypy is equivalent to a $p$-permutation equivalence and restricts to an isotypy in the sense of Brou\'{e}. To prove these results we first establish that the group $T_{\OO}(B)$ of trivial source $B$-modules, where $B$ is a block of a finite group algebra, is isomorphic to groups of ``coherent character tuples.'' This provides a refinement of work by Boltje and Carman which characterizes the ring $T_{\OO}(G)$ of trivial source $\OO G$-modules, where $G$ is a finite group, in terms of coherent character tuples.
	\end{abstract}
	
	\section{Introduction}
	
	Let $G$ be a finite group, let $p$ be a prime number, and let $(\K,\OO,F)$ be a $p$-modular system large enough for $G$. Let $A$ be a block of $\OO G$ with defect group $D$ and let $B$ denote the Brauer correspondent of $A$ (which, we recall, is a block of $\OO N_G(D)$). In \cite{Broue_1990} Brou\'{e} conjectured that if $D$ is abelian then the bounded derived categories $D^b(\lmod{A})$ and $D^b(\lmod{B})$ are equivalent. Later, in \cite{Rickard_1996}, Rickard refined Brou\'{e}'s conjecture by postulating the existence of a special type of derived equivalence between $A$ and $B$: what is known now as a splendid Rickard equivalence. Turning to Grothendieck groups, a splendid Rickard equivalence induces a $p$\textit-permutation equivalence between $A$ and $B$ as defined by Boltje and Perepelitsky in \cite{Boltje_2020}. The existence of a splendid Rickard equivalence or of a $p$-permutation equivalence in the situation of the abelian defect group conjecture would provide an explanation for the phenomenon of an isotypy between $A$ and $B$: a block equivalence defined at the ``character level'' which has been observed in all examples computed to date. 
	
	The main aim of the present article is to provide a closer examination of the construction of an isotypy from a $p$-permutation equivalence (or, less generally, from a splendid Rickard equivalence). In this pursuit we are led to a new type of block equivalence that we call a \textit{strong isotypy} --- see Definition \ref{defn:strongisotypy}. The terminology has been chosen because a strong isotypy can be viewed as an extension of Brou\'{e}'s original conception of isotypy, which appeared in \cite{Broue_1990}. Let $G$ and $H$ be finite groups, $A$ a block of $\OO G$ and $B$ a block of $\OO H$. Recall that an isotypy between $A$ and $B$ is defined relative to a maximal $A$-Brauer pair $(D,e)$, a maximal $B$-Brauer pair $(E,f)$, and an isomorphism of fusion systems $\phi:\mathcal{F}_{(E,f)}(B)\isoto\mathcal{F}_{(D,e)}(A)$, and consists of a family of ``compatible'' perfect isometries $\mu_Q$, $Q\subgp E$, between the centralizers of the Brauer pairs corresponding under $\phi$ (see \cite[Definition 15.3]{Boltje_2020}). The requirement that the perfect isometries $\mu_Q$ be ``compatible'' with one another has two equivalent formulations: one in terms of commutative diagrams involving the generalized decomposition maps and another in terms of character value relations \cite[Proposition 4.7]{Broue_1990}. A strong isotypy is defined relative to the same data as an isotypy and again consists of a family of ``compatible'' virtual characters $\chi_Q$ indexed by the subgroups $Q\subgp E$; but now the characters $\chi_Q$ are defined on the \textit{normalizers} of the relevant Brauer pairs rather than on the centralizers. The ``compatiblity'' requirement placed on the characters $\chi_Q$ in the definition of strong isotypy is expressed in terms of character value relations. A strong isotypy induces an isotypy \`{a} la Brou\'{e} simply by restricting to centralizers --- see Theorem \ref{thm:strongisotimpliesisot}. We also find commutative diagrams associated to a strong isotypy that extend the diagrams appearing in the compatibility condition of the definition of ``isotypy.'' These diagrams again involve perfect isometries and generalized decomposition maps. Roughly speaking, however, our diagrams exist at the ``normalizer level'' rather than the ``centralizer level,'' as in the case of an isotypy. See Theorem \ref{thm:mydiagrams}. 
	
	We are reassured that our notion of strong isotypy is the correct extension of an isotypy in the sense of Brou\'{e} by the results of Section \ref{sec:strongisoandpperm}. In this section we show that a $p$-permutation equivalence between blocks $A$ and $B$ induces a strong isotypy between $A$ and $B$ and conversely that every strong isotypy ``comes from'' a $p$-permutation equivalence. In fact, there is a bijection between the set of $p$-permutation equivalences between $A$ and $B$ and the set of strong isotypies between $A$ and $B$ --- this is stated precisely and shown in Theorem \ref{thm:ppermequivtosi}. If an isotypy is the shadow of a $p$-permutation equivalence, as is expected in the situation of the abelian defect group conjecture, then it is really only a partial shadow: a strong isotypy fills in completely the missing compatibility criteria for the characters associated to a $p$-permutation equivalence.
	
	In \cite[Theorem A]{Boltje_2022} Boltje and Carman introduced a ring of so-called ``coherent character tuples'' that is isomorphic to the trivial source ring $T_\OO(G)$ of a finite group $G$ (we recall their result in Theorem \ref{thm:coherence1}). In this situation a coherent character tuple is a tuple of virtual characters $(\chi_P)_{P\in S_p(G)}$ indexed by the $p$-subgroups of $G$ such that $\chi_P\in R_\K(N_G(P)/P)$ for each $P\in S_p(G)$. In Section \ref{sec:coherenceconditions} we provide refinements of Boltje and Carman's work that describe the trivial source group $T_\OO(B)$, where $B$ is a block of $\OO G$, in terms of coherent character tuples. The main difference here is that the coherent character tuples that describe $T_\OO(B)$ can be indexed by the set of all $B$-Brauer pairs (see Theorem \ref{thm:coherence3}) or by the subgroups of a fixed defect group of $B$ (see Theorem \ref{thm:coherence4}). These theorems provide the avenue to the definition of strong isotypy and to the results described in the previous paragraphs.
	
	Throughout this note $p$ will denote a prime number and any $p$-modular system $(\K,\OO,F)$ will be assumed ``large enough'' for the finite groups under consideration. We write $\overline{\cdot}:\OO\onto F$ for the canonical surjection. 
	
	If $G$ is a finite group then $S_p(G)$ denotes the set of $p$-subgroups of $G$. We write $G_{p'}$ for the set of $g\in G$ of order not divisible by $p$. The elements of $G_{p'}$ are also called $p'$-\textit{elements} of $G$. We write $g\sim_G h$ if $g$ and $h$ are $G$-conjugate elements of $G$. 
	
	If $g,h\in G$ then we write $c_g(h)={}^gh=ghg^{-1}$.
	
	Recall that if $g\in G$ then there exists a $p$-element $g_p\in G$ and a $p'$-element $g_{p'}\in G$ for which $g=g_pg_{p'}=g_{p'}g_p$. Moreover, $g_p,g_{p'}\in\gp{g}$ and the pair $(g_p,g_{p'})$ is unique. The element $g_p$ is called the $p$-\textit{part of }$g$ and the element $g_{p'}$ is the $p'$-\textit{part of }$g$.
	
	If $G$ and $H$ are finite groups then $p_1:G\times H\onto G$ and $p_2:G\times H\onto H$ denote the canonical projections. If $X\subgp G\times H$ then we set
	\begin{equation*}
		k_1(X):=\set{g\in G|(g,1)\in X}\qquad\text{and}\qquad k_2(X):=\set{h\in H|(1,h)\in X}.
	\end{equation*}
	One has $k_i(X)\nor p_i(X)$ for $i=1,2$ and the projections $p_i$ induce isomorphisms $X/(k_1(X)\times k_2(X))\isoto p_i(X)/k_i(X)$.

	If $Q\subgp H$, $P\subgp G$, and $\phi:Q\isoto P$ is an isomorphism, set
	\begin{equation*}
		\Delta(P,\phi,Q):=\set{(\phi(y),y)|y\in Q}\subgp G\times H.
	\end{equation*}
	Subgroups of $G\times H$ of the form $\Delta(P,\phi,Q)$ are called \textit{twisted diagonal} subgroups. Write $S_p^\Delta(G\times H)$ for the collection of twisted diagonal $p$-subgroups of $G\times H$. Note that $S_p^\Delta(G\times H)$ is closed under $G\times H$-conjugation and closed under taking subgroups. In fact, if $\Delta(P,\phi,Q)$ is a twisted diagonal subgroup of $G\times H$ then
	\begin{equation*}
		{}^{(g,h)}\Delta(P,\phi,Q)=\Delta({}^gP,c_g\phi c_h^{-1},{}^hQ)
	\end{equation*} 
	for any $(g,h)\in G\times H$.
	
	All modules are assumed finitely generated unless stated otherwise. If $R$ is a commutative ring and $G$ and $H$ are finite groups we will always assume that the induced $R$-module structures on an $(RG,RH)$-bimodule coincide. In other words, if $M$ is an $(RG,RH)$-bimodule we will assume that $rm=mr$ for all $r\in R$ and $m\in M$. Any $(RG,RH)$-bimodule $M$ may be viewed as a left $R[G\times H]$-module (and vice versa) by defining $(g,h)m=gmh^{-1}$ for all $g\in G$, $h\in H$, and $m\in M$. One obtains an isomorphism of categories $\bimod{RG}{RH}\iso\lmod{R[G\times H]}$ in this way. 
	
	If $R$ is a commutative ring and $G$ is a finite group then $(\cdot)^\ast:RG\to RG$ will denote the antipode of $RG$, defined by $g^\ast=g^{-1}$ for all $g\in G$. The antipode is an $R$-module isomorphism and satisfies $(\alpha\beta)^\ast=\beta^\ast\alpha^\ast$ for all $\alpha,\beta\in RG$.
	
	A construction that will be used several times in Section \ref{sec:coherenceconditions} is the following: let $G$ act on a nonempty set $X$. For each $x\in X$, suppose that $A_x$ is an abelian group. Suppose also that for each $g\in G$ and $x\in X$ we have a group isomorphism $\varphi_{g,x}:A_x\isoto A_{{}^gx}$ such that
	\begin{itemize}
		\item[(1)] $\varphi_{1,x}=\id_{A_x}$ for all $x\in X$; and
		\item[(2)] $\varphi_{h,{}^gx}\circ\varphi_{g,x}=\varphi_{hg,x}$ for all $g,h\in G$ and all $x\in X$.
	\end{itemize}
	The product $A=\prod_{x\in X}A_x$ can then be given a $\Z G$-module structure by defining
	\begin{equation*}
		{}^g(a_x):=(\varphi_{g,{}^{g^{-1}}x}(a_{{}^{g^{-1}}x}))_{x\in X}
	\end{equation*}
	for all $g\in G$ and $(a_x)_{x\in X}\in A$. In other words, if $g\in G$ and $(a_x)\in A$ then the $x$-entry of ${}^g(a_x)$ is $\varphi_{g,{}^{g^{-1}}x}(a_{{}^{g^{-1}}x})$. Note that the subgroup of $G$-fixed points $A^G$ consists of all tuples $(a_x)\in A$ such that $\varphi_{g,x}(a_x)=a_{{}^gx}$ for all $x\in X$ and $g\in G$.

	%check: in sets of Brauer pairs, Brauer elements, Fusion systems, etc. Always write group first and then block: for example, \BP_R(G,B)
	
	%define trace?
	
	%note on identifying $R[G\times H]$ with $RG\tensor_R RH$.
	
	\section{Brauer pairs and fusion systems}
	
	Throughout this section $p$ denotes a prime number and $(\K,\OO,F)$ is a $p$-modular system large enough for the finite groups under consideration. Many of the results of this section hold over both $\OO$ and $F$, so for brevity we let $R\in\set{\OO,F}$ and work over $R$. If $a\in\OO$ we write $\overline{a}$ for the image of $a$ under the canonical projection $\OO\onto F$. If $G$ is a finite group then the canonical projection $\OO\onto F$ extends to an $\OO$-algebra homomorphism $\OO G\onto FG$ and we write $\overline{\alpha}$ for the image of $\alpha\in\OO G$ under this map. If $\alpha\in FG$ we set $\overline{\alpha}=\alpha$.
	
	Let $G$ be a finite group. We write $\Bl(RG)$ for the set of block algebras of $RG$ and $\bli(RG)$ for the set of block idempotents of $RG$. If $B\in\Bl(RG)$ then $e_B\in\bli(RG)$ denotes the identity of $B$. Recall that the coefficient reduction map $\overline{\cdot}:\OO G\onto FG$ induces a bijection between the blocks of $\OO G$ and the blocks of $FG$. 
	
	If $P$ is a $p$-subgroup of $G$ then there is a surjective homomorphism of $RN_G(P)$-algebras
	\begin{equation*}
		\br_P^G:(RG)^P\onto FC_G(P),\qquad \sum_{g\in G}a_gg\mapsto\sum_{g\in C_G(P)}\overline{a_g}g
	\end{equation*}
	called the \textit{Brauer homomorphism}. When the overgroup $G$ is contextually clear we may write $\br_P$ in place of $\br_P^G$. Note that if $\alpha\in(RG)^P$ then $\alpha^\ast\in(RG)^P$ and $\br_P(\alpha^\ast)=\br_P(\alpha)^\ast$. Note also that if $H\subgp G$ and $P$ is a $p$-subgroup of $H$ then ${}^g\br_P^H(\alpha)=\br_{{}^gP}^{{}^gH}({}^g\alpha)$ for any $g\in G$ and $\alpha\in(RH)^P$.
	
	\subsection{Brauer pairs}
	
	An $RG$-\textit{Brauer pair} is an ordered pair $(P,e)$ where $P$ is a $p$-subgroup of $G$ and $e$ is a block idempotent of $RC_G(P)$. The set of $RG$-Brauer pairs is denoted $\BP_R(G)$. The group $G$ acts by conjugation on $\BP_R(G)$ and the stabilizer of $(P,e)\in\BP_R(G)$ is denoted $N_G(P,e)$. Note that for any $RG$-Brauer pair $(P,e)$ one has $PC_G(P)\subgp N_G(P,e)\subgp N_G(P)$. If $(Q,f),(P,e)\in\BP_R(G)$ write $(Q,f)\nor(P,e)$ if $Q\subgp P\subgp N_G(Q,f)$ and $\br_P(f)\overline{e}=\overline{e}$. The transitive closure of this relation makes $\BP_R(G)$ into a partially ordered set, and the action of $G$ by conjugation on $\BP_R(G)$ respects this order --- in other words, $\BP_R(G)$ is a $G$-poset.
	
	Both of the maps $\BP_R(G)\to\BP_R(G)$, $(P,e)\mapsto (P,e^\ast)$, and $\BP_\OO(G)\to\BP_F(G)$, $(P,e)\mapsto (P,\overline{e})$ are $G$-poset isomorphisms.
	
	An important fact about $RG$-Brauer pairs which will be used repeatedly in the sequel is that if $(P,e)\in\BP_R(G)$ and $Q\subgp P$ then there exists a unique block $f\in\bli(RC_G(Q))$ such that $(Q,f)\leq(P,e)$. See \cite[Theorem 2.10(a)]{Aschbacher_2011} for a proof of this fact in a more general setting.
	
	Let $H\subgp G$. If $(P,e)\in\BP_R(H)$ and $g\in G$ then we set ${}^g(P,e):=({}^gP,{}^ge)\in\BP_R({}^gH)$. Notice that ${}^1(P,e)=(P,e)$ and ${}^g({}^{g'}(P,e))={}^{gg'}(P,e)$ for any $g,g'\in G$ and $(P,e)\in\BP_R(H)$. If $(Q,f),(P,e)\in\BP_R(H)$ are such that $(Q,f)\leq(P,e)$ then ${}^g(Q,f)\leq{}^g(P,e)$ for any $g\in G$. It follows that for each $g\in G$ we have an isomorphism of posets
	\begin{align*}
		{}^g(\cdot):\BP_R(H)&\isoto\BP_R({}^gH)\\
			(P,e)&\mapsto{}^g(P,e).
	\end{align*}
	
	Let $B\in\Bl(RG)$. An $RG$-Brauer pair $(P,e)$ \textit{belongs to} $B$ if $\br_P(e_B)\overline{e}=\overline{e}$. If $(P,e)$ belongs to $B$ we also say that $(P,e)$ is a $B$-\textit{Brauer pair}. In the sequel the set of $RG$-Brauer pairs that belong to $B$ will be denoted by $\BP_R(G,B)$, $\BP_R(G,e_B)$, or simply by $\BP_R(B)$. Recall that $\BP_R(B)$ is a $G$-subposet of $\BP_R(G)$ and that if $(Q,f),(P,e)\in\BP_R(G)$ are such that $(Q,f)\leq(P,e)$ then $(Q,f)$ is a $B$-Brauer pair if and only if $(P,e)$ is a $B$-Brauer pair. If $(P,e)\in\BP_R(G,B)$ then $(P,e^\ast)\in\BP_R(G,B^\ast)$. The $G$-poset isomorphism $\BP_\OO(G)\isoto\BP_F(G)$ described above restricts to a $G$-poset isomorphism $\BP_\OO(B)\isoto\BP_F(\overline{B})$ for any block $B$ of $\OO G$. If $H\subgp G$, $B\in\Bl(RH)$, and $g\in G$ then the map ${}^g(\cdot):\BP_R(H)\isoto\BP_R({}^gH)$ defined above restricts to a poset isomorphism $\BP_R(H,B)\isoto\BP_R({}^gH,{}^gB)$.
	
	A \textit{Brauer element of }$RG$ is an ordered pair $(u,e)$ where $u$ is a $p$-element of $G$ and $e$ is a block idempotent of $RC_G(u)$. Write $\BE_R(G)$ for the set of Brauer elements of $RG$. Notice that if $(u,e)\in\BE_R(G)$ then $(\gp{u},e)\in\BP_R(G)$. The group $G$ acts on $\BE_R(G)$ by conjugation and the map $\BE_R(G)\to\BP_R(G)$ which sends $(u,e)$ to $(\gp{u},e)$ is $G$-equivariant.
	
	Let $B\in\Bl(RG)$ and let $(u,e)\in\BE_R(G)$. Say $(u,e)$ \textit{belongs to }$B$ or is a $B$-\textit{Brauer element} if $\br_{\gp{u}}(e_B)\overline{e}=\overline{e}$. In what follows the set of Brauer elements of $RG$ that belong to $B$ will be denoted by $\BE_R(G,B)$, by $\BE_R(G,e_B)$, or by $\BE_R(B)$. Notice that if $(u,e)$ is a Brauer element of $RG$ then $(u,e)$ belongs to $B$ if and only if $(\gp{u},e)$ belongs to $B$. The set of $B$-Brauer elements $\BE_R(B)$ is stable under conjugation by $G$. Note also that the subsets $\BE_R(B)$, where $B$ runs through the blocks of $RG$, form a partition of $\BE_R(G)$. 
	
	%reference for lemma below? restate in terms of block covering?
	
	\begin{lemma}\label{lem:blockcent}
		Suppose that $P$ is a normal $p$-subgroup of $G$. Then every central idempotent of $RG$ belongs to $RC_G(P)$. More generally, if $P,Q\in S_p(G)$ and $P\nor G$ then every central idempotent of $RC_G(Q)$ belongs to $RC_G(PQ)$.
	\end{lemma}
	
	Let $N$ be a normal subgroup of $G$. Then $G$ permutes the blocks of $RN$ by conjugation. If $c\in\bli(RN)$ then $\tr_{\Stab_G(c)}^G(c)$ is a (nonzero) central idempotent of $RG$ contained in $RN$. Recall that a block idempotent $b$ of $RG$ is said to \textit{cover} $c$ if $b\tr_{\Stab_G(c)}^G(c)\neq 0$, or equivalently if $bc\neq 0$. We refer to \cite[IV.6]{Aschbacher_2011} for more information about block covering. If $b\in\bli(RG)$ then there exists a unique $G$-orbit of blocks of $RN$ that are covered by $b$, and if $c\in\bli(RN)$ then there exists a block $b\in\bli(RG)$ covering $c$. Thus we have a surjective map
	\begin{align*}
		\bli(RG)	&\onto G\backslash\bli(RN)\\
		b&\mapsto\set{c\in\bli(RN)|bc\neq 0}.
	\end{align*}
	Note that if the map above is also injective, then
	\begin{equation*}
		\bli(RG)=\set{\tr_{\Stab_G(c)}^G(c)|c\in\bli(RN)}.
	\end{equation*}
	
	%restate lemma below in terms of injectivity of map?
	
	\begin{lemma}\label{lem:blockcent2}
		Assume that $P,Q\in S_p(G)$ and $P\nor G$. Then
		\begin{equation*}
			\bli(RC_G(Q))=\set{\tr_{\Stab_{C_G(Q)}(e)}^{C_G(Q)}(e)|e\in\bli(RC_G(PQ))}.
		\end{equation*}
		In particular, if $P$ is a normal $p$-subgroup of $G$ then
		\begin{equation*}
			\bli(RG)=\set{\tr_{N_G(P,e)}^G(e)|e\in\bli(RC_G(P))}.
		\end{equation*}
	\end{lemma}
	
	\begin{proof}
		We have $C_G(PQ)=C_G(P)\cap C_G(Q)\nor C_G(Q)$ since $P\nor G$. Now if $b\in\bli(RC_G(Q))$ then $b$ is a central idempotent of $RC_G(PQ)$ by Lemma \ref{lem:blockcent}. It follows that $b$ is equal to the sum of the blocks of $RC_G(PQ)$ covered by $b$. In particular, the map
		\begin{align*}
			\bli(RC_G(Q))	&\onto C_G(Q)\backslash\bli(RC_G(PQ))\\
			b				&\mapsto \set{c\in\bli(RC_G(PQ))|bc\neq 0}
		\end{align*}
		is injective. The result follows.
	\end{proof}

	\begin{lemma}
		Let $(P,e)\in\BP_R(G)$. Then $e$ is a block idempotent of $RN_G(P,e)$.
	\end{lemma}

	\begin{proof}
		Set $I=N_G(P,e)$. Then $P\nor I$ and $C_I(P)=C_G(P)$. So by Lemma \ref{lem:blockcent2}, $\tr_{N_I(P,e)}^I(e)=e$ is a block idempotent of $RI$.
	\end{proof}
	
	%compare this lemma to later lemma about Brauer pairs in normalizer subsystem.
	
	%add something about unique block covering?
	
	\begin{lemma}\label{lem:brauereltnormalizer}
		Let $B\in\Bl(RG)$. If $(P,e)\in\BP_R(B)$ write $I_{(P,e)}=N_G(P,e)$. Fix a $B$-Brauer pair $(P,e)\in\BP_R(B)$ and let $(Q,\epsilon)\in\BP_R(I_{(P,e)},e)$.
		\begin{itemize}
			\item[(a)] There exists a block $f\in\bli(RC_G(PQ))$ such that $\epsilon f\neq 0$. Any two $f,f'\in\bli(RC_G(PQ))$ with $\epsilon f\neq 0\neq \epsilon f'$ are $C_{I_{(P,e)}}(Q)$-conjugate.
			\item[(b)] Let $f\in\bli(RC_G(PQ))$ such that $\epsilon f\neq 0$. Then $(PQ,f)\in\BP_R(B)$, $(P,e)\nor (PQ,f)$, and
			\begin{equation}\label{eqn:epsilon}
				\epsilon=\tr_{\Stab_{C_{I_{(P,e)}}(Q)}(f)}^{C_{I_{(P,e)}}(Q)}(f).
			\end{equation}
			Moreover, we have
			\begin{equation}\label{eqn:subgpequalities}
				\Stab_{C_{I_{(P,e)}}(Q)}(f)=C_{I_{(P,e)}\cap I_{(PQ,f)}}(Q)=C_{I_{(PQ,f)}}(Q)\cap N_G(P).
			\end{equation}
		\end{itemize}
	\end{lemma}
	
	\begin{proof}
		Since $C_G(P)\nor I_{(P,e)}$ we have $C_G(PQ)\nor C_{I_{(P,e)}}(Q)$. Now $\epsilon$ is a block idempotent of $RC_{I_{(P,e)}}(Q)$, so there exists a unique $C_{I_{(P,e)}}(Q)$-orbit of blocks of $RC_G(PQ)$ covered by $\epsilon$. Part (a) follows. 
		
		Next let $f\in\bli(RC_G(PQ))$ such that $\epsilon f\neq 0$. Since $P,Q\in S_p(I_{(P,e)})$, $P\nor I_{(P,e)}$, and $C_{I_{(P,e)}}(PQ)=C_G(PQ)$, Lemma \ref{lem:blockcent2} gives us the equality in (\ref{eqn:epsilon}). Now certainly $(PQ,f)\in\BP_R(G)$. If we show that $(P,e)\nor(PQ,f)$ then since $(P,e)$ is a $B$-Brauer pair it will follow that $(PQ,f)$ is a $B$-Brauer pair as well. Clearly $P\subgp PQ\subgp I_{(P,e)}$ so to prove that $(P,e)\nor(PQ,f)$ it remains to see that $\br_{PQ}(e)\overline{f}=\overline{f}$, or equivalently that $\br_{PQ}(e)\overline{f}\neq 0$. Suppose, by way of contradiction, that $\br_{PQ}(e)\overline{f}=0$. One may readily compute that $\br_{PQ}(e)=\br_Q^{I_{(P,e)}}(e)$, so then $\br_Q^{I_{(P,e)}}(e)\overline{f}=0$. It follows that $\br_Q^{I_{(P,e)}}(e)\overline{{}^xf}=0$ for any $x\in C_{I_{(P,e)}}(Q)$, hence $\br_Q^{I_{(P,e)}}(e)\overline{\epsilon}=0$. But this contradicts the assumption that $(Q,\epsilon)\in\BP_R(I_{(P,e)},e)$. Thus we find that $\br_{PQ}(e)\overline{f}\neq 0$ and $(P,e)\nor(PQ,f)$.
		
		To complete the proof of part (b) it remains to verify the equalities in (\ref{eqn:subgpequalities}). Let $g\in C_{I_{(P,e)}}(Q)$ such that ${}^gf=f$. Since $C_{I_{(P,e)}}(Q)\subgp N_G(PQ)$ it follows that $g\in I_{(PQ,f)}$, hence $g\in C_{I_{(P,e)}\cap I_{(PQ,f)}}(Q)$. So $\Stab_{C_{I_{(P,e)}}(Q)}(f)\subgp C_{I_{(P,e)}\cap I_{(PQ,f)}}(Q)$. Because $I_{(P,e)}\subgp N_G(P)$ we have that $C_{I_{(P,e)}\cap I_{(PQ,f)}}(Q)\subgp C_{I_{(PQ,f)}}(Q)\cap N_G(P)$. Finally, let $h\in C_{I_{(PQ,f)}}(Q)\cap N_G(P)$. Conjugating the containment $(P,e)\nor (PQ,f)$ by $h$ yields $(P,{}^he)\nor(PQ,f)$, so ${}^he=e$ by uniqueness. Therefore $h\in I_{(P,e)}$. Since $h$ centralizes $Q$ and fixes $f$ we have $h\in\Stab_{C_{I_{(P,e)}}(Q)}(f)$. The element $h$ was arbitrary, so we conclude that $C_{I_{(PQ,f)}}(Q)\cap N_G(P)\subgp \Stab_{C_{I_{(P,e)}}(Q)}(f)$. The proof is complete.
	\end{proof}

	\begin{lemma}\label{lem:bijection1}
		Let $B\in\Bl(RG)$ and let $(P,e)\in\BP_R(B)$. Write $I_{(P,e)}=N_G(P,e)$ and let $Q$ be a $p$-subgroup of $I_{(P,e)}$. Set
		\begin{align*}
			\mathcal{E}=\{(\epsilon,f)|\epsilon&\in\bli(RC_{I_{(P,e)}}(Q)),f\in\bli(RC_G(PQ))\\
			&\text{ such that }(Q,\epsilon)\in\BP_R(I_{(P,e)},e)\text{ and }\epsilon f\neq 0\}
		\end{align*}
		and set
		\begin{equation*}
			\mathcal{F}=\set{f\in\bli(RC_G(PQ))|(P,e)\nor(PQ,f)}.
		\end{equation*}
		Then the map $\mathcal{E}\to\mathcal{F}$, $(\epsilon,f)\mapsto f$, is a well-defined bijection.
	\end{lemma}
	
	\begin{proof}
		The map is well-defined and injective by part (b) of Lemma \ref{lem:brauereltnormalizer}. Let $f\in\mathcal{F}$. Then $\epsilon:=\tr_{\Stab_{C_{I_{(P,e)}}(Q)}(f)}^{C_{I_{(P,e)}}(Q)}(f)$ is a block idempotent of $RC_{I_{(P,e)}}(Q)$ by Lemma \ref{lem:blockcent2}. Now $(P,e)\nor(PQ,f)$ so for any $x\in C_{I_{(P,e)}}(Q)$ we have $(P,e)\nor (PQ,{}^xf)$, hence $\br_{PQ}(e)\overline{{}^xf}=\overline{{}^xf}$. Since $\br_{PQ}(e)=\br_Q^{I_{(P,e)}}(e)$ it follows that $\br_Q^{I_{(P,e)}}(e)\overline{\epsilon}=\overline{\epsilon}$. In other words, $(Q,\epsilon)\in\BP_R(I_{(P,e)},e)$. The pair $(\epsilon,f)$ is an element of $\mathcal{E}$ that maps to $f$, so the map $\mathcal{E}\to\mathcal{F}$ is surjective.
	\end{proof}
	
	%start
	
	Now let $G$ and $H$ be finite groups and assume that the $p$-modular system $(\K,\OO,F)$ is large enough for $G\times H$. It is not hard to see that if $P\subgp G\times H$ then
	\begin{equation*}
		C_{G\times H}(P)=C_G(p_1(P))\times C_H(p_2(P))
	\end{equation*}
	where $p_1:G\times H\onto G$ and $p_2:G\times H\onto H$ denote the canonical projections. Thus, after identifying $RC_{G\times H}(P)$ with $R[C_G(p_1(P))]\tensor_R R[C_H(p_2(P))]$, every block idempotent of $RC_{G\times H}(P)$ is of the form $e\tensor f$ for uniquely determined block idempotents $e\in\bli(R C_G(p_1(P)))$ and $f\in\bli(RC_H(p_2(P)))$. In particular, every $R[G\times H]$-Brauer pair is of the form $(P,e\tensor f)$ where $P\in S_p(G\times H)$, $e\in\bli(R C_G(p_1(P)))$, and $f\in\bli(RC_H(p_2(P)))$. Note that if $(P,e\tensor f)\in\BP_R(G\times H)$ then $(p_1(P),e)\in\BP_R(G)$ and $(p_2(P),f)\in\BP_R(H)$.
	
	\begin{lemma}\label{lem:GHBrauerpaircontainment}
		Let $G$ and $H$ be finite groups.
		\begin{itemize}
			\item[(a)] Let $(Q,c\tensor d),(P,e\tensor f)\in\BP_R(G\times H)$. If $(Q,c\tensor d)\nor(P,e\tensor f)$ then $(p_1(Q),c)\nor(p_1(P),e)$ in $\BP_R(G)$ and $(p_2(Q),d)\nor(p_2(P),f)$ in $\BP_R(H)$.
			\item[(b)] Let $Q,P\in S_p(G\times H)$ and assume $Q\nor P$. Let $(p_1(Q),c),(p_1(P),e)\in\BP_R(G)$ and $(p_2(Q),d),(p_2(P),f)\in\BP_R(H)$ be such that $(p_1(Q),c)\leq(p_1(P),e)$ and $(p_2(Q),d)\leq(p_2(P),f)$. Then $(Q,c\tensor d)\nor (P,e\tensor f)$ in $\BP_R(G\times H)$.
			\item[(c)] Let $(Q,c\tensor d),(P,e\tensor f)\in\BP_R(G\times H)$. Then $(Q,c\tensor d)\leq(P,e\tensor f)$ if and only if $Q\subgp P$, $(p_1(Q),c)\leq(p_1(P),e)$ in $\BP_R(G)$ and $(p_2(Q),d)\leq(p_2(P),f)$ in $\BP_R(H)$.
		\end{itemize}
	\end{lemma}
	
	\begin{proof}
		(a) Assume that $(Q,c\tensor d)\nor(P,e\tensor f)$ is a normal containment of $R[G\times H]$-Brauer pairs. Then by definition $Q\subgp P\subgp N_{G\times H}(Q,c\tensor d)$ and $\br_P^{G\times H}(c\tensor d)\cdot\overline{e\tensor f}=\overline{e\tensor f}$. Since $Q\nor P$ we have $p_i(Q)\nor p_i(P)$ for $i=1,2$. Let $x\in p_1(P)$. Then there exists an element $y\in H$ such that $(x,y)\in P$. Since $P\subgp N_{G\times H}(Q,c\tensor d)$ we have $y\in N_H(p_2(Q))$. Now $x\in N_G(p_1(Q))$ and $c\in\bli(RC_G(p_1(Q)))$ so ${}^xc\in\bli(RC_G(p_1(Q)))$. Likewise ${}^yd\in\bli(RC_H(p_2(Q)))$. Since $c\tensor d={}^{(x,y)}(c\tensor d)={}^xc\tensor{}^yd$ it follows that ${}^xc=c$. Thus we have shown that $p_1(P)\subgp N_G(p_1(Q),c)$. A similar argument shows that $p_2(P)\subgp N_H(p_2(Q),d)$. Since $\br_P^{G\times H}(c\tensor d)=\br_{p_1(P)}^G(c)\tensor\br_{p_2(P)}^H(d)$ in $FC_{G\times H}(P)=(FC_G(p_1(P)))\tensor_F (FC_H(p_2(P)))$ the equality $\br_P^{G\times H}(c\tensor d)\cdot\overline{e\tensor f}=\overline{e\tensor f}$ gives that
		\begin{equation*}
			\br_{p_1(P)}^G(c)\overline{e}\tensor\br_{p_2(P)}^H(d)\overline{f}=\overline{e}\tensor\overline{f}.
		\end{equation*}
		So $\br_{p_1(P)}^G(c)\overline{e}=\overline{e}$ and $\br_{p_2(P)}^H(d)\overline{f}=\overline{f}$. Thus we find that $(p_1(Q),c)\nor(p_1(P),e)$ in $\BP_R(G)$ and $(p_2(Q),d)\nor(p_2(P),f)$ in $\BP_R(H)$. 
		
		(b) Let $Q\nor P$ be $p$-subgroups of $G\times H$ and let $(p_1(Q),c),(p_1(P),e)\in\BP_R(G)$, $(p_2(Q),d),(p_2(P),f)\in\BP_R(H)$ be such that $(p_1(Q),c)\leq(p_1(P),e)$ and $(p_2(Q),d)\leq(p_2(P),f)$. Since $p_i(Q)\nor p_i(P)$ for $i=1,2$, \cite[Proposition 4.2(b)]{Boltje_2020} implies that $(p_1(Q),c)\nor(p_1(P),e)$ and $(p_2(Q),d)\nor(p_2(P),f)$. In particular, we have $p_1(P)\subgp N_G(p_1(Q),c)$, $p_2(P)\subgp N_H(p_2(Q),d)$, $\br_{p_1(P)}^G(c)\overline{e}=\overline{e}$, and $\br_{p_2(P)}^H(d)\overline{f}=\overline{f}$. Let $(x,y)\in P$. Then $x\in p_1(P)$ so ${}^xc=c$ and $y\in p_2(P)$ so ${}^yd=d$. It follows that ${}^{(x,y)}(c\tensor d)={}^xc\tensor {}^yd=c\tensor d$ and hence $P\subgp N_{G\times H}(Q,c\tensor d)$. Now $\br_P^{G\times H}(c\tensor d)=\br_{p_1(P)}^G(c)\tensor\br_{p_2(P)}^H(d)$, so
		\begin{align*}
			\br_P^{G\times H}(c\tensor d)\cdot(\overline{e\tensor f})	&=(\br_{p_1(P)}^G(c)\tensor\br_{p_2(P)}^H(d))\cdot(\overline{e}\tensor\overline{f})\\
			&=\br_{p_1(P)}^G(c)\overline{e}\tensor\br_{p_2(P)}^H(d)\overline{f}\\
			&=\overline{e}\tensor\overline{f}\\
			&=\overline{e\tensor f}.
		\end{align*}
		We conclude that $(Q,c\tensor d)\nor(P,e\tensor f)$, as desired.
		
		(c) Let $(Q,c\tensor d),(P,e\tensor f)\in\BP_R(G\times H)$. If $(Q,c\tensor d)\leq(P,e\tensor f)$ then $Q\subgp P$ and part (a) implies that $(p_1(Q),c)\leq(p_1(P),e)$ in $\BP_R(G)$ and $(p_2(Q),d)\leq(p_2(P),f)$ in $\BP_R(H)$. Suppose conversely that $Q\subgp P$, $(p_1(Q),c)\leq(p_1(P),e)$, and $(p_2(Q),d)\leq(p_2(P),f)$. Let
		\begin{equation*}
			Q=Q_0\nor Q_1\nor\cdots\nor Q_n=P
		\end{equation*}
		be a subnormal chain of $p$-subgroups in $G\times H$. Then 
		\begin{equation*}
			p_i(Q)=p_i(Q_0)\nor p_i(Q_1)\nor\cdots\nor p_i(Q_n)=p_i(P)
		\end{equation*}
		is a subnormal chain of $p$-subgroups in $G$ (if $i=1$) or $H$ (if $i=2$). By \cite[Proposition 4.2(a)]{Boltje_2020} there exist Brauer pairs $(p_1(Q_i),c_i)\in\BP_R(G)$ and $(p_2(Q_i),d_i)\in\BP_R(H)$, $i=0,\ldots,n$, such that
		\begin{equation*}
			(p_1(Q),c)=(p_1(Q_0),c_0)\nor(p_1(Q_1),c_1)\nor\cdots\nor(p_1(Q_n),c_n)=(p_1(P),e)
		\end{equation*}
		in $\BP_R(G)$ and
		\begin{equation*}
			(p_2(Q),d)=(p_2(Q_0),d_0)\nor(p_2(Q_1),d_1)\nor\cdots\nor(p_2(Q_n),d_n)=(p_2(P),f)
		\end{equation*}
		in $\BP_R(H)$. Part (b) then implies that $(Q_i,c_i\tensor d_i)\nor(Q_{i+1},c_{i+1}\tensor d_{i+1})$ for each $i=0,\ldots,n-1$. We conclude that $(Q,c\tensor d)\leq (P,e\tensor f)$.
	\end{proof}
	
	\begin{lemma}\label{lem:pi1pi2}
		Let $G$ and $H$ be finite groups. The maps $\pi_1:\BP_R(G\times H)\to\BP_R(G)$ and $\pi_2:\BP_R(G\times H)\to\BP_R(H)$ defined by $\pi_1(P,e\tensor f)=(p_1(P),e)$ and $\pi_2(P,e\tensor f)=(p_2(P),f)$ are surjective morphisms of posets. Furthermore, if $(g,h)\in G\times H$ and $(P,e\tensor f)\in\BP_R(G\times H)$ then one has
		\begin{equation*}
			\pi_1({}^{(g,h)}(P,e\tensor f))={}^g\pi_1(P,e\tensor f)
		\end{equation*}
		and
		\begin{equation*}
			\pi_2({}^{(g,h)}(P,e\tensor f))={}^h\pi_2(P,e\tensor f).
		\end{equation*}
	\end{lemma}
	
	\begin{proof}
		The maps $\pi_1$ and $\pi_2$ are morphisms of posets by part (c) of Lemma \ref{lem:GHBrauerpaircontainment}. Let $(P,e)\in\BP_R(G)$. If $f$ is any block idempotent of $RH$ then $(P\times\set{1},e\tensor f)\in\BP_R(G\times H)$ and $\pi_1(P\times\set{1},e\tensor f)=(P,e)$. So $\pi_1$ is surjective. In a similar way one can show that $\pi_2$ is surjective. If $(g,h)\in G\times H$ and $(P,e\tensor f)\in\BP_R(G\times H)$ then
		\begin{align*}
			\pi_1({}^{(g,h)}(P,e\tensor f))	&=\pi_1({}^{(g,h)}P,{}^ge\tensor{}^hf)\\
			&=(p_1({}^{(g,h)}P),{}^ge)\\
			&=({}^gp_1(P),{}^ge)\\
			&={}^g(p_1(P),e)\\
			&={}^g\pi_1(P,e\tensor f).
		\end{align*}
		A similar computation shows that $\pi_2({}^{(g,h)}(P,e\tensor f))={}^h\pi_2(P,e\tensor f)$. The proof is complete.
	\end{proof}
	
	\begin{lemma}\label{lem:brauerpairAB}
		Let $A\in\Bl(RG)$, $B\in\Bl(RH)$, and let $(P,e\tensor f)\in\BP_R(G\times H)$. Then $(P,e\tensor f)\in\BP_R(A\tensor_R B)$ if and only if $(p_1(P),e)\in\BP_R(A)$ and $(p_2(P),f)\in\BP_R(B)$. In particular, if $\pi_1$ and $\pi_2$ are as in Lemma \ref{lem:pi1pi2} then
		\begin{equation*}
			\pi_1(\BP_R(A\tensor_R B))=\BP_R(A)\qquad\text{and}\qquad\pi_2(\BP_R(A\tensor_R B))=\BP_R(B).
		\end{equation*}
	\end{lemma}
	
	\begin{proof}
		Let $e_A$ and $f_B$ denote the identities of $A$ and $B$, respectively. Then $e_A\tensor f_B$ is the identity of $A\tensor_R B$. Note that
		\begin{equation*}
			\br_P^{G\times H}(e_A\tensor f_B)=\br_{p_1(P)}^G(e_A)\tensor\br_{p_2(P)}^H(f_B).
		\end{equation*}
		Now by definition $(P,e\tensor f)\in\BP_R(A\tensor_R B)$ if and only if $\br_P^{G\times H}(e_A\tensor f_B)\cdot(\overline{e\tensor f})=\overline{e\tensor f}$. This equality holds if and only if $\br_{p_1(P)}^G(e_A)\overline{e}\tensor\br_{p_2(P)}^H(f_B)\overline{f}=\overline{e}\tensor\overline{f}$, which in turn holds if and only if $\br_{p_1(P)}^G(e_A)\overline{e}=\overline{e}$ and $\br_{p_2(P)}^H(f_B)\overline{f}=\overline{f}$. Thus we see that $(P,e\tensor f)\in\BP_R(A\tensor_R B)$ if and only if $(p_1(P),e)\in\BP_R(A)$ and $(p_2(P),f)\in\BP_R(B)$.
		
		Now let $\pi_1$ and $\pi_2$ be as in Lemma \ref{lem:pi1pi2}. By what we have just shown, $\pi_1(\BP_R(A\tensor_R B))\subseteq\BP_R(A)$ and $\pi_2(\BP_R(A\tensor_R B))\subseteq\BP_R(B)$. If $(P,e)\in\BP_R(A)$ then the $R[G\times H]$-pair $(P\times\set{1},e\tensor f_B)$ belongs to $A\tensor_R B$ and maps to $(P,e)$ under $\pi_1$. This shows that $\pi_1(\BP_R(A\tensor_R B))=\BP_R(A)$. Similarly, $\pi_2(\BP_R(A\tensor_R B))=\BP_R(B)$.
	\end{proof}
	
	%end
	
	\subsection{Fusion systems}
	
	In this subsection we establish some facts about block fusion systems for later use. We refer the reader to \cite{Aschbacher_2011} for the definitions of any terms not recalled here. 
	
	If $\mathcal{F}$ is a fusion system over a $p$-group $S$ then elements $x,y\in S$ are said to be $\mathcal{F}$-\textit{conjugate} if there exists an $\mathcal{F}$-isomorphism $\varphi:\gp{x}\isoto\gp{y}$ satisfying $\varphi(x)=y$. If $Q\subgp S$ then $\mathcal{N}_{\mathcal{F}}(Q)$ and $\mathcal{C}_{\mathcal{F}}(Q)$ respectively denote the normalizer and centralizer subsystems of $Q$ in $\mathcal{F}$. If $\mathcal{E}$ is another fusion system defined over a $p$-group $R$ then a group isomorphism $\phi:S\isoto R$ induces an isomorphism of fusion systems $\phi:\mathcal{F}\isoto\mathcal{E}$ if
	\begin{equation*}
		\Hom_{\mathcal{E}}(\phi(P),\phi(Q))=\phi\circ\Hom_{\mathcal{F}}(P,Q)\circ\phi^{-1}
	\end{equation*}
	for all $P,Q\subgp S$. 
	
	If $\mathcal{E}$ and $\mathcal{F}$ are fusion systems defined over $p$-groups $R$ and $S$, respectively, then $\mathcal{E}\times\mathcal{F}$ is the fusion system over $R\times S$ generated by the set of all morphisms of the form $(\psi_1,\psi_2)\in\Hom(P_1\times Q_1,P_2\times Q_2)$ with $\psi_1\in\Hom_{\mathcal{E}}(P_1,P_2)$ and $\psi_2\in\Hom_{\mathcal{F}}(Q_1,Q_2)$. If $U,V\subgp R\times S$ then $\Hom_{\mathcal{E}\times\mathcal{F}}(U,V)$ consists of those group homomorphisms $\psi:U\to V$ for which there exist $\psi_1\in\Hom_{\mathcal{E}}(p_1(U),p_1(V))$ and $\psi_2\in\Hom_{\mathcal{F}}(p_2(U),p_2(V))$ such that $(\psi_1,\psi_2)|_U=\psi$.
	
	\begin{lemma}\label{lem:twisteddiagfullynormalized}
		Let $\mathcal{E}$ and $\mathcal{F}$ be fusion systems over $p$-groups $R$ and $S$, respectively. Assume that at least one of $\mathcal{E}$ or $\mathcal{F}$ is saturated. Let $\phi:S\isoto R$ be an isomorphism that induces an isomorphism of fusion systems $\phi:\mathcal{F}\isoto\mathcal{E}$. Let $Q$ be a fully $\mathcal{F}$-normalized subgroup of $S$ and set $P=\phi(Q)$. Then $\Delta(P,\phi,Q)$ is a fully $\mathcal{E}\times\mathcal{F}$-normalized subgroup of $R\times S$.
	\end{lemma}

	\begin{proof}
		Since $\phi:\mathcal{F}\isoto\mathcal{E}$ is an isomorphism, $P$ is a fully $\mathcal{E}$-normalized subgroup of $R$. Notice that we have
		\begin{align*}
			p_1(N_{R\times S}(\Delta(P,\phi,Q)))=N_R(P),\qquad p_2(N_{R\times S}(\Delta(P,\phi,Q)))=N_S(Q),\\
			k_1(N_{R\times S}(\Delta(P,\phi,Q)))=C_R(P),\qquad k_2(N_{R\times S}(\Delta(P,\phi,Q)))=C_S(Q).
		\end{align*}
		Therefore
		\begin{equation*}
			N_R(P)/C_R(P)\iso N_{R\times S}(\Delta(P,\phi,Q))/(C_R(P)\times C_S(Q))\iso N_S(Q)/C_S(Q),
		\end{equation*}
		and in particular
		\begin{equation*}
			|N_{R\times S}(\Delta(P,\phi,Q))|=|N_S(Q)|\cdot |C_R(P)|=|N_R(P)|\cdot |C_S(Q)|.
		\end{equation*}
		Now let $V\subgp R\times S$ and suppose that $\psi:\Delta(P,\phi,Q)\isoto V$ is an $\mathcal{E}\times\mathcal{F}$-isomorphism. Then there exists an $\mathcal{E}$-isomorphism $\psi_1:P\isoto p_1(V)$ and an $\mathcal{F}$-isomorphism $\psi_2:Q\isoto p_2(V)$ such that $(\psi_1,\psi_2)|_{\Delta(P,\phi,Q)}=\psi$. In particular, $V=\Delta(p_1(V),\psi_1\phi\psi_2^{-1},p_2(V))$. It follows that
		\begin{align*}
			p_1(N_{R\times S}(V))\subgp N_R(p_1(V)),\qquad p_2(N_{R\times S}(V))\subgp N_S(p_2(V)),\\
			k_1(N_{R\times S}(V))=C_R(p_1(V)),\qquad k_2(N_{R\times S}(V))=C_S(p_2(V)).
		\end{align*}
		
		Suppose that $\mathcal{E}$ is saturated. Then $P$ is fully $\mathcal{E}$-centralized, so $|C_R(P)|\geq |C_R(p_1(V))|$. We also have $|N_S(Q)|\geq |N_S(p_2(V))|$ since $Q$ is fully $\mathcal{F}$-normalized. It follows that
		\begin{align*}
			|N_{R\times S}(V)|	&=|p_2(N_{R\times S}(V))|\cdot|C_R(p_1(V))|\\
								&\leq |N_S(Q)|\cdot |C_R(P)|=|N_{R\times S}(\Delta(P,\phi,Q))|,
		\end{align*}
		and hence $\Delta(P,\phi,Q)$ is fully $\mathcal{E}\times\mathcal{F}$-normalized. A similar argument achieves the same result if instead $\mathcal{F}$ is saturated.
	\end{proof}
	
	Let $B$ be a block of $RG$ and let $(D,e_D)\in\BP_R(B)$ be a maximal $B$-Brauer pair. For each subgroup $P\subgp D$ let $e_P$ denote the unique block idempotent of $RC_G(P)$ such that $(P,e_P)\leq (D,e_D)$. The \textit{fusion system of }$B$ \textit{associated to} $(D,e_D)$ is the subcategory $\mathcal{F}=\mathcal{F}_{(D,e_D)}(G,B)$ of the category of all finite groups whose objects are the subgroups of $D$ and with morphism sets $\Hom_{\mathcal{F}}(P,Q)$, for two subgroups $P$ and $Q$ of $D$, defined as the set of group homomorphisms $\varphi:P\to Q$ for which there exists a $g\in G$ satisfying $\varphi=c_g$ and ${}^g(P,e_P)\leq(Q,e_Q)$. By the results of \cite[IV.3]{Aschbacher_2011} the fusion system $\mathcal{F}$ is saturated. Since any two maximal $B$-Brauer pairs are $G$-conjugate (see \cite[IV, Theorem 2.20]{Aschbacher_2011}) the isomorphism class of $\mathcal{F}$ does not depend on the choice of maximal $B$-Brauer pair. If $B\in\Bl(\OO G)$ note that $\mathcal{F}_{(D,e_D)}(G,B)=\mathcal{F}_{(D,\overline{e_D})}(G,\overline{B})$.
	
	Keep the notation set above. If $P,Q\subgp D$ and $\varphi:P\to Q$ is a morphism in $\mathcal{F}$ then $\varphi$ is an $\mathcal{F}$-isomorphism if and only if $\varphi=c_g$ for some $g\in G$ such that ${}^g(P,e_P)=(Q,e_Q)$. In fact, if $\varphi:P\isoto Q$ is an $\mathcal{F}$-isomorphism and $g\in G$ is such that $\varphi=c_g$ and ${}^g(P,e_P)\leq (Q,e_Q)$ then necessarily ${}^g(P,e_P)=(Q,e_Q)$.
	
	Note that if $(D,e_D)\in\BP_R(B)$ is a maximal $B$-Brauer pair then $(D,e_D^\ast)$ is a maximal $B^\ast$-Brauer pair and $\mathcal{F}_{(D,e_D)}(G,B)=\mathcal{F}_{(D,e_D^\ast)}(G,B^\ast)$.
	
	%could be helpful to restate lemma below in the case of two blocks.
	
	%part (b) says: if $\epsilon_Q'$ is the unique block of $RC_{I_P}(Q)$ covering $e_{PQ}$ then $(Q,\epsilon_Q')\leq(N_D(P),e_{N_D(P)}).
	
	\begin{lemma}\label{lem:normalizerbrauerpairs}
		Let $B\in\Bl(RG)$ and let $(D,e_D)\in\BP_R(B)$ be a maximal $B$-Brauer pair. If $P\subgp D$ write $e_P$ for the unique block idempotent of $RC_G(P)$ such that $(P,e_P)\leq(D,e_D)$ and set $I_P=N_G(P,e_P)$. Let $\mathcal{F}=\mathcal{F}_{(D,e_D)}(G,B)$. Fix $P\subgp D$.
		\begin{itemize}
			\item[(a)] $(N_D(P),e_{N_D(P)})$ is an $RI_Pe_P$-Brauer pair which is maximal if and only if $P$ is fully $\mathcal{F}$-normalized. In this case $\mathcal{N}_{\mathcal{F}}(P)=\mathcal{F}_{(N_D(P),e_{N_D(P)})}(I_P,e_P)$.
			\item[(b)] If $Q\subgp N_D(P)$ write $\epsilon_Q$ for the unique block idempotent of $RC_{I_P}(Q)$ such that $(Q,\epsilon_Q)\leq(N_D(P),e_{N_D(P)})$ is a containment of $RI_Pe_P$-Brauer pairs. Then $\epsilon_Q$ is the unique block of $RC_{I_P}(Q)$ covering $e_{PQ}$. Moreover,
			\begin{equation*}
				\epsilon_Q=\tr_{\Stab_{C_{I_P}(Q)}(e_{PQ})}^{C_{I_P}(Q)}(e_{PQ})
			\end{equation*} 
			and $\Stab_{C_{I_P}(Q)}(e_{PQ})=C_{I_P\cap I_{PQ}}(Q)$.
		\end{itemize}
	\end{lemma}
	
	\begin{proof}
		Part (a) is well-known: see \cite[IV, Theorem 3.19]{Aschbacher_2011}. Note that if $Q\subgp N_D(P)$ then $Q\subgp I_P$ and $C_{I_P}(PQ)=C_G(PQ)\nor C_{I_P}(Q)$, so by Lemma \ref{lem:blockcent2} there is a unique block of $RC_{I_P}(Q)$ that covers $e_{PQ}$. We must show that for each subgroup $Q\subgp N_D(P)$ the block $\epsilon_Q$ covers $e_{PQ}$. Suppose first that $P\subgp Q\subgp N_D(P)$. Then $C_{I_P}(Q)=C_G(Q)$, so $(Q,\epsilon_Q)$ is an $RG$-Brauer pair. The containment $(Q,\epsilon_Q)\leq(N_D(P),e_{N_D(P)})$ then implies the corresponding containment of $RG$-Brauer pairs. Hence $\epsilon_Q=e_Q=e_{PQ}$ and $\epsilon_Q$ clearly covers $e_{PQ}$ in this case. Now suppose that $P\not\subgp Q$. By what we have just shown $\epsilon_{PQ}=e_{PQ}$. Let $Q=Q_0\nor Q_1\nor\cdots\nor Q_n=PQ$ be a subnormal chain of subgroups. Then for each $i=0,\ldots,n-1$ we have a normal containment of $RI_Pe_P$-Brauer pairs $(Q_i,\epsilon_{Q_i})\nor(Q_{i+1},\epsilon_{Q_{i+1}})$. Now Lemma \ref{lem:blockcent} implies that $\epsilon_{Q_i}\in RC_{I_P}(PQ_i)=RC_G(PQ)$, so $\br_{Q_{i+1}}^{I_P}(\epsilon_{Q_i})=\overline{\epsilon_{Q_i}}$ for each $i$. It follows that $\overline{\epsilon_{Q_i}\epsilon_{Q_{i+1}}}=\overline{\epsilon_{Q_{i+1}}}$, which in turn implies that $\overline{\epsilon_Qe_{PQ}}\neq 0$. Thus we see that $\epsilon_Q$ covers $e_{PQ}$ for each $Q\subgp N_D(P)$. The final statement follows from part (b) of Lemma \ref{lem:brauereltnormalizer}.
	\end{proof}
	
	In the next couple of lemmas $G$ and $H$ denote finite groups and the $p$-modular system $(\K,\OO,F)$ is assumed large enough for $G\times H$.
	
	\begin{lemma}\label{lem:maxlABpair}
		Let $A\in\Bl(RG)$ and $B\in\Bl(RH)$. Let $(D,e_D)\in\BP_R(A)$ and $(E,f_E)\in\BP_R(B)$ be maximal Brauer pairs. Then $(D\times E,e_D\tensor f_E)$ is a maximal $A\tensor_R B$-Brauer pair. 
	\end{lemma}
	
	\begin{proof}
		By Lemma \ref{lem:brauerpairAB} we know that $(D\times E,e_D\tensor f_E)\in\BP_R(A\tensor_R B)$. Suppose that $(P,e\tensor f)\in\BP_R(A\tensor_R B)$ is such that $(D\times E,e_D\tensor f_E)\leq(P,e\tensor f)$. Then $(D,e_D)\leq(p_1(P),e)$ and $(E,f_E)\leq(p_2(P),f)$ by Lemma \ref{lem:pi1pi2}. By Lemma \ref{lem:brauerpairAB} we have $(p_1(P),e)\in\BP_R(A)$ and $(p_2(P),f)\in\BP_R(B)$. Since $(D,e_D)$ and $(E,f_E)$ are maximal Brauer pairs it follows that $(D,e_D)=(p_1(P),e)$ and $(E,f_E)=(p_2(P),f)$. Then $P\subgp p_1(P)\times p_2(P)=D\times E$, so we have $(D\times E,e_D\tensor f_E)=(P,e\tensor f)$. Therefore $(D\times E,e_D\tensor f_E)$ is a maximal $A\tensor_R B$-Brauer pair.
	\end{proof}
	
	\begin{lemma}\label{lem:tensorfusion}
		Let $A\in\Bl(RG)$ and $B\in\Bl(RH)$. Let $(D,e_D)\in\BP_R(A)$ and $(E,f_E)\in\BP_R(B)$ be maximal Brauer pairs. If $P\subgp D$ write $e_P$ for the unique block idempotent of $RC_G(P)$ such that $(P,e_P)\leq(D,e_D)$ and if $Q\subgp E$ write $f_Q$ for the unique block idempotent of $RC_H(Q)$ such that $(Q,f_Q)\leq(E,f_E)$. 
		\begin{itemize}
			\item[(a)] If $S\subgp D\times E$ then
			\begin{equation*}
				(S,e_{p_1(S)}\tensor f_{p_2(S)})\leq(D\times E,e_D\tensor f_E)
			\end{equation*}
			is a containment of $A\tensor_R B$-Brauer pairs.
			\item[(b)] Set $\mathcal{A}=\mathcal{F}_{(D,e_D)}(G,A)$ and $\mathcal{B}=\mathcal{F}_{(E,f_E)}(H,B)$. Then
			\begin{equation*}
				\mathcal{F}_{(D\times E,e_D\tensor f_E)}(G\times H,A\tensor_R B)=\mathcal{A}\times\mathcal{B}.
			\end{equation*}
		\end{itemize}
	\end{lemma}
	
	\begin{proof}
		(a) Let $S\subgp D\times E$. Then we have $(p_1(S),e_{p_1(S)})\leq(D,e_D)$ and $(p_2(S),f_{p_2(S)})\leq (E,f_E)$, so $(S,e_{p_1(S)}\tensor f_{p_2(S)})\leq(D\times E,e_D\tensor f_E)$ by Part (c) of Lemma \ref{lem:GHBrauerpaircontainment}.
		
		(b) Recall from Lemma \ref{lem:maxlABpair} that $(D\times E,e_D\tensor f_E)$ is a maximal $A\tensor_R B$-Brauer pair. So the fusion system $\mathcal{F}:=\mathcal{F}_{(D\times E,e_D\tensor f_E)}(G\times H,A\tensor_R B)$ is well-defined. Now $\mathcal{F}$ and $\mathcal{A}\times\mathcal{B}$ are both fusion systems over $D\times E$ so to show $\mathcal{F}=\mathcal{A}\times\mathcal{B}$ we only need to compare their morphism sets. Let $P_1,P_2\subgp D$, $Q_1,Q_2\subgp E$, $\psi_1\in\Hom_{\mathcal{A}}(P_1,P_2)$, and $\psi_2\in\Hom_{\mathcal{B}}(Q_1,Q_2)$. Then $\psi_1=c_g$ for some $g\in G$ such that ${}^g(P_1,e_{P_1})\leq(P_2,e_{P_2})$ and $\psi_2=c_h$ for some $h\in H$ such that ${}^h(Q_1,f_{Q_1})\leq(Q_2,f_{Q_2})$. Observe that $(\psi_1,\psi_2)=c_{(g,h)}$ and
		\begin{align*}
			{}^{(g,h)}(P_1\times Q_1,e_{P_1}\tensor f_{Q_1})	&=({}^gP_1\times {}^hQ_1,{}^ge_{P_1}\tensor {}^hf_{Q_1})\\
			&\leq(P_2\times Q_2,e_{P_2}\tensor f_{Q_2})
		\end{align*}
		by Part (c) of Lemma \ref{lem:GHBrauerpaircontainment}. So $(\psi_1,\psi_2)\in\Hom_{\mathcal{F}}(P_1\times Q_1,P_2\times Q_2)$. Since morphisms of the form $(\psi_1,\psi_2)$ generate $\mathcal{A}\times\mathcal{B}$ it follows that $\mathcal{A}\times\mathcal{B}\subseteq\mathcal{F}$. Now let $S,T\subgp D\times E$ and let $\psi\in\Hom_{\mathcal{F}}(S,T)$. Then by definition $\psi=c_{(g,h)}$ for some $(g,h)\in G\times H$ such that
		\begin{equation*}
			{}^{(g,h)}(S,e_{p_1(S)}\tensor f_{p_2(S)})\leq (T,e_{p_1(T)}\tensor f_{p_2(T)}).
		\end{equation*}
		Note that this implies ${}^g(p_1(S),e_{p_1(S)})\leq(p_1(T),e_{p_1(T)})$ and ${}^h(p_2(S),f_{p_2(S)})\leq(p_2(T),f_{p_2(T)})$ by Lemma \ref{lem:pi1pi2}. Therefore if we set $\psi_1=c_g:p_1(S)\to p_1(T)$ and $\psi_2=c_h:p_2(S)\to p_2(T)$ then $\psi_1\in\Hom_{\mathcal{A}}(p_1(S),p_1(T))$ and $\psi_2\in\Hom_{\mathcal{B}}(p_2(S),p_2(T))$. Now consider the morphism $(\psi_1,\psi_2):p_1(S)\times p_2(S)\to p_1(T)\times p_2(T)$. By definition, $(\psi_1,\psi_2)$ is a morphism in $\mathcal{A}\times\mathcal{B}$. Since $(\psi_1,\psi_2)|_S=\psi$ we see that $\psi$ is a morphism in $\mathcal{A}\times\mathcal{B}$. Thus we have $\mathcal{F}\subseteq\mathcal{A}\times\mathcal{B}$ and the proof is complete.
	\end{proof}
	
	\section{Class functions and characters}\label{sec:classfunctionsand}
	
	%rewrite some of this with just field of char 0 instead of p-mod system...
	
	Let $G$ be a finite group and let $(\K,\OO,F)$ stand for a $p$-modular system large enough for $G$. Write $\Irr_\K(G)=\Irr(\K G)$ for the set of characters of irreducible $\K G$-modules. Let $CF(G;\K)$ denote the $\K$-algebra of $\K$-valued class functions on $G$. Recall that $CF(G;\K)$ is endowed with the scalar product
	\begin{equation*}
		(\chi,\psi)_G=\frac{1}{|G|}\sum_{g\in G}\chi(g)\psi(g^{-1})
	\end{equation*}
	where $\chi,\psi\in CF(G;\K)$. The set $\Irr_\K(G)$ is an orthonormal basis for $CF(G;\K)$. Write $R_\K(G)=R(\K G)$ for the $\Z$-span of $\Irr_\K(G)$ in $CF(G;\K)$. Recall that $R_\K(G)$ is isomorphic to the Grothendieck ring of the category $\lmod{\K G}$ of finite-dimensional $\K G$-modules and that $\Irr_\K(G)$ forms a $\Z$-basis of $R_\K(G)$. 
	
	If $N$ is a normal subgroup of $G$ we write $R_\K(G/N)=R(\K[G/N])$ for the $\Z$-span in $R_\K(G)$ of the irreducible characters of $\K G$ that contain $N$ in their kernel. In other words, we will identify each character of the quotient group $G/N$ with its inflation to $G$. 
	
	Let $H\subgp G$ and let $g\in G$. Then the map
	\begin{align*}
		{}^g(\cdot):CF(H;\K)&\isoto CF({}^gH;\K)\\
		\chi&\mapsto({}^g\chi:{}^gh\mapsto\chi(h)\text{ for all }h\in H)
	\end{align*}
	is a $\K$-algebra isomorphism and satisfies $({}^g\chi,{}^g\psi)_{{}^gH}=(\chi,\psi)_H$ for all $\chi,\psi\in CF(H;\K)$. Note that ${}^g(\cdot)$ restricts to a ring isomorphism $R_\K(H)\isoto R_\K({}^gH)$. More generally, if $N\nor H\subgp G$ and $g\in G$ then ${}^g(\cdot)$ restricts to a ring isomorphism $R_\K(H/N)\isoto R_\K({}^gH/{}^gN)$. Note also that if $\chi\in CF(H;\K)$ and $g_1,g_2\in G$ then ${}^{g_2}({}^{g_1}\chi)={}^{g_2g_1}\chi$.
	
	A $\K$-linear map $\chi:\K G\to\K$ is called a \textit{central function} if $\chi(\alpha\beta)=\chi(\beta\alpha)$ for all $\alpha,\beta\in\K G$. Write $\CentFun(\K G)$ for the space of all central functions on $\K G$. We make $\CentFun(\K G)$ into a $Z(\K G)$-module by defining $(z\cdot\chi)(\alpha)=\chi(\alpha z)$ for all $z\in Z(\K G)$, $\chi\in\CentFun(\K G)$, and $\alpha\in\K G$. Each class function $G\to\K$ extends uniquely to a central function on $\K G$, and this gives rise to a $\K$-isomorphism $CF(G;\K)\isoto\CentFun(\K G)$. Thus $CF(G;\K)$ inherits a $Z(\K G)$-module structure given by $(z\cdot\chi)(g)=\chi(gz)$ for all $z\in Z(\K G)$, $\chi\in CF(G;\K)$, and $g\in G$. Note that $\chi(gz)$ is the value of the extension of $\chi$ to $\K G$ at $gz$. We will identify each class function on $G$ with its extension to $\K G$ without further comment.
	
	If $e$ is a central idempotent of $\K G$ set $CF(G,e;\K):=e\cdot CF(G;\K)$. Notice that
	\begin{equation*}
		CF(G,e;\K)=\set{\chi\in CF(G;\K)|\chi(ge)=\chi(g)\text{ for all }g\in G}.
	\end{equation*}
	Since we are assuming $\K$ is large enough, the set $\Irr_\K(G,e)=\Irr(\K Ge)$ of characters of irreducible $\K Ge$-modules forms a basis of $CF(G,e;\K)$. Write $R_\K(G,e)=R(\K Ge)$ for the $\Z$-span of $\Irr_\K(G,e)$ in $CF(G,e;\K)$. Recall that $R_\K(G,e)$ is isomorphic to the Grothendieck group of the category $\lmod{\K Ge}$ of finite-dimensional $\K Ge$-modules and that $\Irr_\K(G,e)$ forms a $\Z$-basis of $R_\K(G,e)$. Note also that $R_\K(G,e)=e\cdot R_\K(G)$.
	
	If $H\subgp G$, $g\in G$, and $e\in Z(\K H)$ is an idempotent then the isomorphism ${}^g(\cdot):R_\K(H)\isoto R_\K({}^gH)$ restricts to a group isomorphism $R_\K(H,e)\isoto R_\K({}^gH,{}^ge)$.
	
	If $I$ is a set of pairwise orthogonal idempotents of $Z(\K G)$ whose sum equals 1 then we have orthogonal decompositions
	\begin{equation*}
		CF(G;\K)=\bigoplus_{e\in I}CF(G,e;\K)\qquad\text{and}\qquad R_\K(G)=\bigoplus_{e\in I}R_\K(G,e).
	\end{equation*}
	In particular, for any $\chi,\psi\in CF(G;\K)$ one has
	\begin{equation*}
		(\chi,\psi)_G=\sum_{e\in I}(e\chi,e\psi)_G.
	\end{equation*}
	
	If $\overline{e}\in Z(FG)$ is an idempotent with lift $e\in Z(\OO G)\subseteq Z(\K G)$ then we write
	\begin{equation*}
		CF(G,\overline{e};\K):=CF(G,e;\K)\qquad\text{and}\qquad R_\K(G,\overline{e}):=R_\K(G,e).
	\end{equation*}
	Since every central idempotent of $FG$ has a unique lift to a central idempotent of $\OO G$, for any idempotent $f\in Z(FG)$ we may speak of $CF(G,f;\K)$ or $R_\K(G,f)$ without explicitly naming a lift of $f$ to $Z(\OO G)$.
	
	Let $N\nor G$ and let $e\in Z(\K G)$ be an idempotent. We set
	\begin{equation*}
		R_\K(G/N,e):=R_\K(G/N)\cap R_\K(G,e).
	\end{equation*}
	Note that $R_\K(G/N,e)$ is the subgroup of $R_\K(G)$ spanned by those $\chi\in\Irr_\K(G,e)$ that contain $N$ in their kernel. If $\overline{e}\in Z(FG)$ is an idempotent with lift $e\in Z(\OO G)$ set $R_\K(G/N,\overline{e}):=R_\K(G/N,e)$.
	
	Let $N\nor H\subgp G$ and let $e\in Z(\K H)$ be an idempotent. If $g\in G$ then the isomorphism ${}^g(\cdot):R_\K(H)\to R_\K({}^gH)$ restricts to a group isomorphism $R_\K(H/N,e)\isoto R_\K({}^gH/{}^gN,{}^ge)$.
	
	If $\chi\in CF(G;\K)$ we denote by $\chi^\circ$ the $\K$-valued class function on $G$ defined by
	\begin{equation*}
		\chi^\circ(g):=\chi(g^{-1})\qquad g\in G.
	\end{equation*}
	Note that if $\alpha\in\K G$ then $\chi^\circ(\alpha)=\chi(\alpha^\ast)$, where $\alpha^\ast$ is the image of $\alpha$ under the antipode of $\K G$. It follows easily that if $e$ is a central idempotent of $\K G$ and $\chi\in CF(G,e;\K)$ then $\chi^\circ\in CF(G,e^\ast;\K)$. If $M$ is a $\K G$-module affording the character $\chi$ then $\chi^\circ$ is the character of the dual (left) $\K G$-module $M^\circ:=\Hom_\K(M,\K)$. In particular, if $\chi\in\Irr(\K G)$ then $\chi^\circ\in\Irr(\K G)$. 
	
	Now let $G$ and $H$ be finite groups (and assume that $\K$ is large enough for $G\times H$). We identify $\K[G\times H]$ with $\K G\tensor_{\K}\K H$ via $(g,h)\mapsto g\tensor h$. Recall that the space of $\K$-valued class functions on $G\times H$ is identified with the space of $\K$-valued central functions on $\K[G\times H]$. Since we identify $\K[G\times H]$ with $\K G\tensor_{\K}\K H$, a central function on $\K[G\times H]$ is the same thing as a $\K$-bilinear map $\chi:\K G\times\K H\to\K$ that satisfies $\chi(\alpha\alpha',\beta\beta')=\chi(\alpha'\alpha,\beta'\beta)$ for all $\alpha,\alpha'\in\K G$ and all $\beta,\beta'\in\K H$.
	
	Let $e$ be a central idempotent of $\K G$ and let $f$ be a central idempotent of $\K H$. Then $e\tensor f$ is a central idempotent of $\K[G\times H]$. Under our identifications, we have
	\begin{equation*}
		CF(G\times H,e\tensor f;\K)=\set{\chi\in CF(G\times H;\K)|\chi(ge,hf)=\chi(g,h)\text{ for all }g\in G,h\in H}.
	\end{equation*}
	Let $\mu\in CF(G\times H,e\tensor f;\K)$. Then for any $g\in G$
	\begin{equation*}
		\mu(g,\cdot)\in CF(H,f;\K)
	\end{equation*}
	and for any $h\in H$
	\begin{equation*}
		\mu(\cdot,h)\in CF(G,e;\K).
	\end{equation*}
	More generally, for any $\alpha\in\K G$ we have
	\begin{equation*}
		\mu(\alpha,\cdot)\in CF(H,f;\K)
	\end{equation*}
	and for any $\beta\in\K H$ we have
	\begin{equation*}
		\mu(\cdot,\beta)\in CF(G,e;\K).
	\end{equation*}
	
	Now let $\mu\in CF(G\times H,e\tensor f^\ast;\K)$ where again $e$ is a central idempotent of $\K G$ and $f$ is a central idempotent of $\K H$ (so that $f^\ast$ is a central idempotent of $\K H$). If $\nu\in CF(H,f;\K)$ define $\mu\tensor_H\nu\in CF(G,e;\K)$ by the formula
	\begin{equation*}
		(\mu\tensor_H\nu)(g)=\frac{1}{|H|}\sum_{h\in H}\mu(g,h)\nu(h)\qquad g\in G.
	\end{equation*}
	If $M$ is a $(\K Ge,\K Hf)$-bimodule (i.e., a left $\K[G\times H](e\tensor f^\ast)$-module) with character $\mu$ and $N$ is a $\K Hf$-module with character $\nu$ then the character of $M\tensor_{\K H}N$ is $\mu\tensor_H\nu$, whence the definition above. (See \cite[Lemma 7.1.3]{Bouc_2010} for a proof of this character formula.) This construction yields a $\K$-bilinear map
	\begin{align*}
		-\tensor_H-:CF(G\times H,e\tensor f^\ast;\K)\times CF(H,f;\K)\to CF(G,e;\K)
	\end{align*} 
	defined by $(\mu,\nu)\mapsto \mu\tensor_H\nu$, and this bilinear map restricts to a biadditive map
	\begin{equation*}
		R_\K(G\times H,e\tensor f^\ast)\times R_\K(H,f)\to R_\K(G,e).
	\end{equation*}
	If $\chi\in\Irr(\K Ge)$ and $\psi,\psi'\in\Irr(\K Hf)$ then
	\begin{equation*}
		(\chi\times\psi^\circ)\tensor_H\psi'=\begin{cases}
			\chi	&\text{if }\psi=\psi'\\
			0		&\text{else.}
		\end{cases}
	\end{equation*}
	Notice that if $\mu\in CF(G\times H,e\tensor f^\ast;\K)$, $\nu\in CF(H,f;\K)$ and $g\in G$ then
	\begin{equation*}
		(\mu\tensor_H\nu)(g)=(\mu(g,\cdot),\nu^\circ)_H=(\mu(g,\cdot)^\circ,\nu)_H.
	\end{equation*}
	More generally, for any $\alpha\in\K G$ we have
	\begin{equation*}
		(\mu\tensor_H\nu)(\alpha)=(\mu(\alpha,\cdot),\nu^\circ)_H=(\mu(\alpha,\cdot)^\circ,\nu)_H.
	\end{equation*}

	\begin{proposition}
		Let $G$ and $H$ be finite groups, let $e$ be a central idempotent of $\K G$ and let $f$ be a central idempotent of $\K H$. The map
		\begin{align*}
			CF(G\times H,e\tensor f^\ast;\K)&\to\Hom_\K(CF(H,f;\K),CF(G,e;\K))\\
			\mu&\mapsto\mu\tensor_H-
		\end{align*}
		is an isomorphism of $\K$-vector spaces, with inverse
		\begin{align*}
			\Hom_\K(CF(H,f;\K),CF(G,e;\K))	&\to CF(G\times H,e\tensor f^\ast;\K)\\
			I&\mapsto\sum_{\psi\in\Irr(\K Hf)}I(\psi)\times\psi^\circ.
		\end{align*}
	\end{proposition}

	\begin{proof}
		Let $\Phi$ and $\Psi$ denote the maps defined first and second, respectively, in the statement of the proposition. We recall that if $\theta\in CF(H,f;\K)$ then $I(\theta)\times\theta^\circ\in CF(G\times H,e\tensor f^\ast;\K)$ is the class function defined by $(I(\theta)\times\theta^\circ)(g,h)=I(\theta)(g)\theta^\circ(h)$. Now if $I:CF(H,f;\K)\to CF(G,e;\K)$ is $\K$-transformation then
		\begin{equation*}
			\Phi(\Psi(I))=(\sum_{\psi\in\Irr(\K Hf)}I(\psi)\times\psi^\circ)\tensor_H-.
		\end{equation*}
		Let $\psi'\in\Irr(\K Hf)$ and let $g\in G$. Then
		\begin{align*}
			[(\sum_{\psi\in\Irr(\K Hf)}I(\psi)\times\psi^\circ)\tensor_H\psi'](g)	&=\sum_{\psi\in\Irr(\K Hf)}\frac{1}{|H|}\sum_{h\in H}I(\psi)(g)\psi^\circ(h)\psi'(h)\\
			&=\sum_{\psi\in\Irr(\K Hf)}I(\psi)(g)\cdot(\psi',\psi)_H\\
			&=I(\psi')(g).
		\end{align*}
		Since this holds for all $g\in G$, and all $\psi'\in\Irr(\K Hf)$ we find that $\Phi(\Psi(I))=I$. In particular, $\Phi$ is surjective. A dimension count then shows that $\Phi$ is an isomorphism, hence also $\Phi^{-1}=\Psi$. The proof is complete.
	\end{proof}

	\begin{corollary}
		Let $G$ and $H$ be finite groups, let $e$ be a central idempotent of $\K G$ and let $f$ be a central idempotent of $\K H$. The map
		\begin{align*}
			R_\K(G\times H,e\tensor f^\ast)&\to\Hom_\Z(R_\K(H,f),R_\K(G,e))\\
			\mu&\mapsto\mu\tensor_H-
		\end{align*}
		is a group isomorphism, with inverse
		\begin{align*}
			\Hom_\Z(R_\K(H,f),R_\K(G,e))	&\to R_\K(G\times H,e\tensor f^\ast)\\
			I&\mapsto\sum_{\psi\in\Irr(\K Hf)}I(\psi)\times\psi^\circ.
		\end{align*}
	\end{corollary}
	
	\begin{lemma}\label{lem:conjandtensor} 
		Suppose that $G'\subgp G$ and $H'\subgp H$. Let $e$ be a central idempotent of $\K G'$ and let $f$ be a central idempotent of $\K H'$. If $\mu\in CF(G'\times H',e\tensor f^\ast;\K)$ and $\nu\in CF(H',f;\K)$ then for any $(x,y)\in G\times H$ we have
	\begin{equation*}
		({}^{(x,y)}\mu)\tensor_{{}^yH'}({}^y\nu)={}^x(\mu\tensor_{H'}\nu).
	\end{equation*}
	\end{lemma}
	
	\section{Brauer characters and the generalized decomposition map}\label{sec:Brauercharsand}
	
	Let $G$ be a finite group and let $(\K,\OO,F)$ be a $p$-modular system large enough for $G$. Let $CF_{p'}(G;\K)$ denote the subspace of $CF(G;\K)$ consisting of class functions $\chi:G\to\K$ for which $\chi(g)=0$ if $g\notin G_{p'}$. 
	
	Write $\IBr_F(G)=\IBr(FG)$ for the set of irreducible Brauer characters of $FG$. By convention we view each irreducible Brauer character of $FG$ as an element of $CF_{p'}(G;\K)$; or in other words, we extend each irreducible Brauer character to a class function on $G$ that vanishes on $G-G_{p'}$. With this convention, $\IBr_F(G)$ is a basis of $CF_{p'}(G;\K)$. Another basis for this space is given by the set $\PrInd_\OO(G)=\Pr\Ind(\OO G)$, which is the set of characters of projective indecomposable $\OO G$-modules. Write $R_F(G)$ for the subgroup of $CF_{p'}(G;\K)$ spanned by $\IBr_F(G)$. Then $R_F(G)$ is isomorphic to the Grothendieck ring of the category $\lmod{FG}$ of finite-dimensional $FG$-modules and $\IBr_F(G)$ is a $\Z$-basis of $R_F(G)$. 
	
	Let $u\in G$ be a $p$-element. If $\chi\in CF(G;\K)$ set $d_G^u(\chi)\in CF_{p'}(C_G(u);\K)$ equal to the class function on $C_G(u)$ defined by
	\begin{equation*}
		d_G^u(\chi)(s):=\begin{cases}
			\chi(us)	& \text{if }s\in C_G(u)_{p'}\\
			0			& \text{if }s\notin C_G(u)_{p'}.
		\end{cases}
	\end{equation*}
	This construction yields a $\K$-linear map $d_G^u:CF(G;\K)\to CF_{p'}(C_G(u);\K)$ called the \textit{generalized decomposition map }(\textit{associated to }$u$). When $u=1$ we obtain the usual decomposition map $d_G:CF(G;\K)\to CF_{p'}(G;\K)$.
	
	%state conjugation thing below more generally for $H\subgp G$.
	
	Note that for any $g\in G$ and any $\chi\in CF(G;\K)$ one has ${}^gd_G^u(\chi)=d_G^{{}^gu}(\chi)$. Note also that for any $p$-element $u\in G$ and any $\chi\in CF(G;\K)$ one has $d_G^u(\chi^\circ)=d_G^{u^{-1}}(\chi)^\circ$. 
	
	%Add to prop: the map is an isometry.
	%The map in the proposition below is not an isometry in general with respect to the inner product $(\chi,\psi)=1/|G|\sum_{g\in G}\chi(g)\psi(g^{-1})$. Here's what's going on: over \C, there are actually two different inner products on CF(G;\C): one is defined as in the previous sentence, the other is (\chi,\psi)':=1/|G|\sum_{g\in G}\chi(g)\overline{\psi(g)} where \overline denotes complex conjugation. Now the map in the proposition is actually an isometry if one uses (\cdot,\cdot)'. Now the issue is we are working with K, not \C. So to make this work one would need to define a morphism K\to K that restricts to "complex conjugation" on \Q(roots of unity in K). This is possible, since in general a ring homomorphism K\to L between fields extends to a homomorphism L\to L (the homomorphism may not be surjective, though, and need not be unique). My plan for now is to see if I can do without this isometry.
	
	\begin{proposition}\label{prop:localclassfunction1}
		Let $\mathcal{U}$ be a set of representatives for the $G$-conjugacy classes of $p$-elements of $G$. Then
		\begin{equation*}
			\bigoplus_{u\in\mathcal{U}}d_G^u:CF(G;\K)\to\bigoplus_{u\in\mathcal{U}}CF_{p'}(C_G(u);\K)
		\end{equation*}
		is an isomorphism of $\K$-vector spaces. Moreover, for any $\chi,\psi\in CF(G;\K)$ one has
		\begin{equation*}
			(\chi,\psi)_G=\sum_{u\in\mathcal{U}}(d_G^u(\chi),d_G^{u^{-1}}(\psi))_{C_G(u)}.
		\end{equation*}
	\end{proposition}
	
	\begin{proof}
		For each $u\in\mathcal{U}$ let $\mathcal{S}_u$ be a set of representatives of the $C_G(u)$-conjugacy classes of $p'$-elements of $C_G(u)$. Then $\mathcal{G}:=\cup_{u\in\mathcal{U}}\set{us|s\in\mathcal{S}_u}$ is a set of representatives for the $G$-conjugacy classes of $G$. It follows that
		\begin{equation*}
			\dim_\K CF(G;\K)=\dim_\K \bigoplus_{u\in\mathcal{U}}CF_{p'}(C_G(u);\K).
		\end{equation*}
		In particular, to show $\dsum_{u}d_G^u$ is a $\K$-isomorphism it is enough to show that it is injective. 
		
		Let $\chi\in\ker(\dsum_{u}d_G^u)$. Then $d_G^u(\chi)=0$ for all $u\in\mathcal{U}$. In particular, $\chi(us)=d_G^u(\chi)(s)=0$ for all $u\in\mathcal{U}$ and $s\in\mathcal{S}_u$. But every element of $G$ is conjugate to an element of the form $us$ for some $u\in\mathcal{U}$ and $s\in\mathcal{S}_u$. Therefore $\chi=0$, which proves that $\dsum_{u}d_G^u$ is a $\K$-isomorphism.
		
		Finally, if $\chi,\psi\in CF(G;\K)$ then
		\begin{align*}
			\sum_{u\in\mathcal{U}}(d_G^u(\chi),d_G^{u^{-1}}(\psi))_{C_G(u)}	&=\sum_{u\in\mathcal{U}}\frac{1}{|C_G(u)|}\sum_{s\in C_G(u)}d_G^u(\chi)(s)d_G^{u^{-1}}(\psi)(s^{-1})\\
			&=\sum_{u\in\mathcal{U}}\frac{1}{|C_G(u)|}\sum_{s\in C_G(u)_{p'}}\chi(us)\psi(u^{-1}s^{-1})\\
			&=\sum_{u\in\mathcal{U}}\frac{1}{|C_G(u)|}\sum_{s\in\mathcal{S}_u}\frac{|C_G(u)|}{|C_{C_G(u)}(s)|}\chi(us)\psi((us)^{-1})\\
			&=\sum_{u\in\mathcal{U}}\sum_{s\in\mathcal{S}_u}\frac{1}{|C_G(us)|}\chi(us)\psi((us)^{-1})\\
			&=\sum_{g\in\mathcal{G}}\frac{1}{|C_G(g)|}\chi(g)\psi(g^{-1})\\
			&=\frac{1}{|G|}\sum_{g\in G}\chi(g)\psi(g^{-1})\\
			&=(\chi,\psi)_G.
		\end{align*}
	\end{proof}
	
	%Alternative: define the generalized decomposition numbers to be the numbers that appear in the matrix. Then show how they are computed using discussion below.
	
	We recall the definition of the \textit{generalized decomposition numbers}, following Radha Kessar's treatment in \cite[IV.5.1]{Aschbacher_2011}. Let $u\in G$ be a $p$-element and let $\chi\in\Irr_\K(G)$. Write
	\begin{equation*}
		\Res_{C_G(u)}^G\chi=\sum_{\zeta\in\Irr_\K(C_G(u))}n_{\chi,\zeta}\zeta\qquad n_{\chi,\zeta}\in\N_0.
	\end{equation*}
	Since $u\in Z(C_G(u))$, if $\zeta\in\Irr_\K(C_G(u))$ then there exists a root of unity $\lambda_{u,\zeta}\in \K$ of $p$-power order such that
	\begin{equation*}
		\zeta(uy)=\lambda_{u,\zeta}\zeta(y)\qquad\text{for all }y\in C_G(u).
	\end{equation*}
	For each $\zeta\in\Irr_\K(C_G(u))$ write
	\begin{equation*}
		d_{C_G(u)}^1(\zeta)=\sum_{\tau\in\IBr_F(C_G(u))}d_{\zeta,\tau}^{(u)}\tau\qquad d_{\zeta,\tau}^{(u)}\in\N_0.
	\end{equation*}
	Note that the $d_{\zeta,\tau}^{(u)}$ are just the usual decomposition numbers for the character $\zeta$ of the group $C_G(u)$. Now define, for each $\tau\in\IBr_F(C_G(u))$, the \textit{generalized decomposition number}
	\begin{equation*}
		d_{\chi,\tau}^u=\sum_{\zeta\in\Irr_\K(C_G(u))}n_{\chi,\zeta}\lambda_{u,\zeta}d_{\zeta,\tau}^{(u)}.
	\end{equation*}
	
	\begin{proposition}\label{prop:gendecomp1}
		Let $u\in G$ be a $p$-element and let $\chi\in\Irr_\K(G)$. Keep the notation set above. Then
		\begin{equation*}
			d_G^u(\chi)=\sum_{\tau\in\IBr_F(C_G(u))}d_{\chi,\tau}^u\tau.
		\end{equation*}
		In particular, the matrix of $d_G^u:CF(G;\K)\to CF_{p'}(C_G(u);\K)$ with respect to the bases $\Irr_\K(G)$ of $CF(G;\K)$ and $\IBr_F(C_G(u))$ of $CF_{p'}(C_G(u);\K)$ is the matrix whose entry in row $\tau$, column $\chi$ is $d_{\chi,\tau}^u$.
	\end{proposition}
	
	\begin{proof}
		Let $s\in C_G(u)$. We must show that $d_G^u(\chi)(s)=\sum_{\tau}d_{\chi,\tau}^u\tau(s)$. This is clear if $s\notin C_G(u)_{p'}$, so assume that $s\in C_G(u)_{p'}$. Then
		\begin{align*}
			\chi(us)	&=\Res_{C_G(u)}^G(\chi)(us)=\sum_{\zeta\in\Irr_\K(C_G(u))}n_{\chi,\zeta}\zeta(us)\\
			&=\sum_{\zeta\in\Irr_\K(C_G(u))}n_{\chi,\zeta}\lambda_{u,\zeta}\zeta(s)\\
			&=\sum_{\zeta\in\Irr_\K(C_G(u))}n_{\chi,\zeta}\lambda_{u,\zeta}\left(\sum_{\tau\in\IBr_F(C_G(u))}d_{\zeta,\tau}^{(u)}\tau(s)\right)\\
			&=\sum_{\tau\in\IBr_F(C_G(u))}\left(\sum_{\zeta\in\Irr_\K(C_G(u))}n_{\chi,\zeta}\lambda_{u,\zeta}d_{\zeta,\tau}^{(u)}\right)\tau(s)\\
			&=\sum_{\tau\in\IBr_F(C_G(u))}d_{\chi,\tau}^{u}\tau(s).
		\end{align*}
		So $d_G^u(\chi)(s)=\sum_{\tau}d_{\chi,\tau}^u\tau(s)$, and the proof is complete.
	\end{proof}
	
	If $e$ is a central idempotent of $FG$ write $\IBr_F(G,e)=\IBr(FGe)$ for the set of Brauer characters of irreducible $FGe$-modules. If $e$ is a central idempotent of $\OO G$ write $\PrInd_\OO(G,e)=\PrInd(\OO Ge)$ for the set of characters of projective indecomposable $\OO Ge$-modules.
	
	Let $e\in Z(\OO G)$ be an idempotent. Set $CF_{p'}(G,e;\K):=e\cdot CF_{p'}(G;\K)$. Note that $\PrInd_\OO(G,e)$ is a basis of $CF_{p'}(G,e;\K)$. It follows that
	\begin{equation*}
		CF_{p'}(G,e;\K)=CF_{p'}(G;\K)\cap CF(G,e;\K),
	\end{equation*}
	so $CF_{p'}(G,e;\K)$ is equal to the subspace of $CF(G;\K)$ formed by the class functions $\chi:G\to\K$ that satisfy $\chi(g)=0$ if $g\notin G_{p'}$ and $\chi(ge)=\chi(g)$ for all $g\in G$. Since the Cartan matrix of $FG$ is invertible $\IBr_F(G,\overline{e})$ is also a basis of $CF_{p'}(G,e;\K)$. Note that if $I$ is a set of pairwise orthogonal idempotents of $Z(\OO G)$ whose sum equals 1 then
	\begin{equation*}
		CF_{p'}(G;\K)=\bigoplus_{e\in I}CF_{p'}(G,e;\K).
	\end{equation*}
	
	If $(u,e)\in\BE_\OO(G)$ then the \textit{generalized decomposition map }(\textit{associated to }$(u,e)$) is the $\K$-linear map
	\begin{align*}
		d_G^{u,e}:CF(G;\K)	&\to CF_{p'}(C_G(u),e;\K)\\
		\chi&\mapsto e\cdot d_G^u(\chi).
	\end{align*}
	In other words, $d_G^{u,e}$ is the composition
	% https://q.uiver.app/?q=WzAsMyxbMCwwLCJDRihHO1xcSykiXSxbMSwwLCJDRl97cCd9KENfRyh1KTtcXEspIl0sWzIsMCwiQ0Zfe3AnfShDX0codSksZTtcXEspIl0sWzAsMSwiZF9HXnUiXSxbMSwyLCJlXFxjZG90LSJdXQ==
	\[\begin{tikzcd}
		{CF(G;\K)} & {CF_{p'}(C_G(u);\K)} & {CF_{p'}(C_G(u),e;\K).}
		\arrow["{d_G^u}", from=1-1, to=1-2]
		\arrow["{e\cdot-}", from=1-2, to=1-3]
	\end{tikzcd}\]
	
	Notice that if $u\in G$ is a $p$-element then $d_G^u=\sum_{e\in\bli(\OO C_G(u))}d_G^{u,e}$.
	
	\begin{proposition}\label{prop:gendecompvalues}
		Let $(u,e)\in\BE_\OO(G)$ and let $\chi\in CF(G;\K)$. If $s\in C_G(u)$ then
		\begin{equation*}
			d_G^{u,e}(\chi)(s)=\begin{cases}
				\chi(use)	&\text{if }s\in C_G(u)_{p'}\\
				0			&\text{if }s\notin C_G(u)_{p'}.
			\end{cases}
		\end{equation*}
	\end{proposition}
	
	\begin{proof}
		Since the map $d_G^{u,e}$ is $\K$-linear we may assume without loss that $\chi\in\Irr_\K(G)$. Let $s\in C_G(u)$. Since $d_G^{u,e}(\chi)\in CF_{p'}(C_G(u);\K)$ we know that $d_G^{u,e}(\chi)(s)=0$ if $s\notin C_G(u)_{p'}$. So we assume that $s\in C_G(u)_{p'}$. For $\zeta\in\Irr_\K(C_G(u))$, $\tau\in\IBr_F(C_G(u))$ let $n_{\chi,\zeta}$, $\lambda_{u,\zeta}$, and $d_{\zeta,\tau}^{(u)}$ be as in the definition of the generalized decomposition numbers $d_{\chi,\tau}^u$. By Proposition \ref{prop:gendecomp1}, 
		\begin{equation*}
			d_G^u(\chi)=\sum_{\tau\in\IBr_F(C_G(u))}d_{\chi,\tau}^u\tau.
		\end{equation*}
		Now if $\tau\in\IBr_F(C_G(u),\overline{e})$ then $\tau\in CF_{p'}(C_G(u),e;\K)$, so $e\cdot\tau=\tau$. If $\tau\notin\IBr_F(C_G(u),\overline{e})$ then $e\cdot\tau=0$. It follows that
		\begin{equation*}
			d_G^{u,e}(\chi)=e\cdot d_G^u(\chi)=\sum_{\tau\in\IBr_F(C_G(u),\overline{e})}d_{\chi,\tau}^u\tau.
		\end{equation*}
		Now since $use\in\OO C_G(u)$ we have
		\begin{align*}
			\chi(use)	&=(\Res_{C_G(u)}^G\chi)(use)\\
			&=\sum_{\zeta\in\Irr_\K(C_G(u))}n_{\chi,\zeta}\zeta(use)\\
			&=\sum_{\zeta\in\Irr_\K(C_G(u),e)}n_{\chi,\zeta}\zeta(us)\\
			&=\sum_{\zeta\in\Irr_\K(C_G(u),e)}n_{\chi,\zeta}\lambda_{u,\zeta}\zeta(s)\\
			&=\sum_{\zeta\in\Irr_\K(C_G(u),e)}n_{\chi,\zeta}\lambda_{u,\zeta}\left(\sum_{\tau\in\IBr_F(C_G(u))}d_{\zeta,\tau}^{(u)}\tau(s)\right)\\
			&=\sum_{\tau\in\IBr_F(C_G(u))}\left(\sum_{\zeta\in\Irr_\K(C_G(u),e)}n_{\chi,\zeta}\lambda_{u,\zeta}d_{\zeta,\tau}^{(u)}\right)\tau(s)\\
			&=\sum_{\tau\in\IBr_F(C_G(u),\overline{e})}\left(\sum_{\zeta\in\Irr_\K(C_G(u),e)}n_{\chi,\zeta}\lambda_{u,\zeta}d_{\zeta,\tau}^{(u)}\right)\tau(s)\\
			&=\sum_{\tau\in\IBr_F(C_G(u),\overline{e})}d_{\chi,\tau}^u\tau(s)\\
			&=d_G^{u,e}(\chi)(s).
		\end{align*}
		In the above, we have used the fact that the usual decomposition numbers $d_{\zeta,\tau}^{(u)}$ are zero unless $\zeta$ and $\tau$ belong to the same block of $C_G(u)$. This completes the proof.
	\end{proof}
	
	\begin{corollary}\label{cor:conjandgendecomp}
		Let $H\subgp G$, $(u,e)\in\BE_{\OO}(H)$, and $\chi\in CF(H;\K)$. Then for any $g\in G$ one has
		\begin{equation*}
			{}^gd_H^{u,e}(\chi)=d_{{}^gH}^{{}^gu,{}^ge}({}^g\chi)
		\end{equation*}
	\end{corollary}

	Proposition \ref{prop:gendecompvalues} also implies that $d_G^{u,e}(\chi^\circ)=d_G^{u^{-1},e^\ast}(\chi)^\circ$ for any $(u,e)\in\BE_{\OO}(G)$ and any $\chi\in CF(G;K)$.
	
	Note that the matrix of $d_G^{u,e}:CF(G;\K)\to CF_{p'}(C_G(u),e;\K)$ with respect to the bases $\Irr_\K(G)$ and $\IBr_F(C_G(u),\overline{e})$ has entry $d_{\chi,\tau}^u$ in row $\tau$, column $\chi$. Brauer's 2nd Main Theorem states that certain columns of this matrix must be zero. 
	
	\begin{theorem}\label{thm:Brauers2nd}
		(Brauer's 2nd Main Theorem) Let $(u,e)\in\BE_{\OO}(G)$ and let $B\in\Bl(\OO G)$. If $(u,e)\notin\BE_\OO(B)$ then $d_{\chi,\tau}^u=0$ for any $\chi\in\Irr_K(G,B)$ and any $\tau\in\IBr_F(C_G(u),\overline{e})$.
	\end{theorem}
	
	Let $(u,e)\in\BE_\OO(G)$. Then there exists a unique block $B\in\Bl(\OO G)$ such that $(u,e)\in\BE_\OO(B)$. If $A$ is any block of $\OO G$ different from $B$ and $\chi\in\Irr_\K(G,A)$ then Brauer's 2nd Main Theorem implies that $d_G^{u,e}(\chi)=0$. It follows that the map $d_G^{u,e}$ is completely determined by its restriction to the subspace $CF(G,B;K)$ spanned by the irreducible characters of $\K G$ that belong to $B$.
	
	\begin{proposition}\label{prop:localclassfunction2}
		Let $B\in\Bl(\OO G)$ and let $\mathcal{U}$ be a set of representatives for the $G$-conjugacy classes of $\BE_\OO(B)$. Then
		\begin{equation*}
			\bigoplus_{(u,e)\in\mathcal{U}}d_G^{u,e}:CF(G,B;\K)\to\bigoplus_{(u,e)\in\mathcal{U}}CF_{p'}(C_G(u),e;\K)
		\end{equation*}
		is an isomorphism of $\K$-vector spaces. Moreover, for any $\chi,\psi\in CF(G,B;\K)$ one has
		\begin{equation*}
			(\chi,\psi)_G=\sum_{(u,e)\in\mathcal{U}}(d_G^{u,e}(\chi),d_G^{u^{-1},e}(\psi))_{C_G(u)}.
		\end{equation*}
	\end{proposition}
	
	\begin{proof}
		Let $\chi\in CF(G,B;\K)$ be such that $d_G^{u,e}(\chi)=0$ for all $(u,e)\in\mathcal{U}$. I claim that $d_G^v(\chi)=0$ for all $p$-elements $v\in G$. Suppose not. Let $v\in G$ be a $p$-element such that $d_G^v(\chi)\neq 0$. Then there exists a block $f\in\bli(\OO C_G(v))$ such that $f\cdot d_G^v(\chi)\neq 0$, or in other words, $d_G^{v,f}(\chi)\neq 0$. Now $\chi=\sum_{\theta\in\Irr_\K(B)}c_\theta\theta$ for some scalars $c_\theta\in\K$, hence
		\begin{equation*}
			0\neq d_G^{v,f}(\chi)=\sum_{\theta\in\Irr_\K(B)}c_\theta d_G^{v,f}(\theta).
		\end{equation*}
		So $d_G^{v,f}(\theta)\neq 0$ for some $\theta\in\Irr_\K(B)$. Brauer's 2nd Main Theorem then implies that $(v,f)\in\BE_\OO(B)$. Let $(u,e)\in\mathcal{U}$ and $g\in G$ be such that ${}^g(v,f)=(u,e)$. Then one can verify that ${}^gd_G^{v,f}(\chi)=d_G^{u,e}(\chi)$. But by assumption $d_G^{u,e}(\chi)=0$, so $d_G^{v,f}(\chi)=0$, a contradiction. This proves the claim: we have $d_G^v(\chi)=0$ for all $p$-elements $v\in G$. Proposition \ref{prop:localclassfunction1} then gives that $\chi=0$. Thus we have shown that the map $\dsum_{(u,e)\in\mathcal{U}}d_G^{u,e}$ is injective, and it remains to see that the map is also surjective.
		
		Let $\sum_{(u,e)\in\mathcal{U}}\psi_{(u,e)}\in\dsum_{(u,e)\in\mathcal{U}}CF_{p'}(C_G(u),e;\K)$. So $\psi_{(u,e)}\in CF_{p'}(C_G(u),e;\K)$ for each $(u,e)\in\mathcal{U}$. If $(v,f)\in\BE_\OO(B)\setminus\mathcal{U}$ then there exists an element $g\in G$ such that ${}^g(v,f)\in\mathcal{U}$. Set $\psi_{(v,f)}={}^{g^{-1}}\psi_{{}^g(v,f)}$. Note that $\psi_{(v,f)}\in CF_{p'}(C_G(v),f;\K)$ and that the definition of $\psi_{(v,f)}$ does not depend on the choice of $g$. Now let $\mathcal{V}$ be a set of representatives for the $G$-conjugacy classes of $p$-elements of $G$. For each $v\in\mathcal{V}$ set
		\begin{equation*}
			\varphi_v:=\sum_{(v,f)\in\BE_\OO(B)}\psi_{(v,f)}\in CF_{p'}(C_G(v);\K).
		\end{equation*}
		The sum above is taken over all $B$-Brauer elements whose first component is equal to $v$. If there are no such Brauer elements then $\varphi_v=0$. By Proposition \ref{prop:localclassfunction1} there exists a class function $\chi\in CF(G;\K)$ such that $d_G^v(\chi)=\varphi_v$ for each $v\in\mathcal{V}$. I claim that $d_G^{u,e}(\chi)=\psi_{(u,e)}$ for all $(u,e)\in\mathcal{U}$. Let $(u,e)\in\mathcal{U}$. Since $u\in G$ is a $p$-element of $G$ there exists an element $g\in G$ such that $v:={}^gu\in\mathcal{V}$. Set $f:={}^ge$, so ${}^g(u,e)=(v,f)$. Note that $(v,f)\in\BE_\OO(B)$. Now
		\begin{equation*}
			d_G^{v,f}(\chi)=f\cdot d_G^v(\chi)=f\cdot\varphi_v=\psi_{(v,f)}={}^g\psi_{{}^{g^{-1}}(v,f)}={}^g\psi_{(u,e)}
		\end{equation*}
		and therefore $d_G^{u,e}(\chi)={}^{g^{-1}}d_G^{v,f}(\chi)=\psi_{(u,e)}$. This proves the claim. Since each $(u,e)\in\mathcal{U}$ belongs to $B$ we have $d_G^{u,e}(\chi)=d_G^{u,e}(e_B\cdot\chi)$ for all $(u,e)\in\mathcal{U}$. Now $e_B\cdot\chi\in CF(G,B;\K)$ and we have
		\begin{equation*}
			\left(\bigoplus_{(u,e)\in\mathcal{U}}d_G^{u,e}\right)(e_B\cdot\chi)=\sum_{(u,e)\in\mathcal{U}}\psi_{(u,e)}.
		\end{equation*}
		We have shown that $\dsum_{(u,e)\in\mathcal{U}}d_G^{u,e}$ is surjective, hence is an isomorphism.
		
		Now let $\chi,\psi\in CF(G,B;\K)$. Continue to let $\mathcal{V}$ denote a set of representatives for the $G$-conjugacy classes of $p$-elements of $G$. We compute:
		\begin{align*}
			\sum_{(u,e)\in\mathcal{U}}(d_G^{u,e}(\chi),d_G^{u^{-1},e}(\psi))_{C_G(u)}	&=\sum_{(u,e)\in\mathcal{U}}\frac{1}{|G:\Stab_G(u,e)|}\sum_{(v,f)\in\Orb_G(u,e)}(d_G^{v,f}(\chi),d_G^{v^{-1},f}(\psi))_{C_G(v)}\\
			&=\sum_{(u,e)\in\mathcal{U}}\frac{|C_G(u)|}{|G|}\sum_{(v,f)\in\Orb_G(u,e)}(d_G^{v,f}(\chi),d_G^{v^{-1},f}(\psi))_{C_G(v)}\\
			&=\sum_{(v,f)\in\BE_\OO(B)}\frac{|C_G(v)|}{|G|}(d_G^{v,f}(\chi),d_G^{v^{-1},f}(\psi))_{C_G(v)}\\
			&=\sum_{(v,f)\in\BE_\OO(G)}\frac{|C_G(v)|}{|G|}(d_G^{v,f}(\chi),d_G^{v^{-1},f}(\psi))_{C_G(v)}\\
			&=\sum_{\substack{v\in G\\\text{a }p\text{-element}}}\frac{|C_G(v)|}{|G|}\sum_{f\in\bli(\OO C_G(v))}(f\cdot d_G^v(\chi),f\cdot d_G^{v^{-1}}(\psi))_{C_G(v)}\\
			&=\sum_{\substack{v\in G\\\text{a }p\text{-element}}}\frac{|C_G(v)|}{|G|}(d_G^v(\chi),d_G^{v^{-1}}(\psi))_{C_G(v)}\\
			&=\sum_{v\in\mathcal{V}}(d_G^v(\chi),d_G^{v^{-1}}(\psi))_{C_G(v)}\\
			&=(\chi,\psi)_G.
		\end{align*}
		The fourth equality above follows from Brauer's 2nd Main Theorem (Theorem \ref{thm:Brauers2nd}) and the final equality follows from Proposition \ref{prop:localclassfunction1}. The proof is complete.
	\end{proof}
	
	Let $B\in\Bl(\OO G)$. If $\chi\in CF(G,B;\K)$ then Propositions \ref{prop:gendecompvalues} and \ref{prop:localclassfunction2} imply that $\chi$ is completely determined by the values $\chi(use)$ where $u\in G$ is a $p$-element, $s\in C_G(u)_{p'}$, and $e\in\bli(\OO C_G(u))$ such that $(u,e)\in\BE_\OO(B)$.
	
	Now let $G$ and $H$ be finite groups and assume that $(\K,\OO,F)$ is large enough for $G\times H$. Let $e\in Z(\OO G)$ and $f\in Z(\OO H)$ be central idempotents. Then $e\tensor f^\ast$ is a central idempotent of $\OO[G\times H]$. Let $\mu\in CF_{p'}(G\times H,e\tensor f^\ast;\K)$. Then for any $\nu\in CF(H,f;\K)$ we have $\mu\tensor_H\nu\in CF_{p'}(G,e;\K)$. Indeed, $\mu\tensor_H\nu$ is certainly an element of $CF(G,e;\K)$ and if $g\in G\setminus G_{p'}$ then $(g,h)\notin(G\times H)_{p'}$ for any $h\in H$ hence
	\begin{equation*}
		(\mu\tensor_H\nu)(g)=\frac{1}{|H|}\sum_{h\in H}\mu(g,h)\nu(h)=0.
	\end{equation*}
	
	\begin{proposition}\label{prop:Brouediagram}
		Let $G$ and $H$ be finite groups, let $A\in\Bl(\OO G)$, and let $B\in\Bl(\OO H)$. Let $\mathcal{U}$ denote a set of representatives for the $G$-conjugacy classes of $\BE_\OO(A)$ and let $\mathcal{V}$ denote a set of representatives for the $H$-conjugacy classes of $\BE_\OO(B)$. Let $\mu\in CF(G\times H,A\tensor B^\ast;\K)$. Then the $\K$-linear map
		\begin{align*}
			\bigoplus_{(v,f)\in\mathcal{V}}CF_{p'}(C_H(v),f;\K)	&\to\bigoplus_{(u,e)\in\mathcal{U}}CF_{p'}(C_G(u),e;\K)\\
			\sum_{(v,f)\in\mathcal{V}}\nu_{(v,f)}&\mapsto\sum_{(u,e)\in\mathcal{U}}\sum_{(v,f)\in\mathcal{V}}d_{G\times H}^{(u,v),e\tensor f^\ast}(\mu)\tensor_{C_H(v)}\nu_{(v,f)}
		\end{align*}
		is the unique map making the diagram below commute:
		% https://q.uiver.app/#q=WzAsNCxbMCwwLCJDRihILEI7XFxLKSJdLFsxLDAsIkNGKEcsQTtcXEspIl0sWzAsMSwiXFxiaWdvcGx1c197KHYsZilcXGluXFxtYXRoY2Fse1Z9fUNGX3twJ30oQ19IKHYpLGY7XFxLKSJdLFsxLDEsIlxcYmlnb3BsdXNfeyh1LGUpXFxpblxcbWF0aGNhbHtVfX1DRl97cCd9KENfRyh1KSxlO1xcSykiXSxbMCwxLCJcXG11XFx0ZW5zb3JfSC0iXSxbMCwyLCJcXGRzdW1feyh2LGYpXFxpblxcbWF0aGNhbHtWfX1kX0hee3YsZn0iLDJdLFsxLDMsIlxcZHN1bV97KHUsZSlcXGluXFxtYXRoY2Fse1V9fWRfR157dSxlfSJdLFsyLDMsIiIsMix7InN0eWxlIjp7ImJvZHkiOnsibmFtZSI6ImRhc2hlZCJ9fX1dXQ==
		\[\begin{tikzcd}
			{CF(H,B;\K)} & {CF(G,A;\K)} \\
			{\bigoplus_{(v,f)\in\mathcal{V}}CF_{p'}(C_H(v),f;\K)} & {\bigoplus_{(u,e)\in\mathcal{U}}CF_{p'}(C_G(u),e;\K)}
			\arrow["{\mu\tensor_H-}", from=1-1, to=1-2]
			\arrow["{\bigoplus_{(v,f)\in\mathcal{V}}d_H^{v,f}}"', from=1-1, to=2-1]
			\arrow["{\bigoplus_{(u,e)\in\mathcal{U}}d_G^{u,e}}", from=1-2, to=2-2]
			\arrow[dashed, from=2-1, to=2-2]
		\end{tikzcd}\]
	\end{proposition}

	\begin{proof}
		Note that if $(u,e)\in\mathcal{U}$ and $(v,f)\in\mathcal{V}$ then $((u,v),e\tensor f^\ast)$ is a Brauer element of $\OO[G\times H]$. Therefore $d_{G\times H}^{(u,v),e\tensor f^\ast}(\mu)\in CF_{p'}(C_G(u)\times C_H(v),e\tensor f^\ast;\K)$. In particular, $d_{G\times H}^{(u,v),e\tensor f^\ast}(\mu)\tensor_{C_H(v)}-$ defines a $\K$-linear map $CF_{p'}(C_H(v),f;\K)\to CF_{p'}(C_G(u),e;\K)$. Therefore the map defined in the statement makes sense.
		
		Let $\nu\in CF(H,B;\K)$. To see that the diagram commutes we must show that
		\begin{equation*}
			\sum_{(u,e)\in\mathcal{U}}d_G^{u,e}(\mu\tensor_H\nu)=\sum_{(u,e)\in\mathcal{U}}\sum_{(v,f)\in\mathcal{V}}d_{G\times H}^{(u,v),e\tensor f^\ast}(\mu)\tensor_{C_H(v)}d_H^{v,f}(\nu).
		\end{equation*}
		To accomplish this we will check that
		\begin{equation*}
			d_G^{u,e}(\mu\tensor_H\nu)=\sum_{(v,f)\in\mathcal{V}}d_{G\times H}^{(u,v),e\tensor f^\ast}(\mu)\tensor_{C_H(v)}d_H^{v,f}(\nu)
		\end{equation*}
		for any fixed $(u,e)\in\mathcal{U}$. Let $s\in C_G(u)_{p'}$. Note first that for any $(v,f)\in\mathcal{V}$ we have
		\begin{equation*}
			d_H^{v,f^\ast}(\mu(use,\cdot))=d_{G\times H}^{(u,v),e\tensor f^\ast}(\mu)(s,\cdot).
		\end{equation*}
		We compute:
		\begin{align*}
			d_G^{u,e}(\mu\tensor_H\nu)(s)	&=(\mu\tensor_H\nu)(use)\\
				&=(\nu,\mu(use,\cdot)^\circ)_H\\
				&=\sum_{(v,f)\in\mathcal{V}}(d_H^{v,f}(\nu),d_H^{v^{-1},f}(\mu(use,\cdot)^\circ))_{C_H(v)}\\
				&=\sum_{(v,f)\in\mathcal{V}}(d_H^{v,f}(\nu),d_H^{v,f^\ast}(\mu(use,\cdot))^\circ)_{C_H(v)}\\
				&=\sum_{(v,f)\in\mathcal{V}}(d_H^{v,f}(\nu),d_{G\times H}^{(u,v),e\tensor f^\ast}(\mu)(s,\cdot)^\circ)_{C_H(v)}\\
				&=\sum_{(v,f)\in\mathcal{V}}(d_{G\times H}^{(u,v),e\tensor f^\ast}(\mu)\tensor_{C_H(v)}d_H^{v,f}(\nu))(s).
		\end{align*}
		The first equality above holds by Proposition \ref{prop:gendecompvalues} and the third by Proposition \ref{prop:localclassfunction2}. Since $s$ was an arbitrary element of $C_G(u)_{p'}$ we find that
		\begin{equation*}
			d_G^{u,e}(\mu\tensor_H\nu)=\sum_{(v,f)\in\mathcal{V}}d_{G\times H}^{(u,v),e\tensor f^\ast}(\mu)\tensor_{C_H(v)}d_H^{v,f}(\nu).
		\end{equation*}
		It follows that
		\begin{equation*}
			\sum_{(u,e)\in\mathcal{U}}d_G^{u,e}(\mu\tensor_H\nu)=\sum_{(u,e)\in\mathcal{U}}\sum_{(v,f)\in\mathcal{V}}d_{G\times H}^{(u,v),e\tensor f^\ast}(\mu)\tensor_{C_H(v)}d_H^{v,f}(\nu)
		\end{equation*}
		and since $\nu$ was an arbitrary class function in $CF(H,B;\K)$ the diagram commutes.
		
		The uniqueness of the map follows from the fact that the vertical arrows in the diagram are isomorphisms, thanks to Proposition \ref{prop:localclassfunction2}.
	\end{proof}
	
	\section{Trivial source modules}
	
	Let $G$ be a finite group and let $(\K,\OO,F)$ be a $p$-modular system large enough for $G$. Let $R\in\set{\OO,F}$. In this section we establish some results about trivial source $RG$-modules that will be needed in the sequel. Recall that an $RG$-module $M$ is called a \textit{trivial source} or $p$-\textit{permutation module} if $\Res_P^GM$ is a permutation $RP$-module for all $p$-subgroups $P$ of $G$. We write $\ltriv{RG}$ for the category of (finitely generated) trivial source $RG$-modules and we let $T_R(G)=T(RG)$ denote the Grothendieck ring of $\ltriv{RG}$. 
	
	Recall that if $N$ is a trivial source $FG$-module then there exists a unique (up to isomorphism) trivial source $\OO G$-module $M$ such that $F\tensor_\OO M\iso N$.
	%in particular, grothendieck groups are iso over \OO or over F and if M is lift of N, M' is lift of N', then M and M' are iso if and only if N and N' are.
	
	%suggestion for notation: let M be triv source RG-module, P a p-subgroup. Then M(P) will be viewed over R. 
	Let $M$ be a trivial source $\OO G$-module and let $P$ be a $p$-subgroup of $G$. We will write $\overline{M}(P)\in\ltriv{FN_G(P)}$ for the usual Brauer construction applied to $M$ and we will write $M(P)\in\ltriv{\OO N_G(P)}$ for a lift of $\overline{M}(P)$. Of course, $M(P)$ is only well-defined up to isomorphism. If $(P,e)\in\BP_F(G)$ we write
	\begin{equation*}
		\overline{M}(P,e):=e\Res_{N_G(P,e)}^{N_G(P)}\overline{M}(P)\in\ltriv{FN_G(P,e)e}
	\end{equation*}
	and if $(P,e)\in\BP_\OO(G)$ then we write
	\begin{equation*}
		M(P,e):=e\Res_{N_G(P,e)}^{N_G(P)}M(P)\in\ltriv{\OO N_G(P,e)e}.
	\end{equation*}
	Note that $M(P,e)$ is only well-defined up to isomorphism.
	
	% in lemma below, clarify that the "unique lift" belongs to C_I_(P,e)(Q)...
	
	\begin{lemma}\label{lem:brauerconst}
		Let $B\in\Bl(\OO G)$, $(P,e)\in\BP_\OO(B)$, and set $I_{(P,e)}=N_G(P,e)$. Let $M\in\ltriv{B}$ and let $Q$ be a $p$-subgroup of $I_{(P,e)}$. Then there is an isomorphism of $\OO N_{I_{(P,e)}}(Q)$-modules
		\begin{equation*}
			(M(P,e))(Q)\iso e(Q)\cdot\Res_{N_{I_{(P,e)}}(Q)}^{N_G(PQ)}(M(PQ))
		\end{equation*}
		where $e(Q)$ denotes the unique lift of $\br_Q^{I_{(P,e)}}(e)$ to a central idempotent of $\OO N_{I_{(P,e)}}(Q)$.
	\end{lemma}
	
	\begin{proof}
		Note that $(M(P,e))(Q)$ is an $\OO$-lift of $\overline{(\overline{M}(P,\overline{e}))}(Q)$ and that $e(Q)\cdot\Res_{N_{I_{(P,e)}}(Q)}^{N_G(PQ)}(M(PQ))$ is an $\OO$-lift of $\br_Q^{I_{(P,e)}}(\overline{e})\cdot\Res_{N_{I_{(P,e)}}(Q)}^{N_G(PQ)}(\overline{M}(PQ))$. Thus we only need to show that there is an isomorphism of $FN_{I_{(P,e)}}(Q)$-modules
		\begin{equation*}
			\overline{(\overline{M}(P,\overline{e}))}(Q)\iso \br_Q^{I_{(P,e)}}(\overline{e})\cdot\Res_{N_{I_{(P,e)}}(Q)}^{N_G(PQ)}(\overline{M}(PQ)).
		\end{equation*}
		First note that there is an isomorphism of $FN_{I_{(P,e)}}(Q)$-modules
		\begin{equation*}
			\overline{(\overline{M}(P,\overline{e}))}(Q)\iso\br_Q^{I_{(P,e)}}(\overline{e})\cdot(\overline{(\Res_{I_{(P,e)}}^{N_G(P)}(\overline{M}(P)))}(Q))
		\end{equation*}
		by \cite[Lemma 3.7]{Boltje_2020} (with the group $G$ of the Lemma replaced with $I_{(P,e)}$, $M$ replaced with $\Res_{I_{(P,e)}}^{N_G(P)}(\overline{M}(P))$, $P$ replaced with $Q$, $i$ with $\overline{e}$, and $H$ with $I_{(P,e)}$). Now observe that we have an isomorphism of $FN_{I_{(P,e)}}(Q)$-modules
		\begin{align*}
			\overline{(\Res_{I_{(P,e)}}^{N_G(P)}(\overline{M}(P)))}(Q)	&=\overline{(\Res_{N_{I_{(P,e)}}(Q)}^{I_{(P,e)}}\Res_{I_{(P,e)}}^{N_G(P)}(\overline{M}(P)))}(Q)\\
			&=\overline{(\Res_{N_{I_{(P,e)}}(Q)}^{N_G(P)}(\overline{M}(P)))}(Q)\\
			&=\Res_{N_{I_{(P,e)}}(Q)}^{N_G(P)\cap N_G(Q)}(\overline{(\overline{M}(P))}(Q))\\
			&\iso \Res_{N_{I_{(P,e)}}(Q)}^{N_G(P)\cap N_G(Q)}(\Res_{N_G(P)\cap N_G(Q)}^{N_G(PQ)}(\overline{M}(PQ)))\\
			&=\Res_{N_{I_{(P,e)}}(Q)}^{N_G(PQ)}(\overline{M}(PQ)).
		\end{align*}
		So the result follows.
	\end{proof}
	
	%The lemma below is needed for the proof that pperm implies isotypy, but it should be recorded next to the one above. Is the assumption that U is abelian necessary?
	
	\begin{lemma}\label{lem:forppermimpliesisotypy}
		Let $G$ be a finite group, $B\in\Bl(\OO G)$. Let $M\in\ltriv{B}$, $(P,e)\in\BP_\OO(B)$, and $(U,\epsilon)\in\BP_\OO(C_G(P),e)$ such that $U$ is abelian. Then $(PU,\epsilon)$ is a $B$-Brauer pair, $(P,e)\nor(PU,\epsilon)$, and there is an isomorphism of $\OO C_G(PU)$-modules
		\begin{equation*}
			\Res_{C_G(PU)}^{N_{C_G(P)}(U,\epsilon)}((\Res_{C_G(P)}^{N_G(P,e)}M(P,e))(U,\epsilon))\iso\Res_{C_G(PU)}^{N_G(PU,\epsilon)}M(PU,\epsilon).
		\end{equation*}
	\end{lemma}
	
	\begin{proof}
		Since $C_{C_G(P)}(U)=C_G(PU)$ the idempotent $\epsilon$ is a block of $\OO C_G(PU)$. In particular, $(PU,\epsilon)\in\BP_\OO(G)$. Since $(U,\epsilon)\in\BP_\OO(C_G(P),e)$ we have that $\br_U^{C_G(P)}(e)\overline{\epsilon}=\overline{\epsilon}$. But $\br_U^{C_G(P)}(e)=\br_{PU}^G(e)$, so $(P,e)\nor(PU,\epsilon)$ and in particular $(PU,\epsilon)$ belongs to $B$. Now by Lemma \ref{lem:brauerconst} there is an isomorphism of $\OO[N_G(P,e)\cap N_G(U)]$-modules
		\begin{equation*}
			(M(P,e))(U)\iso e(U)\cdot\Res_{N_G(P,e)\cap N_G(U)}^{N_G(PU)}M(PU)
		\end{equation*}
		where $e(U)$ is the unique lift of $\br_U^{N_G(P,e)}(e)$ to a central idempotent of $\OO[N_G(P,e)\cap N_G(U)]$. Since $\br_U^{N_G(P,e)}(e)=\br_{PU}^G(e)$ we have $e(U)\cdot\epsilon=\epsilon$. By restricting to $C_G(PU)$ and cutting with $\epsilon$ we obtain an isomorphism of $\OO C_G(PU)$-modules
		\begin{align*}
			\epsilon\cdot\Res_{C_G(PU)}^{N_G(P,e)\cap N_G(U)}((M(P,e))(U))	&\iso \epsilon\Res_{C_G(PU)}^{N_G(PU)}M(PU)\\
			&=\Res_{C_G(PU)}^{N_G(PU,\epsilon)}M(PU,\epsilon).
		\end{align*}
		By applying Remark 3.2(b) of \cite{Boltje_2020}, noting that $U\subgp C_G(PU)$ since $U$ is assumed to be abelian, we compute that
		\begin{align*}
			\epsilon\cdot&\Res_{C_G(PU)}^{N_G(P,e)\cap N_G(U)}((M(P,e))(U))\\	&=\epsilon\cdot((\Res_{C_G(PU)}^{N_G(P,e)}(M(P,e)))(U))\\
			&=\epsilon\cdot((\Res_{C_G(PU)}^{C_G(P)}\Res_{C_G(P)}^{N_G(P,e)}M(P,e))(U))\\
			&=\epsilon\cdot\Res_{C_G(PU)}^{N_{C_G(P)}(U)}((\Res_{C_G(P)}^{N_G(P,e)}M(P,e))(U))\\
			&=\epsilon\cdot\Res_{C_G(PU)}^{N_{C_G(P)}(U,\epsilon)}\Res_{N_{C_G(P)}(U,\epsilon)}^{N_{C_G(P)}(U)}((\Res_{C_G(P)}^{N_G(P,e)}M(P,e))(U))\\
			&=\Res_{C_G(PU)}^{N_{C_G(P)}(U,\epsilon)}(\epsilon\cdot\Res_{N_{C_G(P)}(U,\epsilon)}^{N_{C_G(P)}(U)}((\Res_{C_G(P)}^{N_G(P,e)}M(P,e))(U)))\\
			&=\Res_{C_G(PU)}^{N_{C_G(P)}(U,\epsilon)}((\Res_{C_G(P)}^{N_G(P,e)}M(P,e))(U,\epsilon)).
		\end{align*}
	\end{proof}
	
	\section{Coherence conditions}\label{sec:coherenceconditions}
	
	Let $G$ be a finite group and let $(\K,\OO,F)$ be a $p$-modular system large enough for $G$. Consider the product
	\begin{equation*}
		\prod_{P\in S_p(G)}R_\K(N_G(P)/P).
	\end{equation*}
	If $(\chi_P)_{P\in S_p(G)}$ is an element of this product and $g\in G$ define ${}^g(\chi_P)$ to be the tuple whose $P$th entry, for $P\in S_p(G)$, is equal to ${}^g\chi_{{}^{g^{-1}}P}$. Thus
	\begin{equation*}
		{}^g(\chi_P)_{P\in S_p(G)}=({}^g\chi_{{}^{g^{-1}}P})_{P\in S_p(G)}. 
	\end{equation*}
	This action makes $\prod_{P\in S_p(G)}R_\K(N_G(P)/P)$ into a $\Z G$-algebra. Therefore the subset $\left(\prod_{P\in S_p(G)}R_\K(N_G(P)/P)\right)^G$ of $G$-fixed tuples forms a unital subring. Note that this subring consists of all tuples $(\chi_P)_{P\in S_p(G)}$ that satisfy ${}^g\chi_P=\chi_{{}^gP}$ for all $g\in G$ and all $p$-subgroups $P\subgp G$.
	
	In \cite{Boltje_2022} Boltje and Carman introduce a ring homomorphism
	\begin{align*}
		\beta_G:=\beta:T_\OO(G)&\to\left(\prod_{P\in S_p(G)}R_\K(N_G(P)/P)\right)^G\\
		[M]&\mapsto(\chi_{M(P)})_{P\in S_p(G)}
	\end{align*}
	where $M$ is a trivial source $\OO G$-module and $\chi_{M(P)}$ is the character of an $\OO$-lift $M(P)$ of the Brauer construction of $M$ at $P$. Their main theorem about this homomorphism is given below.
	
	\begin{theorem}\label{thm:coherence1}
		(\cite[Theorem A]{Boltje_2022}) The ring homomorphism $\beta_G$ is injective and its image consists of those tuples $(\chi_P)\in \left(\prod_{P\in S_p(G)}R_\K(N_G(P)/P)\right)^G$ that satisfy
		\begin{equation*}
			\chi_P(x)=\chi_{P\gp{x_p}}(x)
		\end{equation*}
		for all $P\in S_p(G)$ and $x\in N_G(P)$, where $x_p$ denotes the $p$-part of $x$. 
	\end{theorem}

	Alternatively, the image of $\beta_G$ consists of those $G$-invariant tuples $(\chi_P)$ in $\prod_{P\in S_p(G)}R_\K(N_G(P)/P)$ that satisfy
	\begin{equation*}
		\chi_P(us)=\chi_{P\gp{u}}(s)
	\end{equation*}
	for all $P\in S_p(G)$, $u\in N_G(P)$ a $p$-element, and $s\in C_{N_G(P)}(u)_{p'}$. This condition is in turn equivalent to the condition that
	\begin{equation*}
		d_{N_G(P)}^u(\chi_P)=d_{C_{N_G(P)}(u)}^1(\Res_{C_{N_G(P)}(u)}^{N_G(P\gp{u})}(\chi_{P\gp{u}}))
	\end{equation*}
	for all $P\in S_p(G)$ and all $p$-elements $u\in N_G(P)$.
	
	Now let $e\in Z(\OO G)$ be an idempotent. If $P$ is a $p$-subgroup of $G$ then $\br_P(e)$ is an $N_G(P)$-stable central idempotent of $FC_G(P)$ and hence also a central idempotent of $FN_G(P)$. Thus we can consider the group $R_\K(N_G(P)/P,\br_P(e))$ which is the subgroup of $R_\K(N_G(P))$ spanned by the irreducible characters $\chi\in\Irr_\K(N_G(P))$ that belong to the unique lift of $\br_P(e)$ to a central idempotent of $\OO N_G(P)$ and that contain $P$ in their kernel. Note that if $\br_P(e)=0$ then $R_\K(N_G(P)/P,\br_P(e))=\set{0}$.
	
	We have
	\begin{equation*}
		\prod_{P\in S_p(G)}R_\K(N_G(P)/P,\br_P(e))\subseteq\prod_{P\in S_p(G)}R_\K(N_G(P)/P)
	\end{equation*}
	and the product on the left is a $G$-stable subgroup of the product on the right. Therefore we can consider the subgroup of $G$-fixed tuples
	\begin{equation*}
		\left(\prod_{P\in S_p(G)}R_\K(N_G(P)/P,\br_P(e))\right)^G
	\end{equation*}
	which again consists of all tuples $(\chi_P)\in\prod_{P\in S_p(G)}R_\K(N_G(P)/P,\br_P(e))$ that satisfy ${}^g\chi_P=\chi_{{}^gP}$ for all $P\in S_p(G)$ and all $g\in G$. Boltje and Carman show that the image of $T_\OO(G,e)$ under the map $\beta_G$ is contained in this subgroup. More precisely we have the following:
	
	\begin{theorem}\label{thm:coherence2}
		(\cite[Corollary 3.3]{Boltje_2022}) Let $e$ be a central idempotent of $\OO G$. Then the image of $T_\OO(G,e)$ under $\beta_G$ is the subgroup of tuples $(\chi_P)\in\left(\prod_{P\in S_p(G)}R_\K(N_G(P)/P,\br_P(e))\right)^G$ that satisfy $\chi_P(x)=\chi_{P\gp{x_p}}(x)$ for all $P\in S_p(G)$ and $x\in N_G(P)$.
	\end{theorem}
	
	Let $B\in\Bl(\OO G)$. Our goal now is to give new ``coherent character conditions'' as in Theorems \ref{thm:coherence1} and \ref{thm:coherence2} that describe the group $T_\OO(B)$. 
	
	If $(P,e)\in\BP_\OO(G)$ set $I_{(P,e)}:=N_G(P,e)$. Note that $N_G(P,e)=N_G(P,\overline{e})$.
	
	Consider the product
	\begin{equation*}
		\prod_{(P,e)\in\BP_\OO(B)}R_\K(I_{(P,e)}/P,e).
	\end{equation*}
	If $(\chi_{(P,e)})$ is a tuple in this product and $g\in G$ define ${}^g(\chi_{(P,e)})=({}^g\chi_{{}^{g^{-1}}(P,e)})$. This defines an action of $G$ on $\prod_{(P,e)\in\BP_\OO(B)}R_\K(I_{(P,e)}/P,e)$ via group automorphisms. The subgroup $\left(\prod_{(P,e)\in\BP_\OO(B)}R_\K(I_{(P,e)}/P,e)\right)^G$ of $G$-fixed points consists of all tuples $(\chi_{(P,e)})$ such that ${}^g\chi_{(P,e)}=\chi_{{}^g(P,e)}$ for all $g\in G$ and all $(P,e)\in\BP_\OO(B)$.
	
	Define a map
	\begin{align*}
		\rho_B:=\rho:\prod_{P\in S_p(G)}R_\K(N_G(P)/P,\br_P(e_B))&\to \prod_{(P,e)\in\BP_\OO(B)}R_\K(I_{(P,e)}/P,e)\\
		(\chi_P)_{P\in S_p(G)}	&\mapsto (e\cdot\Res_{I_{(P,e)}}^{N_G(P)}(\chi_P))_{(P,e)\in\BP_\OO(B)}.
	\end{align*}
	Note that $\rho$ is a $G$-equivariant group homomorphism. It is also injective --- this follows from the following
	
	%need reference for lemma below.
	
	\begin{lemma}\label{lem:fong}
		Let $(Q,f)\in\BP_\OO(G)$, set $I:=N_G(Q,f)$ and set $e:=\tr_I^{N_G(Q)}(f)\in\bli(\OO N_G(Q))$. The functor $f\cdot\Res_I^{N_G(Q)}:\lmod{\OO N_G(Q)e}\to\lmod{\OO If}$ is an equivalence of categories, with inverse $\Ind_I^{N_G(Q)}:\lmod{\OO If}\to\lmod{\OO N_G(Q)e}$. In particular, the map $f\cdot\Res_I^{N_G(Q)}:R_\K(N_G(Q),e)\to R_\K(I,f)$ is a group isomorphism with inverse $\Ind_I^{N_G(Q)}$.
	\end{lemma}

	\begin{lemma}\label{lem:iso1}
		The map $\rho_B$ defined above is injective, and restricts to a group isomorphism
		\begin{equation*}
			\rho_B:\left(\prod_{P\in S_p(G)}R_\K(N_G(P)/P,\br_P(e_B))\right)^G\isoto\left(\prod_{(P,e)\in\BP_\OO(B)}R_\K(I_{(P,e)}/P,e)\right)^G.
		\end{equation*}
	\end{lemma}

	\begin{proof}
		We first show that $\rho=\rho_B$ is injective. Let $(\chi_P)_{P\in S_p(G)}\in\ker(\rho)$. Then for all $(P,e)\in\BP_\OO(B)$ we have $e\cdot\Res_{I_{(P,e)}}^{N_G(P)}(\chi_P)=0$. Let $P\in S_p(G)$. We must show that $\chi_P=0$. This is clear if $\br_P(e_B)=0$, so assume $\br_P(e_B)\neq 0$. Write $e(P)$ for the unique central idempotent of $\OO N_G(P)$ such that $\overline{e(P)}=\br_P(e_B)$. By Lemma \ref{lem:blockcent2} we may write
		\begin{equation*}
			e(P)=\sum_{i=1}^n\tr_{I_{(P,e_i)}}^{N_G(P)}(e_i)
		\end{equation*}
		for some blocks $e_i\in\bli(\OO C_G(P))$. Note that for each $i$ in the range $1\leq i\leq n$ we have $e(P)e_i=e_i$, hence $\br_P(e_B)\overline{e_i}=\overline{e_i}$ and $(P,e_i)\in\BP_\OO(B)$. Now by Lemma \ref{lem:fong} we have an isomorphism
		\begin{align*}
			R_\K(N_G(P),\br_P(e_B))	=&\bigoplus_{i=1}^n R_\K(N_G(P),\tr_{I_{(P,e_i)}}^{N_G(P)}(e_i))\\
				\underset{\dsum e_i\cdot\Res_{I_{(P,e_i)}}^{N_G(P)}}{\isoto}&\bigoplus_{i=1}^n R_\K(I_{(P,e_i)},e_i)
		\end{align*}
		and the image of $\chi_P\in R_\K(N_G(P),\br_P(e_B))$ under this isomorphism is $\sum_{i=1}^ne_i\cdot\Res_{I_{(P,e_i)}}^{N_G(P)}(\chi_P)=0$. Therefore $\chi_P=0$. This shows that $\rho$ is injective.
		
		Since $\rho$ is $G$-equivariant we have that 
		\begin{equation*}
			\rho\left(\left(\prod_{P\in S_p(G)}R_\K(N_G(P)/P,\br_P(e_B))\right)^G\right)\subseteq \left(\prod_{(P,e)\in\BP_\OO(B)}R_\K(I_{(P,e)}/P,e)\right)^G.
		\end{equation*}
		To see that the reverse containment also holds, let $(\chi_{(P,e)})$ be a $G$-invariant tuple in $\prod_{(P,e)\in\BP_\OO(B)}R_\K(I_{(P,e)}/P,e)$. Define a tuple $(\psi_P)\in \prod_{P\in S_p(G)}R_\K(N_G(P)/P,\br_P(e_B))$ as follows: let $P\in S_p(G)$. If $\br_P(e_B)=0$ set $\psi_P=0$. Otherwise, let $e(P)$ denote the lift of $\br_P(e_B)$ to a central idempotent of $\OO N_G(P)$ and choose blocks $e_1,\ldots,e_n\in\bli(\OO C_G(P))$ such that $e(P)=\sum_{i=1}^n\tr_{I_{(P,e_i)}}^{N_G(P)}(e_i)$. Then $(P,e_i)\in\BP_\OO(B)$ for each $1\leq i\leq n$. Note that
		\begin{equation*}
			\Ind_{I_{(P,e_i)}}^{N_G(P)}(\chi_{(P,e_i)})\in R_\K(N_G(P)/P,\tr_{I_{(P,e_i)}}^{N_G(P)}(e_i)).
		\end{equation*}
		Set
		\begin{equation*}
			\psi_P:=\sum_{i=1}^n\Ind_{I_{(P,e_i)}}^{N_G(P)}(\chi_{(P,e_i)})\in R_\K(N_G(P)/P,\br_P(e_B)).
		\end{equation*}
		Since the tuple $(\chi_{(P,e)})$ is $G$-fixed the definition of $\psi_P$ does not depend on the choice of blocks $e_1,\ldots,e_n$. It follows also that $(\psi_P)$ is a $G$-fixed tuple in $\prod_{P\in S_p(G)}R_\K(N_G(P)/P,\br_P(e_B))$. So to complete the proof it remains to show that $\rho(\psi_P)=(\chi_{(P,e)})$, i.e., that for all $(P,e)\in\BP_\OO(B)$ we have
		\begin{equation*}
			e\Res_{I_{(P,e)}}^{N_G(P)}(\psi_P)=\chi_{(P,e)}.
		\end{equation*}
		Let $(P,e)\in\BP_\OO(B)$. Then $\br_P(e_B)\neq 0$. Let $e(P),e_1,\ldots,e_n$ be as above. Without loss of generality we may assume that $e=e_1$. Then
		\begin{align*}
			e\Res_{I_{(P,e)}}^{N_G(P)}(\psi_P)	&=e_1\Res_{I_{(P,e_1)}}^{N_G(P)}\left(\sum_{i=1}^n\Ind_{I_{(P,e_i)}}^{N_G(P)}(\chi_{(P,e_i)})\right)\\
				&=e_1\cdot\sum_{i=1}^{n}\Res_{I_{(P,e_1)}}^{N_G(P)}(\Ind_{I_{(P,e_i)}}^{N_G(P)}(\chi_{(P,e_i)}))\\
				&=e_1\Res_{I_{(P,e_1)}}^{N_G(P)}(\Ind_{I_{(P,e_1)}}^{N_G(P)}(\chi_{(P,e_1)}))\\
				&=\chi_{(P,e)}
		\end{align*}
		where the last equality holds by Lemma \ref{lem:fong}. The proof is complete.
	\end{proof}
	
	Now let 
	\begin{equation*}
		\alpha_B=\alpha:T_\OO(B)\to \left(\prod_{(P,e)\in\BP_\OO(B)}R_\K(I_{(P,e)}/P,e)\right)^G
	\end{equation*} 
	denote the composite of the maps $\beta$ and $\rho$ --- in other words, $\alpha$ is the homomorphism making the diagram below commute:
	% https://q.uiver.app/?q=WzAsMyxbMCwxLCJUX1xcT08oQikiXSxbMCwwLCJcXGxlZnQoXFxwcm9kX3tQXFxpbiBTX3AoRyl9Ul9cXEsoTl9HKFApL1AsXFxicl9QKGVfQikpXFxyaWdodCleRyJdLFsxLDAsIlxcbGVmdChcXHByb2RfeyhQLGUpXFxpblxcQlBfXFxPTyhCKX1SX1xcSyhJX3soUCxlKX0vUCxlKVxccmlnaHQpXkciXSxbMCwxLCJcXGJldGEiXSxbMSwyLCJcXHJobyJdLFswLDIsIlxcYWxwaGEiLDIseyJzdHlsZSI6eyJib2R5Ijp7Im5hbWUiOiJkYXNoZWQifX19XV0=
	\[\begin{tikzcd}
		{\left(\prod_{P\in S_p(G)}R_\K(N_G(P)/P,\br_P(e_B))\right)^G} & {\left(\prod_{(P,e)\in\BP_\OO(B)}R_\K(I_{(P,e)}/P,e)\right)^G} \\
		{T_\OO(B)}
		\arrow["\beta", from=2-1, to=1-1]
		\arrow["\rho", from=1-1, to=1-2]
		\arrow["\alpha"', dashed, from=2-1, to=1-2]
	\end{tikzcd}\]
	Note that if $M\in\ltriv{B}$ then 
	\begin{equation*}
		\alpha([M])=(\chi_{M(P,e)})_{(P,e)\in\BP_\OO(B)}
	\end{equation*}
	where $\chi_{M(P,e)}$ denotes the character of the trivial source $\OO I_{(P,e)}e$-module $M(P,e)=e\Res_{I_{(P,e)}}^{N_G(P)}M(P)$. 

	%Robert's note: rewrite this theorem by saying let (chi_) be a G-fixed tuple. Then (chi) is in the image if and only if C1 if and only if C2.

	\begin{theorem}\label{thm:coherence3}
		Let $B\in\Bl(\OO G)$. The image of $T_\OO(B)$ under the map
		\begin{align*}
			\alpha_B=\alpha:T_\OO(B)&\into \left(\prod_{(P,e)\in\BP_\OO(B)}R_\K(I_{(P,e)}/P,e)\right)^G\\
			[M]&\mapsto(\chi_{M(P,e)})_{(P,e)\in\BP_\OO(B)}\qquad\qquad M\in\ltriv{B}
		\end{align*}
		is equal to the subgroup of tuples $(\chi_{(P,e)})\in \left(\prod_{(P,e)\in\BP_\OO(B)}R_\K(I_{(P,e)}/P,e)\right)^G$ that satisfy
		\begin{equation}\label{C1}\tag{C1}
			\left\{
			\begin{gathered}
			\chi_{(P,e)}(us\epsilon)=\Ind_{C_{I_{(P,e)}\cap I_{(P\gp{u},f)}}(u)}^{C_{I_{(P,e)}}(u)}(\Res_{C_{I_{(P,e)}\cap I_{(P\gp{u},f)}}(u)}^{I_{(P\gp{u},f)}}(\chi_{(P\gp{u},f)}))(s)\\
			\text{for all }(P,e)\in\BP_\OO(B)\text{, }(u,\epsilon)\in\BE_\OO(I_{(P,e)},e)\text{, }s\in C_{I_{(P,e)}}(u)_{p'}\text{,}\\
			\text{and all }f\in\bli(\OO C_G(P\gp{u}))\text{ such that }\epsilon\cdot f\neq 0.
			\end{gathered}
			\right.
		\end{equation}
		The image of $T_\OO(B)$ under $\alpha$ is also equal to the subgroup of tuples $(\chi_{(P,e)})\in \left(\prod_{(P,e)\in\BP_\OO(B)}R_\K(I_{(P,e)}/P,e)\right)^G$ that satisfy
		\begin{equation}\label{C2}\tag{C2}
			\left\{
			\begin{gathered}
				\chi_{(P,e)}(us)=\sum_{\substack{f\in\bli(\OO C_G(P\gp{u}))\\(P,e)\nor(P\gp{u},f)\\s\in I_{(P\gp{u},f)}}}\chi_{(P\gp{u},f)}(s)\\
				\text{for all }(P,e)\in\BP_\OO(B)\text{, all }p\text{-elements }u\in I_{(P,e)}\text{,}\\
				\text{and all }s\in C_{I_{(P,e)}}(u)_{p'}.
			\end{gathered}
			\right.
		\end{equation}
	\end{theorem}

	\begin{proof}
		Before beginning the proof, we note that if $(P,e)\in\BP_\OO(B)$ and $(u,\epsilon)\in\BE_\OO(I_{(P,e)},e)$ then by Lemma \ref{lem:brauereltnormalizer} there exists a block idempotent $f\in\bli(\OO C_G(P\gp{u}))$ satisfying $\epsilon f\neq 0$, so Condition (\ref{C1}) makes sense. In Condition (\ref{C2}) the sum on the right is equal to 0 if no such block idempotent $f$ exists.
		
		For convenience, let $A_1$ and $A_2$ denote the collections of tuples $(\chi_{(P,e)})\in \left(\prod_{(P,e)\in\BP_\OO(B)}R_\K(I_{(P,e)}/P,e)\right)^G$ that satisfy Conditions (\ref{C1}) and (\ref{C2}), respectively. Observe that $A_1$ and $A_2$ are subgroups of $\left(\prod_{(P,e)\in\BP_\OO(B)}R_\K(I_{(P,e)}/P,e)\right)^G$. We must show that
		\begin{equation*}
			\alpha(T_\OO(B))=A_1=A_2.
		\end{equation*}
	
		We first show that $\alpha(T_\OO(B))\subseteq A_1$. Since $A_1$ is a subgroup it suffices to check that $\alpha([M])\in A_1$ for all $M\in\ltriv{B}$. So let $M$ be a trivial source $B$-module, let $(P,e)\in\BP_\OO(B)$, $(u,\epsilon)\in\BE_\OO(I_{(P,e)},e)$, and let $f\in\bli(\OO C_G(P\gp{u}))$ be such that $\epsilon\cdot f\neq 0$. Then by Lemma \ref{lem:brauereltnormalizer} we have $(P\gp{u},f)\in\BP_\OO(B)$, $(P,e)\nor(P\gp{u},f)$, and
		\begin{equation*}
			\epsilon=\tr_{C_{I_{(P,e)}\cap I_{(P\gp{u},f)}}(u)}^{C_{I_{(P,e)}}(u)}(f).
		\end{equation*}
		We must show that 
		\begin{equation*}
			\chi_{M(P,e)}(us\epsilon)=\Ind_{C_{I_{(P,e)}\cap I_{(P\gp{u},f)}}(u)}^{C_{I_{(P,e)}}(u)}(\Res_{C_{I_{(P,e)}\cap I_{(P\gp{u},f)}}(u)}^{I_{(P\gp{u},f)}}(\chi_{M(P\gp{u},f)}))(s)
		\end{equation*}
		for all $s\in C_{I_{(P,e)}}(u)_{p'}$.
		
		Now $M(P,e)$ is a trivial source $\OO I_{(P,e)}e$-module, so we may apply the map $\beta_{I_{(P,e)}}$ of Theorem \ref{thm:coherence2} to obtain a coherent character tuple
		\begin{equation*}
			\beta_{I_{(P,e)}}([M(P,e)])=(\chi_{(M(P,e))(Q)})\in\prod_{Q\in S_p(I_{(P,e)})}R_\K(N_{I_{(P,e)}}(Q)/Q,\br_Q^{I_{(P,e)}}(e)).
		\end{equation*}
		Notice that $\chi_{(M(P,e))(1)}=\chi_{M(P,e)}$. The coherence condition of Theorem \ref{thm:coherence2} then gives that
		\begin{equation*}
			\chi_{M(P,e)}(us)=\chi_{(M(P,e))(\gp{u})}(s)
		\end{equation*}
		for any $s\in C_{I_{(P,e)}}(u)_{p'}$. Now by Lemma \ref{lem:brauerconst} there is an isomorphism of $\OO N_{I_{(P,e)}}(\gp{u})$-modules
		\begin{equation*}
			(M(P,e))(\gp{u})\iso e(\gp{u})\cdot\Res_{N_{I_{(P,e)}}(\gp{u})}^{N_G(P\gp{u})}(M(P\gp{u}))
		\end{equation*}
		where $e(\gp{u})$ is the unique central idempotent of $\OO N_{I_{(P,e)}}(\gp{u})$ satisfying $\overline{e(\gp{u})}=\br_{\gp{u}}^{I_{(P,e)}}(e)$. Since $e(\gp{u})\in\OO C_{I_{(P,e)}}(u)$ we have
		\begin{align*}
			\chi_{M(P,e)}(us)	&=\chi_{(M(P,e))(\gp{u})}(s)\\
			&=\chi_{e(\gp{u})\cdot\Res_{N_{I_{(P,e)}}(\gp{u})}^{N_G(P\gp{u})}(M(P\gp{u}))}(s)\\
			&=\chi_{e(\gp{u})\cdot\Res_{C_{I_{(P,e)}}(u)}^{N_G(P\gp{u})}(M(P\gp{u}))}(s)
		\end{align*}
		for any $s\in C_{I_{(P,e)}}(u)_{p'}$. It follows that
		\begin{equation*}
			d_{I_{(P,e)}}^u(\chi_{M(P,e)})=d_{C_{I_{(P,e)}}(u)}^1(\chi_{e(\gp{u})\cdot\Res_{C_{I_{(P,e)}}(u)}^{N_G(P\gp{u})}(M(P\gp{u}))}),
		\end{equation*}
		which is an equality of class functions in $CF_{p'}(C_{I_{(P,e)}}(u);\K)$. Multiplying both sides of this equality by $\epsilon$ we obtain
		\begin{align*}
			d_{I_{(P,e)}}^{u,\epsilon}(\chi_{M(P,e)})	&=d_{C_{I_{(P,e)}}(u)}^{1,\epsilon}(\chi_{e(\gp{u})\cdot\Res_{C_{I_{(P,e)}}(u)}^{N_G(P\gp{u})}(M(P\gp{u}))})\\
			&=d_{C_{I_{(P,e)}}(u)}^{1}(\epsilon\cdot \chi_{e(\gp{u})\cdot\Res_{C_{I_{(P,e)}}(u)}^{N_G(P\gp{u})}(M(P\gp{u}))})\\
			&=d_{C_{I_{(P,e)}}(u)}^{1}(\chi_{\epsilon\cdot (e(\gp{u})\cdot\Res_{C_{I_{(P,e)}}(u)}^{N_G(P\gp{u})}(M(P\gp{u})))})\\
			&=d_{C_{I_{(P,e)}}(u)}^{1}(\chi_{\epsilon\cdot\Res_{C_{I_{(P,e)}}(u)}^{N_G(P\gp{u})}(M(P\gp{u}))})
		\end{align*}
		where the last equality holds because $(u,\epsilon)\in\BE_\OO(I_{(P,e)},e)$, so that $\epsilon\cdot e(\gp{u})=\epsilon$. Since $\epsilon=\tr_{C_{I_{(P,e)}\cap I_{(P\gp{u},f)}}(u)}^{C_{I_{(P,e)}}(u)}(f)$ we have a decomposition
		\begin{align*}
			\Res_{C_G(P\gp{u})}^{C_{I_{(P,e)}}(u)}&(\epsilon\cdot\Res_{C_{I_{(P,e)}}(u)}^{N_G(P\gp{u})}(M(P\gp{u})))\\
			&=\bigoplus_{x\in[C_{I_{(P,e)}}(u)/C_{I_{(P,e)}\cap I_{(P\gp{u},f)}}(u)]} {}^xf\cdot\Res_{C_G(P\gp{u})}^{N_G(P\gp{u})}(M(P\gp{u})).
		\end{align*}
		The summands are permuted transitively by $C_{I_{(P,e)}}(u)$, and the stabilizer of the $f$-component is $C_{I_{(P,e)}\cap I_{(P\gp{u},f)}}(u)$. It follows that
		\begin{align*}
			\epsilon\cdot\Res_{C_{I_{(P,e)}}(u)}^{N_G(P\gp{u})}&(M(P\gp{u}))\\
			&\iso\Ind_{C_{I_{(P,e)}\cap I_{(P\gp{u},f)}}(u)}^{C_{I_{(P,e)}}(u)}(\Res_{C_{I_{(P,e)}\cap I_{(P\gp{u},f)}}(u)}^{I_{(P\gp{u},f)}}(M(P\gp{u},f))).
		\end{align*}
		So now we have
		\begin{align*}
			d_{I_{(P,e)}}^{u,\epsilon}(\chi_{M(P,e)})	&=d_{C_{I_{(P,e)}}(u)}^{1}(\chi_{\epsilon\cdot\Res_{C_{I_{(P,e)}}(u)}^{N_G(P\gp{u})}(M(P\gp{u}))})\\
			&=d_{C_{I_{(P,e)}}(u)}^{1}(\Ind_{C_{I_{(P,e)}\cap I_{(P\gp{u},f)}}(u)}^{C_{I_{(P,e)}}(u)}(\Res_{C_{I_{(P,e)}\cap I_{(P\gp{u},f)}}(u)}^{I_{(P\gp{u},f)}}(\chi_{M(P\gp{u},f)}))).
		\end{align*}
		In particular, by Proposition \ref{prop:gendecompvalues} we obtain
		\begin{equation*}
			\chi_{M(P,e)}(us\epsilon)=\Ind_{C_{I_{(P,e)}\cap I_{(P\gp{u},f)}}(u)}^{C_{I_{(P,e)}}(u)}(\Res_{C_{I_{(P,e)}\cap I_{(P\gp{u},f)}}(u)}^{I_{(P\gp{u},f)}}(\chi_{M(P\gp{u},f)}))(s)
		\end{equation*}
		for any $s\in C_{I_{(P,e)}}(u)_{p'}$, as desired.
		
		We have shown that $\alpha(T_\OO(B))\subseteq A_1$. Next we show that $A_1\subseteq A_2$. Let $(\chi_{(P,e)})\in A_1$. Let $(P,e)\in\BP_\OO(B)$, $u\in I_{(P,e)}$ a $p$-element, and let $s\in C_{I_{(P,e)}}(u)_{p'}$. By Brauer's 2nd Main Theorem we have
		\begin{equation*}
			\chi_{(P,e)}(us)=d_{I_{(P,e)}}^u(\chi_{(P,e)})(s)=\sum_{\substack{\epsilon\in\bli(\OO C_{I_{(P,e)}}(u))\\(u,\epsilon)\in\BE_\OO(I_{(P,e)},e)}}d_{I_{(P,e)}}^{u,\epsilon}(\chi_{(P,e)})(s).
		\end{equation*}
		For each $\epsilon\in\bli(\OO C_{I_{(P,e)}}(u))$ such that $(u,\epsilon)\in\BE_\OO(I_{(P,e)},e)$ choose a block $f_\epsilon\in\bli(\OO C_G(P\gp{u}))$ such that $\epsilon\cdot f_\epsilon\neq 0$. Note that $f_\epsilon$ exists for each $\epsilon$ by Lemma \ref{lem:brauereltnormalizer} and that $\Stab_{C_{I_{(P,e)}}(u)}(f_\epsilon)=C_{I_{(P,e)}\cap I_{(P\gp{u},f_\epsilon)}}(u)$. Then for each $(u,\epsilon)\in\BE_\OO(I_{(P,e)},e)$ we have
		\begin{align*}
			\chi_{(P,e)}(us\epsilon)	&=\Ind_{C_{I_{(P,e)}\cap I_{(P\gp{u},f_\epsilon)}}(u)}^{C_{I_{(P,e)}}(u)}(\Res_{C_{I_{(P,e)}\cap I_{(P\gp{u},f_\epsilon)}}(u)}^{I_{(P\gp{u},f_\epsilon)}}(\chi_{(P\gp{u},f_\epsilon)}))(s)\\
			&=\sum_{\substack{g\in[C_{I_{(P,e)}}(u)/\Stab_{C_{I_{(P,e)}}(u)}(f_\epsilon)]\\s^g\in\Stab_{C_{I_{(P,e)}}(u)}(f_\epsilon) }}\chi_{(P\gp{u},f_\epsilon)}(s^g)\\
			&=\sum_{\substack{g\in[C_{I_{(P,e)}}(u)/\Stab_{C_{I_{(P,e)}}(u)}(f_\epsilon)]\\s\in\Stab_{C_{I_{(P,e)}}(u)}({}^gf_\epsilon) }}\chi_{(P\gp{u},{}^gf_\epsilon)}(s)\\
			&=\sum_{\substack{f\in\Orb_{C_{I_{(P,e)}}(u)}(f_\epsilon)\\s\in\Stab_{C_{I_{(P,e)}}(u)}(f)}}\chi_{(P\gp{u},f)}(s)\\
			&=\sum_{\substack{f\in\bli(\OO C_G(P\gp{u})\\f\cdot\epsilon\neq 0\\s\in I_{(P\gp{u},f)}}}\chi_{(P\gp{u},f)}(s).
		\end{align*}
		It follows that
		\begin{align*}
			\chi_{(P,e)}(us)	&=\sum_{\substack{\epsilon\in\bli(\OO C_{I_{(P,e)}}(u))\\(u,\epsilon)\in\BE_\OO(I_{(P,e)},e)}}\chi_{(P,e)}(us\epsilon)\\
			&=\sum_{\substack{\epsilon\in\bli(\OO C_{I_{(P,e)}}(u))\\(u,\epsilon)\in\BE_\OO(I_{(P,e)},e)}}\sum_{\substack{f\in\bli(\OO C_G(P\gp{u})\\f\cdot\epsilon\neq 0\\s\in I_{(P\gp{u},f)}}}\chi_{(P\gp{u},f)}(s)\\
			&=\sum_{\substack{f\in\bli(\OO C_G(P\gp{u}))\\(P,e)\nor(P\gp{u},f)\\s\in I_{(P\gp{u},f)}}}\chi_{(P\gp{u},f)}(s)
		\end{align*}
		where the last line follows from the bijection of Lemma \ref{lem:bijection1}. Thus we see that $(\chi_{(P,e)})\in A_2$, and that $A_1\subseteq A_2$.
		
		To complete the proof we must show that $A_2\subseteq\alpha(T_\OO(B))$. Let $(\chi_{(P,e)})\in A_2$. Recall the isomorphism $\rho$ of Lemma \ref{lem:iso1}. Write 
		\begin{equation*}
		(\psi_P)_{P\in S_p(G)}:=\rho^{-1}((\chi_{(P,e)}))\in\left(\prod_{P\in S_p(G)}R_\K(N_G(P)/P,\br_P(e_B))\right)^G.
		\end{equation*}
		I claim that $(\psi_P)\in\beta(T_\OO(B))$. Note that if the claim is correct then it will follow that $(\chi_{(P,e)})\in \rho(\beta(T_\OO(B)))=\alpha(T_\OO(B))$ and the proof will be complete.
		
		By Theorem \ref{thm:coherence2}, the tuple $(\psi_P)$ belongs to $\beta(T_\OO(B))$ if and only if 
		\begin{equation*}
			\psi_P(x)=\psi_{P\gp{x_p}}(x)
		\end{equation*}
		for all $P\in S_p(G)$ and $x\in N_G(P)$, or equivalently,
		\begin{equation*}
			\psi_P(us)=\psi_{P\gp{u}}(s)
		\end{equation*}
		for all $P\in S_p(G)$, all $p$-elements $u\in N_G(P)$, and all $s\in C_{N_G(P)}(u)_{p'}$. Recall from the proof of Lemma \ref{lem:iso1} how the characters $\psi_P$, $P\in S_p(G)$, are defined: if $\br_P(e_B)=0$ then $\psi_P=0$. If $\br_P(e_B)\neq 0$ let $e(P)$ denote the lift of $\br_P(e_B)$ to a central idempotent of $\OO N_G(P)$ and choose blocks $e_1,\ldots, e_n\in\bli(\OO C_G(P))$ such that $e(P)=\sum_{i=1}^n\tr_{I_{(P,e_i)}}^{N_G(P)}(e_i)$. Then
		\begin{equation*}
			\psi_P=\sum_{i=1}^n\Ind_{I_{(P,e_i)}}^{N_G(P)}(\chi_{(P,e_i)}).
		\end{equation*}
		In this case, the value $\psi_P(us)$ for $u\in N_G(P)$ a $p$-element and $s\in C_{N_G(P)}(u)_{p'}$ can be computed:
		\begin{align*}
			\psi_P(us)	&=\sum_{i=1}^n\Ind_{I_{(P,e_i)}}^{N_G(P)}(\chi_{(P,e_i)})(us)\\
			&=\sum_{i=1}^n\sum_{\substack{g\in N_G(P)/I_{(P,e_i)}\\(us)^g\in I_{(P,e_i)}}}\chi_{(P,e_i)}((us)^g)\\
			&=\sum_{i=1}^n\sum_{\substack{g\in N_G(P)/I_{(P,e_i)}\\ us\in I_{(P,{}^ge_i)}}}\chi_{(P,{}^ge_i)}(us)\\
			&=\sum_{\substack{e\in\bli(\OO C_G(P))\\(P,e)\in\BP_\OO(B)\\us\in I_{(P,e)}}}\chi_{(P,e)}(us).
		\end{align*}
	
		Note that the formula
		\begin{equation*}
			\psi_P(us)=\sum_{\substack{e\in\bli(\OO C_G(P))\\(P,e)\in\BP_\OO(B)\\us\in I_{(P,e)}}}\chi_{(P,e)}(us)
		\end{equation*}
		holds for each $P\in S_p(G)$, $u\in N_G(P)$ a $p$-element, and $s\in C_{N_G(P)}(u)_{p'}$ independent of whether $\br_P(e_B)=0$ or $\br_P(e_B)\neq 0$, for in the former case $\psi_P(us)=0$ and there do not exist $B$-Brauer pairs of the form $(P,e)$ so the sum on the right is 0.
		
		Now let $P\in S_p(G)$, $u\in N_G(P)$ a $p$-element, and let $s\in C_{N_G(P)}(u)_{p'}$. If $us\in I_{(P,e)}$ for a Brauer pair $(P,e)$ then $u\in I_{(P,e)}$ and $s\in C_{I_{(P,e)}}(u)_{p'}$, so by Condition (\ref{C2})
		\begin{equation*}
			\psi_P(us)=\sum_{\substack{e\in\bli(\OO C_G(P))\\(P,e)\in\BP_\OO(B)\\us\in I_{(P,e)}}}\sum_{\substack{f\in\bli(\OO C_G(P\gp{u}))\\(P,e)\nor(P\gp{u},f)\\s\in I_{(P\gp{u},f)}}}\chi_{(P\gp{u},f)}(s).
		\end{equation*}
		The sum above can be reindexed after making the following observations: let $\mathcal{I}$ denote the set of ordered pairs $(e,f)$ where $e\in\bli(\OO C_G(P))$ such that $(P,e)\in\BP_\OO(B)$ and $f\in\bli(\OO C_G(P\gp{u}))$ such that $(P,e)\nor(P\gp{u},f)$. Let $\mathcal{J}$ denote the set of block idempotents $f\in\bli(\OO C_G(P\gp{u}))$ such that $(P\gp{u},f)\in\BP_\OO(B)$. The map $\mathcal{I}\to\mathcal{J}$, $(e,f)\mapsto f$ is well-defined and is a bijection by the existence and uniqueness of Brauer pairs contained within a fixed Brauer pair (see \cite[Part IV, Theorem 2.10]{Aschbacher_2011}). Let $\mathcal{I}'$ denote the subset of $(e,f)\in\mathcal{I}$ for which $us\in I_{(P,e)}$ and $s\in I_{(P\gp{u},f)}$ and let $\mathcal{J}'$ denote the subset of $f\in\mathcal{J}$ such that $s\in I_{(P\gp{u},f)}$. The bijection $\mathcal{I}\isoto\mathcal{J}$ clearly maps $\mathcal{I}'$ into $\mathcal{J}'$. In fact the image of $\mathcal{I}'$ is precisely $\mathcal{J}'$, for if $f\in\mathcal{J}'$ and if $e$ is the unique block of $\OO C_G(P)$ such that $(P,e)\nor (P\gp{u},f)$ then $u,s\in N_G(P)\cap I_{(P\gp{u},f)}\subgp I_{(P,e)}$, hence $us\in I_{(P,e)}$ and $(e,f)\in\mathcal{I}'$. Now the sum above is indexed by the elements of $\mathcal{I}'$. Reindexing by the elements of $\mathcal{J}'$ gives
		\begin{equation*}
			\psi_P(us)=\sum_{\substack{f\in\bli(\OO C_G(P\gp{u}))\\(P\gp{u},f)\in\BP_\OO(B)\\s\in I_{(P\gp{u},f)}}}\chi_{(P\gp{u},f)}(s).
		\end{equation*}
		But this is precisely the formula for $\psi_{P\gp{u}}(s)$ given previously. We conclude that $\psi_P(us)=\psi_{P\gp{u}}(s)$ and that $(\psi_P)\in\beta(T_\OO(B))$, proving the claim and completing the proof of the theorem.
	\end{proof}
	
	\begin{corollary}
		Let $B\in\Bl(\OO G)$ and let $(\chi_{(P,e)})\in\alpha(T_\OO(B))$. For each $(P,e)\in\BP_\OO(B)$ set
		\begin{equation*}
			\psi_{(P,e)}:=\Res_{C_G(P)}^{I_{(P,e)}}(\chi_{(P,e)})\in R_\K(C_G(P),e).
		\end{equation*}
		Let $(P,e)\in\BP_\OO(B)$ and let $(u,f)\in\BE_\OO(C_G(P),e)$. Then $(P\gp{u},f)\in\BP_\OO(B)$ and
		\begin{equation*}
			\psi_{(P,e)}(usf)=\psi_{(P\gp{u},f)}(s)
		\end{equation*}
		for all $s\in C_G(P\gp{u})_{p'}$.
	\end{corollary}

	\begin{proof}
		Since $(u,f)$ is a Brauer element for $\OO C_G(P)$ we have by definition that $u$ is a $p$-element of $C_G(P)$ and $f$ is a block idempotent of $\OO C_{C_G(P)}(u)=\OO C_G(P\gp{u})$, so $(P\gp{u},f)$ is a Brauer pair of $\OO G$. Since $(u,f)$ belongs to $e$ we have $\br_{\gp{u}}^{C_G(P)}(e)\overline{f}=\overline{f}$. Now $\br_{\gp{u}}^{C_G(P)}(e)=\br_{P\gp{u}}^G(e)$, so we see that $(P,e)\nor(P\gp{u},f)$ and in particular $(P\gp{u},f)\in\BP_\OO(B)$.
		
		If $s\in C_G(P\gp{u})_{p'}$ then Condition (\ref{C2}) of Theorem \ref{thm:coherence3} gives
		\begin{equation*}
			\psi_{(P,e)}(us)=\sum_{\substack{f'\in\bli(\OO C_G(P\gp{u}))\\(P,e)\nor(P\gp{u},f')\\s\in I_{(P\gp{u},f')}}}\chi_{(P\gp{u},f')}(s)=\sum_{\substack{f'\in\bli(\OO C_G(P\gp{u}))\\(P,e)\nor(P\gp{u},f')}}\psi_{(P\gp{u},f')}(s).
		\end{equation*}
		It follows that
		\begin{equation*}
			d_{C_G(P)}^u(\psi_{(P,e)})=\sum_{\substack{f'\in\bli(\OO C_G(P\gp{u}))\\(P,e)\nor(P\gp{u},f')}}d_{C_G(P\gp{u})}^1(\psi_{(P\gp{u},f')}).
		\end{equation*}
		Projecting onto the $f$-component of this sum, we obtain
		\begin{equation*}
			d_{C_G(P)}^{u,f}(\psi_{(P,e)})=d_{C_G(P\gp{u})}^1(\psi_{(P\gp{u},f)})
		\end{equation*}
		and therefore $\psi_{(P,e)}(usf)=\psi_{(P\gp{u},f)}(s)$ for all $s\in C_G(P\gp{u})_{p'}$, as desired.
	\end{proof}
	
	Now let $B\in\Bl(\OO G)$ and let $(D,e_D)\in\BP_\OO(B)$ be a maximal $B$-Brauer pair. If $P\subgp D$ write $e_P$ for the unique block idempotent of $\OO C_G(P)$ such that $(P,e_P)\leq(D,e_D)$. Set $\mathcal{F}=\mathcal{F}_{(D,e_D)}(G,B)$ and for each subgroup $P\subgp D$ set
	\begin{equation*}
		I_P:=I_{(P,e_P)}=N_G(P,e_P).
	\end{equation*}
	
	Suppose that $P,Q\subgp D$ and that $\varphi:P\isoto Q$ is an $\mathcal{F}$-isomorphism. Let $g\in G$ be such that $\varphi=c_g$ and ${}^g(P,e_P)=(Q,e_Q)$. Then ${}^gI_P=I_Q$ and conjugation by $g$ induces a group isomorphism
	\begin{equation*}
		{}^g(\cdot):R_\K(I_P/P,e_P)\isoto R_\K(I_Q/Q,e_Q).
	\end{equation*}
	If $h\in G$ is another element such that $\varphi=c_h$ and ${}^h(P,e_P)=(Q,e_Q)$ then $h^{-1}g\in I_P$, hence ${}^g(\cdot)={}^h(\cdot)$ as maps $R_\K(I_P/P,e_P)\to R_\K(I_Q/Q,e_Q)$. In light of this, we obtain a well-defined group isomorphism
	\begin{align*}
		{}^\varphi(\cdot):R_\K(I_P/P,e_P)&\isoto R_\K(I_Q/Q,e_Q)\\
							\chi&\mapsto {}^g\chi
	\end{align*}
	where $g\in G$ is any element such that $\varphi=c_g$ and ${}^g(P,e_P)=(Q,e_Q)$, and we call this map \textit{conjugation by }$\varphi$. Note that if $\varphi:P\isoto Q$ and $\psi:Q\isoto R$ are $\mathcal{F}$-isomorphisms and $\chi\in R_\K(I_P/P,e_P)$ then ${}^\psi({}^\varphi\chi)={}^{\psi\varphi}\chi$.
	
	Now consider the product $\prod_{P\subgp D}R_\K(I_P/P,e_P)$. Say a tuple $(\chi_P)_{P\subgp D}\in \prod_{P\subgp D}R_\K(I_P/P,e_P)$ is $\mathcal{F}$-\textit{fixed} if ${}^\varphi\chi_P=\chi_{\varphi(P)}$ for all $P\subgp D$ and all $\mathcal{F}$-isomorphisms $\varphi:P\isoto Q$. The subset $\left(\prod_{P\subgp D}R_\K(I_P/P,e_P)\right)^\mathcal{F}$ of $\mathcal{F}$-fixed tuples forms a subgroup.
	
	\begin{proposition}\label{prop:pi}
		The canonical projection
		\begin{equation*}
			\pi:\prod_{(P,e)\in\BP_\OO(B)}R_\K(I_{(P,e)}/P,e)\onto\prod_{P\subgp D}R_\K(I_P/P,e_P)
		\end{equation*}
		restricts to a group isomorphism
		\begin{equation*}
			\left(\prod_{(P,e)\in\BP_\OO(B)}R_\K(I_{(P,e)}/P,e)\right)^G\isoto \left(\prod_{P\subgp D}R_\K(I_P/P,e_P)\right)^\mathcal{F}.
		\end{equation*}
	\end{proposition}

	\begin{proof}
		Let $(\chi_{(P,e)})\in \left(\prod_{(P,e)\in\BP_\OO(B)}R_\K(I_{(P,e)}/P,e)\right)^G$. For each $P\subgp D$ set $\chi_P:=\chi_{(P,e_P)}$. So then $\pi((\chi_{(P,e)}))=(\chi_P)_{P\subgp D}$. Now let $P,Q\subgp D$ and let $\varphi:P\isoto Q$ be an $\mathcal{F}$-isomorphism. Say $g\in G$ is such that $\varphi=c_g$ and ${}^g(P,e_P)=(Q,e_Q)$. Then
		\begin{equation*}
			{}^\varphi\chi_P={}^g\chi_P={}^g\chi_{(P,e_P)}=\chi_{{}^g(P,e_P)}=\chi_{(Q,e_Q)}=\chi_Q=\chi_{\varphi(P)},
		\end{equation*}
		so $\pi$ restricts to a group homomorphism from $\left(\prod_{(P,e)\in\BP_\OO(B)}R_\K(I_{(P,e)}/P,e)\right)^G$ to $\left(\prod_{P\subgp D}R_\K(I_P/P,e_P)\right)^\mathcal{F}$. 
		
		Suppose that $(\chi_{(P,e)})\in \left(\prod_{(P,e)\in\BP_\OO(B)}R_\K(I_{(P,e)}/P,e)\right)^G$ is such that $\pi((\chi_{(P,e)}))=0$. Then $\chi_{(P,e_P)}=0$ for all $P\subgp D$. Now if $(P,e)$ is any $B$-Brauer pair then there exists some element $g\in G$ such that ${}^g(P,e)\leq(D,e_D)$. Then $0=\chi_{{}^g(P,e)}={}^g\chi_{(P,e)}$, and it follows that $\chi_{(P,e)}=0$. This shows that the restriction of $\pi$ to $\left(\prod_{(P,e)\in\BP_\OO(B)}R_\K(I_{(P,e)}/P,e)\right)^G$ is injective. 
		
		It remains to show that $\pi$ maps $\left(\prod_{(P,e)\in\BP_\OO(B)}R_\K(I_{(P,e)}/P,e)\right)^G$ onto $\left(\prod_{P\subgp D}R_\K(I_P/P,e_P)\right)^\mathcal{F}$. Let $(\chi_P)\in \left(\prod_{P\subgp D}R_\K(I_P/P,e_P)\right)^\mathcal{F}$. For each $(P,e)\in\BP_\OO(B)$ choose an element $g\in G$ such that ${}^g(P,e)\leq(D,e_D)$. Then set
		\begin{equation*}
			\chi_{(P,e)}:={}^{g^{-1}}\chi_{{}^gP}\in R_\K(I_{(P,e)}/P,e).
		\end{equation*}
		Note that the definition of $\chi_{(P,e)}$ does not depend on the choice of $g$: indeed, if $h\in G$ is another element such that ${}^h(P,e)\leq(D,e_D)$ then $\varphi=c_{gh^{-1}}:{}^hP\isoto {}^gP$ is an $\mathcal{F}$-isomorphism, hence ${}^{gh^{-1}}\chi_{{}^hP}={}^\varphi\chi_{{}^hP}=\chi_{{}^gP}$ and therefore ${}^{h^{-1}}\chi_{{}^hP}={}^{g^{-1}}\chi_{{}^gP}$. Observe that the tuple $(\chi_{(P,e)})\in \prod_{(P,e)\in\BP_\OO(B)}R_\K(I_{(P,e)}/P,e)$ just defined is $G$-fixed: to see this, let $(P,e)\in\BP_\OO(B)$ and let $g\in G$. Choose an element $h\in G$ such that ${}^h(P,e)\leq(D,e_D)$. Then $\chi_{(P,e)}={}^{h^{-1}}\chi_{{}^hP}$ by definition. Now ${}^{hg^{-1}}({}^g(P,e))\leq(D,e_D)$, so $\chi_{{}^g(P,e)}={}^{gh^{-1}}\chi_{{}^{hg^{-1}g}P}={}^{gh^{-1}}\chi_{{}^hP}$. It follows that
		\begin{equation*}
			{}^g\chi_{(P,e)}={}^g({}^{h^{-1}}\chi_{{}^hP})=\chi_{{}^g(P,e)}.
		\end{equation*}
		So the tuple $(\chi_{(P,e)})$ is $G$-fixed. Finally, note that if $P\subgp D$ then $\chi_{(P,e_P)}=\chi_P$, so that $\pi((\chi_{(P,e)}))=(\chi_P)$. The proof is complete.
	\end{proof}
	
	Let
	\begin{equation*}
		\delta_B:=\delta:T_\OO(B)\to\left(\prod_{P\subgp D}R_\K(I_P/P,e_P)\right)^\mathcal{F}
	\end{equation*}
	denote the composite of the maps $\alpha$ and $\pi$ --- in other words, $\delta$ is the homomorphism making the diagram below commute:
	% https://q.uiver.app/?q=WzAsMyxbMCwxLCJUX1xcT08oQikiXSxbMCwwLCJcXGxlZnQoXFxwcm9kX3soUCxlKVxcaW5cXEJQX1xcT08oQil9Ul9cXEsoSV97KFAsZSl9L1AsZSlcXHJpZ2h0KV5HIl0sWzEsMCwiXFxsZWZ0KFxccHJvZF97UFxcc3ViZ3AgRH1SX1xcSyhJX1AvUCxlX1ApXFxyaWdodCleXFxtYXRoY2Fse0Z9Il0sWzAsMSwiXFxhbHBoYSJdLFsxLDIsIlxccGkiXSxbMCwyLCJcXGRlbHRhIiwyLHsic3R5bGUiOnsiYm9keSI6eyJuYW1lIjoiZGFzaGVkIn19fV1d
	\[\begin{tikzcd}
		{\left(\prod_{(P,e)\in\BP_\OO(B)}R_\K(I_{(P,e)}/P,e)\right)^G} & {\left(\prod_{P\subgp D}R_\K(I_P/P,e_P)\right)^\mathcal{F}} \\
		{T_\OO(B)}
		\arrow["\alpha", from=2-1, to=1-1]
		\arrow["\pi", from=1-1, to=1-2]
		\arrow["\delta"', dashed, from=2-1, to=1-2]
	\end{tikzcd}\]
	Note that if $M\in\ltriv{B}$ then
	\begin{equation*}
		\delta([M])=(\chi_{M(P,e_P)})_{P\subgp D}
	\end{equation*}
	where $\chi_{M(P,e_P)}$ denotes the character of the trivial source $\OO I_Pe_P$-module $M(P,e_P)=e_P\Res_{I_P}^{N_G(P)}M(P)$.
	
	\begin{theorem}\label{thm:coherence4}
		Let $B\in\Bl(\OO G)$ and let $(D,e_D)\in\BP_\OO(B)$ be a maximal $B$-Brauer pair. If $P\subgp D$ write $e_P$ for the unique block idempotent of $\OO C_G(P)$ such that $(P,e_P)\leq(D,e_D)$ and set $I_P=N_G(P,e_P)$. Let $\mathcal{F}=\mathcal{F}_{(D,e_D)}(G,B)$. The image of $T_\OO(B)$ under the map
		\begin{align*}
			\delta_B:=\delta:T_\OO(B)&\into\left(\prod_{P\subgp D}R_\K(I_P/P,e_P)\right)^\mathcal{F}\\
							[M]&\mapsto (\chi_{M(P,e_P)})_{P\subgp D}\qquad\qquad M\in\ltriv{B}
		\end{align*}
		is equal to the subgroup of tuples $(\chi_P)\in\left(\prod_{P\subgp D}R_\K(I_P/P,e_P)\right)^\mathcal{F}$ that satisfy
		\begin{equation}\label{C3}\tag{C3}
		\left\{
		\begin{gathered}
			\chi_P(us\epsilon_{\gp{u}})=\Ind_{C_{I_P\cap I_{P\gp{u}}}(u)}^{C_{I_P}(u)}(\Res_{C_{I_P\cap I_{P\gp{u}}}(u)}^{I_{P\gp{u}}}(\chi_{P\gp{u}}))(s)\\
			\text{for all }P\subgp D\text{, }u\in N_D(P)\text{, and }s\in C_{I_P}(u)_{p'}\text{,}\\
			\text{where }\epsilon_{\gp{u}}\text{ is the unique block of }\OO C_{I_P}(u)\text{ covering }e_{P\gp{u}}.
		\end{gathered}
		\right.
		\end{equation}
	\end{theorem}

	\begin{proof}
		For ease, let $A$ denote the subgroup of tuples $(\chi_P)\in\left(\prod_{P\subgp D}R_\K(I_P/P,e_P)\right)^\mathcal{F}$ that satisfy Condition (\ref{C3}). We need to show $\delta(T_\OO(B))=A$. By definition of $\delta$, this means we need to show that $\pi(\alpha(T_\OO(B)))=A$. Let $(\chi_{(P,e)})\in\alpha(T_\OO(B))$. Then the tuple $(\chi_{(P,e)})$ satisfies Condition (\ref{C1}) of Theorem \ref{thm:coherence3}. Let $P\subgp D$, $u\in N_D(P)$, and $s\in C_{I_P}(u)_{p'}$. Set $\epsilon_{\gp{u}}=\tr_{C_{I_P\cap I_{P\gp{u}}}(u)}^{C_{I_P}(u)}(e_{P\gp{u}})$. Then by Lemma \ref{lem:normalizerbrauerpairs}, $(u,\epsilon_{\gp{u}})\in\BE_\OO(I_P,e_P)$. Note that $e_{P\gp{u}}\in\bli(\OO C_G(P\gp{u}))$ satisfies $\epsilon_{\gp{u}}\cdot e_{P\gp{u}}\neq 0$. So by Condition (\ref{C1}),
		\begin{equation*}
			\chi_{(P,e_P)}(us\epsilon_{\gp{u}})=\Ind_{C_{I_P\cap I_{P\gp{u}}}(u)}^{C_{I_P}(u)}(\Res_{C_{I_P\cap I_{P\gp{u}}}(u)}^{I_{P\gp{u}}}(\chi_{(P\gp{u},e_{P\gp{u}})}))(s).
		\end{equation*}
		This shows that $\pi((\chi_{(P,e)}))\in A$, hence $\delta(T_\OO(B))\subseteq A$.
		
		It remains to show that $A\subseteq\delta(T_\OO(B))$. Since $\delta(T_\OO(B))=\pi(\alpha(T_\OO(B)))$ and $\pi$ restricts to an isomorphism 
		\begin{equation*}
			\left(\prod_{(P,e)\in\BP_\OO(B)}R_\K(I_{(P,e)}/P,e)\right)^G\isoto \left(\prod_{P\subgp D}R_\K(I_P/P,e_P)\right)^\mathcal{F}.
		\end{equation*}
		it is enough to show that $\pi^{-1}(A)\subseteq\alpha(T_\OO(B))$. Let $(\chi_P)\in A$. Recall from the proof of Proposition \ref{prop:pi} how $\pi^{-1}((\chi_P))$ is defined: let $(P,e)\in\BP_\OO(B)$ and choose an element $g\in G$ such that ${}^g(P,e)\leq(D,e_D)$. Set $\chi_{(P,e)}:={}^{g^{-1}}\chi_{{}^gP}\in R_\K(I_{(P,e)}/P,e)$. Then the definition of $\chi_{(P,e)}$ does not depend on the choice of $g$, and $\pi^{-1}((\chi_P))=(\chi_{(P,e)})_{(P,e)\in\BP_\OO(B)}$. We need to show $(\chi_{(P,e)})\in\alpha(T_\OO(B))$. To achieve this, we show that the tuple $(\chi_{(P,e)})$ satisfies Condition (\ref{C1}) of Theorem \ref{thm:coherence3}.
		
		Let $P$ be a fully $\mathcal{F}$-normalized subgroup of $D$. Let $(u,\epsilon)\in\BE_\OO(I_P,e_P)$, let $s\in C_{I_P}(u)_{p'}$, and let $f\in\bli(\OO C_G(P\gp{u}))$ be such that $\epsilon\cdot f\neq 0$. Note that $(N_D(P),e_{N_D(P)})$ is a maximal $\OO I_Pe_P$-Brauer pair by Lemma \ref{lem:normalizerbrauerpairs}. Therefore there exists an element $g\in I_P$ such that ${}^g(\gp{u},\epsilon)\leq(N_D(P),e_{N_D(P)})$. By Lemma \ref{lem:normalizerbrauerpairs} and the uniqueness of Brauer pairs it follows that ${}^g\epsilon=\epsilon_{\gp{{}^gu}}$. Since $\epsilon\cdot f\neq 0$ we have that $\epsilon_{\gp{{}^gu}}\cdot{}^gf\neq 0$, hence ${}^gf$ is $C_{I_P}({}^gu)$-conjugate to $e_{P\gp{{}^gu}}$. Let $h\in C_{I_P}({}^gu)$ be such that ${}^{hg}f=e_{P\gp{{}^gu}}$. Now by Condition (\ref{C3}) we have that
		\begin{equation*}
			\chi_P({}^gu{}^gs\epsilon_{\gp{{}^gu}})=\Ind_{C_{I_P\cap I_{P\gp{{}^gu}}}({}^gu)}^{C_{I_P}({}^gu)}(\Res_{C_{I_P\cap I_{P\gp{{}^gu}}}({}^gu)}^{I_{P\gp{{}^gu}}}(\chi_{P\gp{{}^gu}}))({}^gs).
		\end{equation*}
		Note that the left hand side of the above equation can be rewritten:
		\begin{equation*}
			\chi_P({}^gu{}^gs\epsilon_{\gp{{}^gu}})=\chi_P({}^g(us\epsilon))=\chi_P(us\epsilon)
		\end{equation*}
		since $g\in I_P$. The right hand side can also be rewritten: since ${}^{hg}(P\gp{u},f)=(P\gp{{}^gu},e_{P\gp{{}^gu}})$ we have $\chi_{P\gp{{}^gu}}={}^{hg}\chi_{(P\gp{u},f)}$, and therefore
		\begin{align*}
			\Ind_{C_{I_P\cap I_{P\gp{{}^gu}}}({}^gu)}^{C_{I_P}({}^gu)}(&\Res_{C_{I_P\cap I_{P\gp{{}^gu}}}({}^gu)}^{I_{P\gp{{}^gu}}}(\chi_{P\gp{{}^gu}}))({}^gs)\\
			&=\Ind_{C_{I_P\cap I_{P\gp{{}^gu}}}({}^gu)}^{C_{I_P}({}^gu)}(\Res_{C_{I_P\cap I_{P\gp{{}^gu}}}({}^gu)}^{I_{P\gp{{}^gu}}}({}^{hg}\chi_{(P\gp{u},f)}))({}^gs)\\
			&={}^{hg}\Ind_{C_{I_P\cap I_{(P\gp{u},f)}}(u)}^{C_{I_P}(u)}(\Res_{C_{I_P\cap I_{(P\gp{u},f)}}(u)}^{I_{(P\gp{u},f)}}(\chi_{(P\gp{u},f)}))({}^gs)\\
			&={}^{gg^{-1}hg}\Ind_{C_{I_P\cap I_{(P\gp{u},f)}}(u)}^{C_{I_P}(u)}(\Res_{C_{I_P\cap I_{(P\gp{u},f)}}(u)}^{I_{(P\gp{u},f)}}(\chi_{(P\gp{u},f)}))({}^gs)\\
			&={}^{g}\Ind_{C_{I_P\cap I_{(P\gp{u},f)}}(u)}^{C_{I_P}(u)}(\Res_{C_{I_P\cap I_{(P\gp{u},f)}}(u)}^{I_{(P\gp{u},f)}}(\chi_{(P\gp{u},f)}))({}^gs)\\
			&=\Ind_{C_{I_P\cap I_{(P\gp{u},f)}}(u)}^{C_{I_P}(u)}(\Res_{C_{I_P\cap I_{(P\gp{u},f)}}(u)}^{I_{(P\gp{u},f)}}(\chi_{(P\gp{u},f)}))(s).
		\end{align*}
		Thus we have
		\begin{equation*}
			\chi_P(us\epsilon)=\Ind_{C_{I_P\cap I_{(P\gp{u},f)}}(u)}^{C_{I_P}(u)}(\Res_{C_{I_P\cap I_{(P\gp{u},f)}}(u)}^{I_{(P\gp{u},f)}}(\chi_{(P\gp{u},f)}))(s).
		\end{equation*}
	
		Now let $(P,e)\in\BP_\OO(B)$, let $(u,\epsilon)\in\BE_\OO(I_{(P,e)},e)$, let $s\in C_{I_{(P,e)}}(u)_{p'}$, and let $f\in\bli(\OO C_G(P\gp{u}))$ be such that $\epsilon\cdot f\neq 0$. We need to show that
		\begin{equation*}
			\chi_{(P,e)}(us\epsilon)=\Ind_{C_{I_{(P,e)}\cap I_{(P\gp{u},f)}}(u)}^{C_{I_{(P,e)}}(u)}(\Res_{C_{I_{(P,e)}\cap I_{(P\gp{u},f)}}(u)}^{I_{(P\gp{u},f)}}(\chi_{(P\gp{u},f)}))(s).
		\end{equation*}
		Let $g\in G$ be such that ${}^g(P,e)\leq(D,e_D)$. Since any subgroup of $D$ is $\mathcal{F}$-conjugate to a fully $\mathcal{F}$-normalized subgroup of $D$ we can assume that ${}^gP$ is fully $\mathcal{F}$-normalized. Recall that $\chi_{(P,e)}={}^{g^{-1}}\chi_{{}^gP}$. Now $({}^gu,{}^g\epsilon)\in\BE_\OO(I_{{}^gP},e_{{}^gP})$, ${}^gs\in C_{I_{{}^gP}}({}^gu)_{p'}$, and ${}^gf$ is a block idempotent of $\OO C_G({}^gP\gp{{}^gu})$ satisfying ${}^g\epsilon\cdot {}^gf\neq 0$. So by the previous paragraph we have
		\begin{equation*}
			\chi_{{}^gP}({}^g(us\epsilon))=\Ind_{C_{I_{{}^gP}\cap I_{({}^gP\gp{{}^gu},{}^gf)}}({}^gu)}^{C_{I_{{}^gP}}({}^gu)}(\Res_{C_{I_{{}^gP}\cap I_{({}^gP\gp{{}^gu},{}^gf)}}({}^gu)}^{I_{({}^gP\gp{{}^gu},{}^gf)}}(\chi_{({}^gP\gp{{}^gu},{}^gf)}))({}^gs).
		\end{equation*}
		The left hand side of this equation can be rewritten:
		\begin{equation*}
			\chi_{{}^gP}({}^g(us\epsilon))={}^{g^{-1}}\chi_{{}^gP}(us\epsilon)=\chi_{(P,e)}(us\epsilon).
		\end{equation*}
		Since $\chi_{({}^gP\gp{{}^gu},{}^gf)}=\chi_{{}^g(P\gp{u},f)}={}^g\chi_{(P\gp{u},f)}$ the right hand side can also be rewritten:
		\begin{align*}
			\Ind_{C_{I_{{}^gP}\cap I_{({}^gP\gp{{}^gu},{}^gf)}}({}^gu)}^{C_{I_{{}^gP}}({}^gu)}(&\Res_{C_{I_{{}^gP}\cap I_{({}^gP\gp{{}^gu},{}^gf)}}({}^gu)}^{I_{({}^gP\gp{{}^gu},{}^gf)}}(\chi_{({}^gP\gp{{}^gu},{}^gf)}))({}^gs)\\
			&=\Ind_{C_{I_{{}^gP}\cap I_{({}^gP\gp{{}^gu},{}^gf)}}({}^gu)}^{C_{I_{{}^gP}}({}^gu)}(\Res_{C_{I_{{}^gP}\cap I_{({}^gP\gp{{}^gu},{}^gf)}}({}^gu)}^{I_{({}^gP\gp{{}^gu},{}^gf)}}({}^g\chi_{(P\gp{u},f)}))({}^gs)\\
			&={}^g\Ind_{C_{I_{(P,e)}\cap I_{(P\gp{u},f)}}(u)}^{C_{I_{(P,e)}}(u)}(\Res_{C_{I_{(P,e)}\cap I_{(P\gp{u},f)}}(u)}^{I_{(P\gp{u},f)}}(\chi_{(P\gp{u},f)}))({}^gs)\\
			&=\Ind_{C_{I_{(P,e)}\cap I_{(P\gp{u},f)}}(u)}^{C_{I_{(P,e)}}(u)}(\Res_{C_{I_{(P,e)}\cap I_{(P\gp{u},f)}}(u)}^{I_{(P\gp{u},f)}}(\chi_{(P\gp{u},f)}))(s).
		\end{align*}
		Thus we have
		\begin{equation*}
			\chi_{(P,e)}(us\epsilon)=\Ind_{C_{I_{(P,e)}\cap I_{(P\gp{u},f)}}(u)}^{C_{I_{(P,e)}}(u)}(\Res_{C_{I_{(P,e)}\cap I_{(P\gp{u},f)}}(u)}^{I_{(P\gp{u},f)}}(\chi_{(P\gp{u},f)}))(s).
		\end{equation*}
		This shows that the tuple $(\chi_{(P,e)})$ satisfies Condition (\ref{C1}) of Theorem \ref{thm:coherence3}, and hence that $(\chi_{(P,e)})\in\alpha(T_\OO(B))$. By definition of $\delta$ we obtain that $(\chi_P)=\pi((\chi_{(P,e)}))\in\delta(T_\OO(B))$, and since $(\chi_P)$ was an arbitrary element of $A$ it follows that $A\subseteq\delta(T_\OO(B))$. The proof is complete.
	\end{proof}
	
	\section{Coherence conditions for trivial source bimodules with twisted diagonal vertices}
	
	Throughout this section $G$ and $H$ denote finite groups and $(\K,\OO,F)$ is a $p$-modular system large enough for $G\times H$. We follow the conventions set up in \cite{Boltje_2020}. In particular, if $R$ is a commutative ring and $M$ is an $(RG,RH)$-bimodule we always assume that the induced left and right $R$-module structures on $M$ coincide. Any $(RG,RH)$-bimodule $M$ may be viewed as a left $R[G\times H]$-module by defining $(g,h)m=gmh^{-1}$, and vice versa. One obtains an isomorphism of categories $\bimod{RG}{RH}\iso\lmod{R[G\times H]}$ in this way. We also identify $R[G\times H]$ with $(RG)\tensor_R(RH)$ via the isomorphism $(g,h)\mapsto g\tensor h$. If $e\in Z(RG)$ and $f\in Z(RH)$ are idempotents then an $(RGe,RHf)$-bimodule $M$ is the same thing as a left $R[G\times H](e\tensor f^\ast)$-module with these conventions. Here $f^\ast$ denotes the image of $f$ under the antipode $(-)^\ast:RH\to RH$, $h\mapsto h^{-1}$. 
	
	Set $T_\OO(G,H):=T(\OO G,\OO H):=T(\OO[G\times H])$. More generally, if $e\in Z(\OO G)$ and $f\in Z(\OO H)$ are idempotents, set
	\begin{equation*}
		T(\OO Ge,\OO Hf):=T(\OO[G\times H](e\tensor f^\ast)).
	\end{equation*}
	
	Let $e\in Z(\OO G)$ and $f\in Z(\OO H)$ be idempotents. Let $T^\Delta(\OO Ge,\OO Hf)=T^\Delta(\OO[G\times H](e\tensor f^\ast))$ denote the subgroup of $T(\OO Ge,\OO Hf)$ spanned by the standard basis elements $[M]$ where $M$ is an indecomposable trivial source $\OO[G\times H](e\tensor f^\ast)$-module with twisted diagonal vertices.
	
	Recall from Section \ref{sec:coherenceconditions} (specifically, Theorem \ref{thm:coherence2}) that we have an injective homomorphism
	\begin{equation*}
		\beta_{G\times H}=\beta:T(\OO Ge,\OO Hf)\into\left(\prod_{P\in S_p(G\times H)}R_\K(N_{G\times H}(P)/P,\br_P(e\tensor f^\ast))\right)^{G\times H}.
	\end{equation*}
	Our next result characterizes the image of $T^\Delta(\OO Ge,\OO Hf)$ under $\beta$.
	
	\begin{theorem}\label{thm:diagcoherence1}
		Let $G$ and $H$ be finite groups and let $e\in Z(\OO G)$, $f\in Z(\OO H)$ be idempotents. With the notation above,
		\begin{equation*}
			\beta(T^\Delta(\OO Ge,\OO Hf))=\set{(\chi_P)\in\beta(T(\OO Ge,\OO Hf))|\chi_P=0\text{ if }P\notin S_p^\Delta(G\times H)}.
		\end{equation*}
	\end{theorem}

	\begin{proof}
		First note that the collection of character tuples $(\chi_P)\in\beta(T(\OO Ge,\OO Hf))$ that satisfy $\chi_P=0$ if $P\notin S_p^\Delta(G\times H)$ is a subgroup of $\beta(T(\OO Ge,\OO Hf))$. For ease, let us denote this subgroup by $B$. So we need to show that $\beta(T^\Delta(\OO Ge,\OO Hf))=B$. Let $M$ be an indecomposable trivial source $\OO[G\times H](e\tensor f^\ast)$-module with twisted diagonal vertices (so $[M]$ is a standard basis element in $T^\Delta(\OO Ge,\OO Hf)$). Recall that $\beta([M])=(\chi_{M(P)})_{P\in S_p(G\times H)}$. Let $P\in S_p(G\times H)\setminus S_p^\Delta(G\times H)$ and suppose that $\chi_{M(P)}\neq 0$. Then $M(P)\neq 0$, hence $\overline{M}(P)\neq 0$ also. By \cite[Lemma 3.6(a)]{Boltje_2020} $P$ must be $G\times H$-conjugate to a subgroup of a vertex of $M$. But this implies that $P$ is a twisted diagonal subgroup of $G\times H$, a contradiction. Thus $\chi_{M(P)}=0$. This shows that $\beta([M])\in B$. Since character tuples of the form $\beta([M])$ generate $\beta(T^\Delta(\OO Ge,\OO Hf))$ it follows that $\beta(T^\Delta(\OO Ge,\OO Hf))\subseteq B$. 
		
		It remains to show that $B\subseteq\beta(T^\Delta(\OO Ge,\OO Hf))$. It suffices to show that $\beta^{-1}(B)\subseteq T^\Delta(\OO Ge,\OO Hf)$. Suppose, by way of contradiction, that $\beta^{-1}(B)\not\subseteq T^\Delta(\OO Ge,\OO Hf)$. Let $m\in\beta^{-1}(B)\setminus T^\Delta(\OO Ge,\OO Hf)$. Write $m=\sum a_{[M]}[M]$ where $[M]$ runs over the standard basis elements of $T(\OO Ge,\OO Hf)$ --- i.e., the isomorphism classes of indecomposable trivial source $\OO[G\times H](e\tensor f^\ast)$-modules --- and where $a_{[M]}\in\Z$. Since $m\notin T^\Delta(\OO Ge,\OO Hf)$ there exists a standard basis element $[M]$ such that $a_{[M]}\neq 0$ and $M$ has a vertex that is not twisted diagonal. Choose a $p$-subgroup $P\in S_p(G\times H)\setminus S_p^\Delta(G\times H)$ maximal with respect to the property that $P$ is a vertex of an indecomposable trivial source $\OO[G\times H](e\tensor f^\ast)$-module $M$ with $a_{[M]}\neq 0$. Since $P$ is not a twisted diagonal $p$-subgroup of $G\times H$ and $\beta(m)\in B$ we have that $\sum a_{[M]}\chi_{M(P)}=0$. Now if $M$ is an indecomposable trivial source $\OO[G\times H](e\tensor f^\ast)$-module and $P$ is not $G\times H$-conjugate to a subgroup of a vertex of $M$ then $M(P)=0$ by \cite[Lemma 3.6(a)]{Boltje_2020}. On the other hand, if $P$ is $G\times H$-conjugate to a subgroup of a vertex of $M$ and $a_{[M]}\neq 0$ then $P$ must be a vertex of $M$ by maximality. Thus we have that
		\begin{equation*}
			\sum_{P\in\operatorname{vtx}(M)}a_{[M]}\chi_{M(P)}=0.
		\end{equation*}
		(Note: the sum above is taken over the set of isomorphism classes of indecomposable trivial source $\OO[G\times H](e\tensor f^\ast)$-modules that have $P$ as a vertex.) Recall that the Brauer construction $M\mapsto \overline{M}(P)$ induces a bijection between the set of isomorphism classes of indecomposable trivial source $\OO[G\times H]$-modules with vertex $P$ and the set of isomorphism classes of projective indecomposable $F[N_{G\times H}(P)/P]$-modules (see \cite[Proposition 3.3(c)]{Boltje_2020}). It follows that the ``$\OO$-lifted'' Brauer construction $M\mapsto M(P)$ induces a bijection between the set of isomorphism classes of indecomposable trivial source $\OO[G\times H]$-modules with vertex $P$ and the set of isomorphism classes of projective indecomposable $\OO[N_{G\times H}(P)/P]$-modules. So we see that the equality $\sum_{P\in\operatorname{vtx}(M)}a_{[M]}\chi_{M(P)}=0$ is a nontrivial dependence relation between the characters of the projective indecomposable $\OO[N_{G\times H}(P)/P]$-modules, contradicting the fact that such characters are always $\K$-linearly independent. Therefore we must have $\beta^{-1}(B)\subseteq T^\Delta(\OO Ge,\OO Hf)$, and hence $B\subseteq\beta(T^\Delta(\OO Ge,\OO Hf))$. The proof is complete.
	\end{proof}

	The corollary below follows from Theorems \ref{thm:coherence2} and \ref{thm:diagcoherence1}.

	\begin{corollary}\label{cor:diagcoherencecor1}
		Let $G$ and $H$ be finite groups and let $e\in Z(\OO G)$, $f\in Z(\OO H)$ be idempotents. The image of $T^\Delta(\OO Ge,\OO Hf)$ under $\beta_{G\times H}$ is the subgroup of tuples $(\chi_P)\in\left(\prod_{P\in S_p(G\times H)}R_\K(N_{G\times H}(P)/P,\br_P^{G\times H}(e\tensor f^\ast))\right)^{G\times H}$ such that
		\begin{itemize}
			\item[(1)] $\chi_P(us,vt)=\chi_{P\gp{(u,v)}}(s,t)$ for all $P\in S_p(G\times H)$, all $p$-elements $(u,v)\in N_{G\times H}(P)$, and all $(s,t)\in C_{N_{G\times H}(P)}(u,v)_{p'}$; and
			\item[(2)] $\chi_P=0$ if $P\notin S_p^\Delta(G\times H)$.
		\end{itemize}
	\end{corollary}

	Let $G$ and $H$ be finite groups and let $e\in Z(\OO G)$, $f\in Z(\OO H)$ be idempotents. Note that if $\Delta(P,\phi,Q)$ is a twisted diagonal $p$-subgroup of $G\times H$ then
	\begin{equation*}
		\br_{\Delta(P,\phi,Q)}^{G\times H}(e\tensor f^\ast)=\br_P^G(e)\tensor\br_Q^H(f)^\ast.
	\end{equation*}
	Write $\pi$ and $i$ for the obvious projection and inclusion maps below:
	% https://q.uiver.app/?q=WzAsMixbMCwwLCJcXHByb2Rfe1BcXGluIFNfcChHXFx0aW1lcyBIKX1SX1xcSyhOX3tHXFx0aW1lcyBIfShQKS9QLFxcYnJfUF57R1xcdGltZXMgSH0oZVxcdGVuc29yIGZeXFxhc3QpKSJdLFswLDEsIlxccHJvZF97XFxEZWx0YShQLFxccGhpLFEpXFxpbiBTX3BeXFxEZWx0YShHXFx0aW1lcyBIKX1SX1xcSyhOX3tHXFx0aW1lcyBIfShcXERlbHRhKFAsXFxwaGksUSkpL1xcRGVsdGEoUCxcXHBoaSxRKSxcXGJyX1BeRyhlKVxcdGVuc29yXFxicl9RXkgoZileXFxhc3QpKSJdLFswLDEsIlxccGkiLDIseyJzdHlsZSI6eyJoZWFkIjp7Im5hbWUiOiJlcGkifX19XSxbMSwwLCJpIiwyLHsib2Zmc2V0Ijo1LCJzdHlsZSI6eyJ0YWlsIjp7Im5hbWUiOiJob29rIiwic2lkZSI6ImJvdHRvbSJ9fX1dXQ==
	\[\begin{tikzcd}
		{\prod_{P\in S_p(G\times H)}R_\K(N_{G\times H}(P)/P,\br_P^{G\times H}(e\tensor f^\ast))} \\
		{\prod_{\Delta(P,\phi,Q)\in S_p^\Delta(G\times H)}R_\K(N_{G\times H}(\Delta(P,\phi,Q))/\Delta(P,\phi,Q),\br_P^G(e)\tensor\br_Q^H(f)^\ast)}
		\arrow["\pi"', two heads, from=1-1, to=2-1]
		\arrow["i"', shift right=5, hook', from=2-1, to=1-1]
	\end{tikzcd}\]
	Since $\pi$ and $i$ are $G\times H$-homomorphisms they restrict to maps on the respective subgroups of $G\times H$-invariant tuples. We will abusively denote these restrictions by $\pi$ and $i$ in what follows. Note that the image of (the restriction of) $i$ is the subgroup of tuples $(\chi_P)$ in $(\prod_{P\in S_p(G\times H)}R_\K(N_{G\times H}(P)/P,\br_P^{G\times H}(e\tensor f^\ast)))^{G\times H}$ that satisfy $\chi_P=0$ if $P\notin S_p^\Delta(G\times H)$.
	
	\begin{corollary}\label{cor:pi1}
		Let $G$ and $H$ be finite groups and let $e\in Z(\OO G)$, $f\in Z(\OO H)$ be idempotents. The composite $\pi\beta$ is injective on $T^\Delta(\OO Ge,\OO Hf)$, and the image of $T^\Delta(\OO Ge,\OO Hf)$ under $\pi\beta$ is the subgroup of tuples $(\chi_{\Delta(P,\phi,Q)})$ in $(\prod_{\Delta(P,\phi,Q)\in S_p^\Delta(G\times H)}R_\K(N_{G\times H}(\Delta(P,\phi,Q))/\Delta(P,\phi,Q),\br_P^G(e)\tensor\br_Q^H(f)^\ast)^{G\times H}$ that satisfy:
		\begin{equation*}
			\chi_{\Delta(P,\phi,Q)}(us,vt)=\begin{cases}
				\chi_{\Delta(P,\phi,Q)\gp{(u,v)}}(s,t)	&\text{if }\Delta(P,\phi,Q)\gp{(u,v)}\in S_p^\Delta(G\times H)\\
				0	&\text{else}
			\end{cases}
		\end{equation*}
		for all $\Delta(P,\phi,Q)\in S_p^\Delta(G\times H)$, all $p$-elements $(u,v)\in N_{G\times H}(\Delta(P,\phi,Q))$, and all $(s,t)\in C_{N_{G\times H}(\Delta(P,\phi,Q))}(u,v)_{p'}$.
	\end{corollary}

	\begin{proof}
		We first show that $\pi\beta$ is injective on $T^\Delta(\OO Ge,\OO Hf)$. Let $m\in T^\Delta(\OO Ge,\OO Hf)$ and suppose that $\pi\beta(m)=0$. Write $\beta(m)=(\chi_P)_{P\in S_p(G\times H)}$. By Theorem \ref{thm:diagcoherence1} $\chi_P=0$ if $P\notin S_p^\Delta(G\times H)$. Since $\pi\beta(m)=0$ also $\chi_P=0$ if $P\in S_p^\Delta(G\times H)$. Therefore $\beta(m)=0$. But $\beta$ is injective, so $m=0$. 
		
		Now let $(\chi_P)\in\beta(T^\Delta(\OO Ge,\OO Hf))$. Then $$(\chi_P)\in\left(\prod_{P\in S_p(G\times H)}R_\K(N_{G\times H}(P)/P,\br_P^{G\times H}(e\tensor f^\ast))\right)^{G\times H}$$ and $(\chi_P)$ satisfies Conditions (1) and (2) of Corollary \ref{cor:diagcoherencecor1}. Of course we have
		\begin{equation*}
			\pi((\chi_P))\in (\prod_{\Delta(P,\phi,Q)\in S_p^\Delta(G\times H)}R_\K(N_{G\times H}(\Delta(P,\phi,Q))/\Delta(P,\phi,Q),\br_P^G(e)\tensor\br_Q^H(f)^\ast)^{G\times H}.
		\end{equation*}
		Let $\Delta(P,\phi,Q)\in S_p^\Delta(G\times H)$, let $(u,v)\in N_{G\times H}(\Delta(P,\phi,Q))$ be a $p$-element, and let $(s,t)\in C_{N_{G\times H}(\Delta(P,\phi,Q))}(u,v)_{p'}$. If $\Delta(P,\phi,Q)\gp{(u,v)}$ is not twisted diagonal then $\chi_{\Delta(P,\phi,Q)\gp{(u,v)}}=0$, so $\chi_{\Delta(P,\phi,Q)}(us,vt)=0$. If $\Delta(P,\phi,Q)\gp{(u,v)}$ is twisted diagonal then $\chi_{\Delta(P,\phi,Q)}(us,vt)=\chi_{\Delta(P,\phi,Q)\gp{(u,v)}}(s,t)$ by Condition (1) of Corollary \ref{cor:diagcoherencecor1}. This shows that $\pi\beta(T^\Delta(\OO Ge,\OO Hf))$ is contained in the subgroup of character tuples specified in the statement of the corollary.
		
		Conversely, suppose that $$(\chi_{\Delta(P,\phi,Q)})\in(\prod_{\Delta(P,\phi,Q)\in S_p^\Delta(G\times H)}R_\K(N_{G\times H}(\Delta(P,\phi,Q))/\Delta(P,\phi,Q),\br_P^G(e)\tensor\br_Q^H(f)^\ast)^{G\times H}$$ and that $(\chi_{\Delta(P,\phi,Q)})$ satisfies
		\begin{equation*}
			\chi_{\Delta(P,\phi,Q)}(us,vt)=\begin{cases}
				\chi_{\Delta(P,\phi,Q)\gp{(u,v)}}(s,t)	&\text{if }\Delta(P,\phi,Q)\gp{(u,v)}\in S_p^\Delta(G\times H)\\
				0	&\text{else}
			\end{cases}
		\end{equation*}
		for all $\Delta(P,\phi,Q)\in S_p^\Delta(G\times H)$, all $p$-elements $(u,v)\in N_{G\times H}(\Delta(P,\phi,Q))$, and all $(s,t)\in C_{N_{G\times H}(\Delta(P,\phi,Q))}(u,v)_{p'}$. It is straightforward to check that $i((\chi_{\Delta(P,\phi,Q)}))\in\beta(T^\Delta(\OO Ge,\OO Hf))$ using Corollary \ref{cor:diagcoherencecor1}. It follows that
		\begin{equation*}
			(\chi_{\Delta(P,\phi,Q)})=\pi i((\chi_{\Delta(P,\phi,Q)}))\in\pi\beta(T^\Delta(\OO Ge,\OO Hf)).
		\end{equation*} 
		This completes the proof.
	\end{proof}
	
	Now let $G$ and $H$ be finite groups, and let $A\in\Bl(\OO G)$, $B\in\Bl(\OO H)$ with respective identities $e_A$ and $f_B$. Then every Brauer pair of $A\tensor_\OO B^\ast$ is of the form $(P,e\tensor f^\ast)$ where $P\in S_p(G\times H)$, $(p_1(P),e)\in\BP_\OO(A)$ and $(p_2(P),f)\in\BP_\OO(B)$. Set
	\begin{equation*}
		Y_{(P,e\tensor f^\ast)}=N_{G\times H}(P,e\tensor f^\ast)
	\end{equation*}
	for each $(P,e\tensor f^\ast)\in\BP_\OO(A\tensor B^\ast)$. Recall from Section \ref{sec:coherenceconditions} that we have an injective group homomorphism
	% https://q.uiver.app/?q=WzAsMixbMCwwLCJUX1xcT08oQSxCKSJdLFsxLDAsIlxcbGVmdChcXHByb2RfeyhQLGVcXHRlbnNvciBmXlxcYXN0KVxcaW5cXEJQX1xcT08oQVxcdGVuc29yIEJeXFxhc3QpfVJfXFxLKFlfeyhQLGVcXHRlbnNvciBmXlxcYXN0KX0vUCxlXFx0ZW5zb3IgZl5cXGFzdClcXHJpZ2h0KV57R1xcdGltZXMgSH0iXSxbMCwxLCJcXGFscGhhIiwwLHsic3R5bGUiOnsidGFpbCI6eyJuYW1lIjoiaG9vayIsInNpZGUiOiJ0b3AifX19XV0=
	\[\begin{tikzcd}
		{T_\OO(A,B)} & {\left(\prod_{(P,e\tensor f^\ast)\in\BP_\OO(A\tensor B^\ast)}R_\K(Y_{(P,e\tensor f^\ast)}/P,e\tensor f^\ast)\right)^{G\times H}}
		\arrow["\alpha", hook, from=1-1, to=1-2]
	\end{tikzcd}\]
	defined by
	\begin{equation*}
		\alpha([M])=(\chi_{M(P,e\tensor f^\ast)})_{(P,e\tensor f^\ast)\in\BP_\OO(A\tensor B^\ast)}\qquad M\in\ltriv{A\tensor B^\ast}.
	\end{equation*}

	\begin{theorem}\label{thm:diagcoherence2}
		Let $G$ and $H$ be finite groups and let $A\in\Bl(\OO G)$, $B\in\Bl(\OO H)$ with respective identities $e_A$ and $f_B$. With the notation above,
		\begin{equation*}
			\alpha(T^\Delta(A,B))=\set{(\chi_{(P,e\tensor f^\ast)})\in\alpha(T(A,B))|\chi_{(P,e\tensor f^\ast)}=0\text{ if }P\notin S_p^\Delta(G\times H)}.
		\end{equation*}
	\end{theorem}

	\begin{proof}
		Recall from Section \ref{sec:coherenceconditions} that we have a commutative diagram
		% https://q.uiver.app/?q=WzAsMyxbMCwwLCJUKEEsQikiXSxbMSwwLCJcXGxlZnQoXFx1bmRlcnNldHtQXFxpbiBTX3AoR1xcdGltZXMgSCl9e1xccHJvZH1SX1xcSyhOX3tHXFx0aW1lcyBIfShQKS9QLFxcYnJfUF57R1xcdGltZXMgSH0oZV9BXFx0ZW5zb3IgZl9CXlxcYXN0KSlcXHJpZ2h0KV57R1xcdGltZXMgSH0iXSxbMSwxLCJcXGxlZnQoXFx1bmRlcnNldHsoUCxlXFx0ZW5zb3IgZl5cXGFzdClcXGluXFxCUF9cXE9PKEFcXHRlbnNvciBCXlxcYXN0KX17XFxwcm9kfVJfXFxLKFlfeyhQLGVcXHRlbnNvciBmXlxcYXN0KX0vUCxlXFx0ZW5zb3IgZl5cXGFzdClcXHJpZ2h0KV57R1xcdGltZXMgSH0iXSxbMCwxLCJcXGJldGEiLDAseyJzdHlsZSI6eyJ0YWlsIjp7Im5hbWUiOiJob29rIiwic2lkZSI6InRvcCJ9fX1dLFsxLDIsIlxccmhvIiwwLHsic3R5bGUiOnsidGFpbCI6eyJuYW1lIjoiaG9vayIsInNpZGUiOiJ0b3AifSwiaGVhZCI6eyJuYW1lIjoiZXBpIn19fV0sWzAsMiwiXFxhbHBoYSIsMix7InN0eWxlIjp7InRhaWwiOnsibmFtZSI6Imhvb2siLCJzaWRlIjoidG9wIn19fV1d
		\[\begin{tikzcd}
			{T(A,B)} & {\left(\underset{P\in S_p(G\times H)}{\prod}R_\K(N_{G\times H}(P)/P,\br_P^{G\times H}(e_A\tensor f_B^\ast))\right)^{G\times H}} \\
			& {\left(\underset{(P,e\tensor f^\ast)\in\BP_\OO(A\tensor B^\ast)}{\prod}R_\K(Y_{(P,e\tensor f^\ast)}/P,e\tensor f^\ast)\right)^{G\times H}}
			\arrow["\beta", hook, from=1-1, to=1-2]
			\arrow["\rho", hook, two heads, from=1-2, to=2-2]
			\arrow["\alpha"', hook, from=1-1, to=2-2]
		\end{tikzcd}\]
		The isomorphism $\rho$ maps a $G\times H$-fixed character tuple $(\chi_P)_{P\in S_p(G\times H)}$ to the tuple whose $(P,e\tensor f^\ast)$-component is $(e\tensor f^\ast)\cdot\Res_{Y_{(P,e\tensor f^\ast)}}^{N_{G\times H}(P)}(\chi_P)$; that is
		\begin{equation*}
			\rho((\chi_P)_{P\in S_p(G\times H)})=((e\tensor f^\ast)\cdot\Res_{Y_{(P,e\tensor f^\ast)}}^{N_{G\times H}(P)}(\chi_P))_{(P,e\tensor f^\ast)\in\BP_\OO(A\tensor B^\ast)}.
		\end{equation*}
		
		Let $(\chi_P)\in\beta(T^\Delta(A,B))$. Then by Theorem \ref{thm:diagcoherence1} we have that $(\chi_P)\in\beta(T(A,B))$ and $\chi_P=0$ if $P\notin S_p^\Delta(G\times H)$. For each $A\tensor B^\ast$-Brauer pair $(P,e\tensor f^\ast)$ set $\psi_{(P,e\tensor f^\ast)}=(e\tensor f^\ast)\cdot\Res_{Y_{(P,e\tensor f^\ast)}}^{N_{G\times H}(P)}(\chi_P)$, so that $\rho((\chi_P))=(\psi_{(P,e\tensor f^\ast)})$. Since $(\chi_P)\in\beta(T(A,B))$ we have that $(\psi_{(P,e\tensor f^\ast)})\in\alpha(T(A,B))$. Moreover if $(P,e\tensor f^\ast)$ is an $A\tensor B^\ast$-Brauer pair such that $P$ is not a twisted diagonal subgroup of $G\times H$ then $\chi_P=0$, hence $\psi_{(P,e\tensor f^\ast)}=0$. Since $\alpha(T^\Delta(A,B))=\rho(\beta(T^\Delta(A,B)))$ this shows that
		\begin{equation*}
			\alpha(T^\Delta(A,B))\subseteq\set{(\chi_{(P,e\tensor f^\ast)})\in\alpha(T(A,B))|\chi_{(P,e\tensor f^\ast)}=0\text{ if }P\notin S_p^\Delta(G\times H)}.
		\end{equation*}
	
		For the reverse containment, suppose that $(\chi_{(P,e\tensor f^\ast)})\in\alpha(T(A,B))$ with the property that $\chi_{(P,e\tensor f^\ast)}=0$ if $P\notin S_p^\Delta(G\times H)$. Set $(\psi_P)=\rho^{-1}((\chi_{(P,e\tensor f^\ast)}))$. Then $(\psi_P)\in\beta(T(A,B))$. Furthermore if $P$ is a $p$-subgroup of $G\times H$ that is not twisted diagonal then the proof of Lemma \ref{lem:iso1} makes it clear that $\psi_P=0$ (indeed, if $\br_P^{G\times H}(e_A\tensor f_B^\ast)=0$ then $\psi_P=0$, and otherwise $\psi_P$ is a sum of characters of the form $\Ind_{Y_{(P,e\tensor f^\ast)}}^{N_{G\times H}(P)}(\chi_{(P,e\tensor f^\ast)})$, which are all 0). So Theorem \ref{thm:diagcoherence1} tells us that $(\psi_P)\in\beta(T^\Delta(A,B))$. It follows that $(\chi_{(P,e\tensor f^\ast)})\in\alpha(T^\Delta(A,B))$. This gives the reverse containment and completes the proof.
	\end{proof}
	
	The following corollary is immediate from Theorem \ref{thm:diagcoherence2} and Condition (\ref{C1}) of Theorem \ref{thm:coherence3}.
	
	\begin{corollary}
		Let $G$ and $H$ be finite groups, let $A\in\Bl(\OO G)$ and let $B\in\Bl(\OO H)$ with respective identities $e_A$ and $f_B$. For each Brauer pair $(P,e\tensor f^\ast)\in\BP_\OO(A\tensor B^\ast)$ set $Y_{(P,e\tensor f^\ast)}=N_{G\times H}(P,e\tensor f^\ast)$. The image of $T^\Delta(A,B)$ under the map
		% https://q.uiver.app/?q=WzAsMixbMCwwLCJUX1xcT08oQSxCKSJdLFsxLDAsIlxcbGVmdChcXHByb2RfeyhQLGVcXHRlbnNvciBmXlxcYXN0KVxcaW5cXEJQX1xcT08oQVxcdGVuc29yIEJeXFxhc3QpfVJfXFxLKFlfeyhQLGVcXHRlbnNvciBmXlxcYXN0KX0vUCxlXFx0ZW5zb3IgZl5cXGFzdClcXHJpZ2h0KV57R1xcdGltZXMgSH0iXSxbMCwxLCJcXGFscGhhIiwwLHsic3R5bGUiOnsidGFpbCI6eyJuYW1lIjoiaG9vayIsInNpZGUiOiJ0b3AifX19XV0=
		\[\begin{tikzcd}
			{T_\OO(A,B)} & {\left(\prod_{(P,e\tensor f^\ast)\in\BP_\OO(A\tensor B^\ast)}R_\K(Y_{(P,e\tensor f^\ast)}/P,e\tensor f^\ast)\right)^{G\times H}}
			\arrow["\alpha", hook, from=1-1, to=1-2]
		\end{tikzcd}\]
		is the subgroup of tuples $(\chi_{(P,e\tensor f^\ast)})\in(\prod_{(P,e\tensor f^\ast)\in\BP_\OO(A\tensor B^\ast)}R_\K(Y_{(P,e\tensor f^\ast)}/P,e\tensor f^\ast))^{G\times H}$ that satisfy $\chi_{(P,e\tensor f^\ast)}=0$ if $P\notin S_p^\Delta(G\times H)$ and
		\begin{align*}
		&\chi_{(P,e\tensor f^\ast)}((us,vt)\epsilon)\\
		&=\Ind_{C_{Y_{(P,e\tensor f^\ast)}\cap Y_{(P\gp{(u,v)},e'\tensor f'^\ast)}}(u,v)}^{C_{Y_{(P,e\tensor f^\ast)}}(u,v)}(\Res_{C_{Y_{(P,e\tensor f^\ast)}\cap Y_{(P\gp{(u,v)},e'\tensor f'^\ast)}}(u,v)}^{Y_{(P\gp{(u,v)},e'\tensor f'^\ast)}}(\chi_{(P\gp{(u,v)},e'\tensor f'^\ast)}))(s,t)
		\end{align*}
		for all $(P,e\tensor f^\ast)\in\BP_\OO(A\tensor B^\ast)$, $((u,v),\epsilon)\in\BE_\OO(Y_{(P,e\tensor f^\ast)},e\tensor f^\ast)$, $(s,t)\in C_{Y_{(P,e\tensor f^\ast)}}(u,v)_{p'}$, and all $e'\in\bli(\OO C_G(p_1(P)\gp{u}))$, $f'\in\bli(\OO C_H(p_2(P)\gp{v}))$ such that $\epsilon\cdot(e'\tensor f'^\ast)\neq 0$.
	\end{corollary}

	\begin{corollary}
		Let $G$ and $H$ be finite groups, let $A\in\Bl(\OO G)$ and let $B\in\Bl(\OO H)$ with respective identities $e_A$ and $f_B$. For each Brauer pair $(P,e\tensor f^\ast)\in\BP_\OO(A\tensor B^\ast)$ set $Y_{(P,e\tensor f^\ast)}=N_{G\times H}(P,e\tensor f^\ast)$. The image of $T^\Delta(A,B)$ under the map
		% https://q.uiver.app/?q=WzAsMixbMCwwLCJUX1xcT08oQSxCKSJdLFsxLDAsIlxcbGVmdChcXHByb2RfeyhQLGVcXHRlbnNvciBmXlxcYXN0KVxcaW5cXEJQX1xcT08oQVxcdGVuc29yIEJeXFxhc3QpfVJfXFxLKFlfeyhQLGVcXHRlbnNvciBmXlxcYXN0KX0vUCxlXFx0ZW5zb3IgZl5cXGFzdClcXHJpZ2h0KV57R1xcdGltZXMgSH0iXSxbMCwxLCJcXGFscGhhIiwwLHsic3R5bGUiOnsidGFpbCI6eyJuYW1lIjoiaG9vayIsInNpZGUiOiJ0b3AifX19XV0=
		\[\begin{tikzcd}
			{T_\OO(A,B)} & {\left(\prod_{(P,e\tensor f^\ast)\in\BP_\OO(A\tensor B^\ast)}R_\K(Y_{(P,e\tensor f^\ast)}/P,e\tensor f^\ast)\right)^{G\times H}}
			\arrow["\alpha", hook, from=1-1, to=1-2]
		\end{tikzcd}\]
		is the subgroup of tuples $(\chi_{(P,e\tensor f^\ast)})\in(\prod_{(P,e\tensor f^\ast)\in\BP_\OO(A\tensor B^\ast)}R_\K(Y_{(P,e\tensor f^\ast)}/P,e\tensor f^\ast))^{G\times H}$ that satisfy $\chi_{(P,e\tensor f^\ast)}=0$ if $P\notin S_p^\Delta(G\times H)$ and
		\begin{equation*}
			\chi_{(P,e\tensor f^\ast)}(us,vt)=\sum_{\substack{e'\in\bli(\OO C_G(p_1(P)\gp{u}))\\(p_1(P),e)\nor(p_1(P)\gp{u},e')\\s\in N_G(p_1(P)\gp{u},e')}}\sum_{\substack{f'\in\bli(\OO C_H(p_2(P)\gp{v}))\\(p_2(P),f)\nor(p_2(P)\gp{v},f')\\t\in N_H(p_2(P)\gp{v},f')}}\chi_{(P\gp{(u,v)},e'\tensor f'^\ast)}(s,t)
		\end{equation*}
		for all $(P,e\tensor f^\ast)\in\BP_\OO(A\tensor B^\ast)$, all $p$-elements $(u,v)\in Y_{(P,e\tensor f^\ast)}$, and all $(s,t)\in C_{Y_{(P,e\tensor f^\ast)}}(u,v)_{p'}$.
	\end{corollary}

	\begin{proof}
		From Theorem \ref{thm:diagcoherence2} we know that $\alpha(T^\Delta(A,B))$ is the subgroup of tuples $(\chi_{(P,e\tensor f^\ast)})$ in $(\prod_{(P,e\tensor f^\ast)\in\BP_\OO(A\tensor B^\ast)}R_\K(Y_{(P,e\tensor f^\ast)}/P,e\tensor f^\ast))^{G\times H}$ satisfying $\chi_{(P,e\tensor f^\ast)}=0$ if $P\notin S_p^\Delta(G\times H)$ and Condition (\ref{C2}) of Theorem \ref{thm:coherence3}, which translates literally to
		\begin{equation*}
			\chi_{(P,e\tensor f^\ast)}(us,vt)=\sum_{\substack{\varphi\in\bli(\OO C_{G\times H}(P\gp{(u,v)}))\\(P,e\tensor f^\ast)\nor(P\gp{(u,v)},\varphi)\\(s,t)\in Y_{(P\gp{(u,v)},\varphi)}}}\chi_{(P\gp{(u,v)},\varphi)}(s,t)
		\end{equation*}
		for all $(P,e\tensor f^\ast)\in\BP_\OO(A\tensor B^\ast)$, all $p$-elements $(u,v)\in Y_{(P,e\tensor f^\ast)}$, and all $(s,t)\in C_{Y_{(P,e\tensor f^\ast)}}(u,v)_{p'}$. But for fixed such $(P,e\tensor f^\ast)$, $(u,v)$, and $(s,t)$ this sum can be reindexed, because the indexing set is in bijection with the set of ordered pairs $(e',f')$ where $e'\in\bli(\OO C_G(p_1(P)\gp{u}))$ is such that $(p_1(P),e)\nor(p_1(P)\gp{u},e')$ and $s\in N_G(p_1(P)\gp{u},e')$ and where $f'\in\bli(\OO C_H(p_2(P)\gp{v}))$ is such that $(p_2(P),f)\nor(p_2(P)\gp{v},f')$ and $t\in N_H(p_2(P)\gp{v},f')$: a bijection is given by $(e',f')\mapsto e'\tensor f'^\ast$. Note that to prove this map is a well-defined bijection one makes use of Lemma \ref{lem:GHBrauerpaircontainment}. Reindexing the sum via this bijection gives the result.
	\end{proof}

	Let $G$ and $H$ be finite groups and let $A\in\Bl(\OO G)$, $B\in\Bl(\OO H)$. Let $\BP_\OO^\Delta(A\tensor B^\ast)$ denote the subset of $\BP_\OO(A\tensor B^\ast)$ consisting of Brauer pairs with twisted diagonal first component. Note that $\BP_\OO^\Delta(A\tensor B^\ast)$ is stable under $G\times H$-conjugation. Let $\pi$ and $i$ denote the obvious projection and inclusion maps in the diagram below:
	% https://q.uiver.app/?q=WzAsMixbMCwwLCJcXHByb2RfeyhSLGVcXHRlbnNvciBmXlxcYXN0KVxcaW5cXEJQX1xcT08oQVxcdGVuc29yIEJeXFxhc3QpfVJfXFxLKFlfeyhSLGVcXHRlbnNvciBmXlxcYXN0KX0vUixlXFx0ZW5zb3IgZl5cXGFzdCkiXSxbMCwxLCJcXHByb2RfeyhcXERlbHRhKFAsXFxwaGksUSksZVxcdGVuc29yIGZeXFxhc3QpXFxpblxcQlBfXFxPT15cXERlbHRhKEFcXHRlbnNvciBCXlxcYXN0KX1SX1xcSyhZX3soXFxEZWx0YShQLFxccGhpLFEpLGVcXHRlbnNvciBmXlxcYXN0KX0vXFxEZWx0YShQLFxccGhpLFEpLGVcXHRlbnNvciBmXlxcYXN0KSJdLFswLDEsIlxccGkiLDIseyJvZmZzZXQiOjUsInN0eWxlIjp7ImhlYWQiOnsibmFtZSI6ImVwaSJ9fX1dLFsxLDAsImkiLDIseyJvZmZzZXQiOjUsInN0eWxlIjp7InRhaWwiOnsibmFtZSI6Imhvb2siLCJzaWRlIjoidG9wIn19fV1d
	\[\begin{tikzcd}
		{\prod_{(R,e\tensor f^\ast)\in\BP_\OO(A\tensor B^\ast)}R_\K(Y_{(R,e\tensor f^\ast)}/R,e\tensor f^\ast)} \\
		{\prod_{(\Delta(P,\phi,Q),e\tensor f^\ast)\in\BP_\OO^\Delta(A\tensor B^\ast)}R_\K(Y_{(\Delta(P,\phi,Q),e\tensor f^\ast)}/\Delta(P,\phi,Q),e\tensor f^\ast)}
		\arrow["\pi"', shift right=5, two heads, from=1-1, to=2-1]
		\arrow["i"', shift right=5, hook, from=2-1, to=1-1]
	\end{tikzcd}\]
	The maps $\pi$ and $i$ are both $G\times H$-equivariant, hence restrict to maps on the respective subgroups of $G\times H$-fixed tuples. The restriction of $\pi$ to the respective subgroups of $G\times H$-fixed tuples is a surjective map, and the image of the restriction of $i$ is equal to the subgroup of $G\times H$-fixed tuples $(\chi_{(R,e\tensor f^\ast)})$ that satisfy $\chi_{(R,e\tensor f^\ast)}=0$ if $R\notin S_p^\Delta(G\times H)$.
	
	The following corollaries can be proved in the same way as Corollary \ref{cor:pi1}.
	
	\begin{corollary}
		Let $G$ and $H$ be finite groups and let $A\in\Bl(\OO G)$, $B\in\Bl(\OO H)$. The composite $\pi\alpha$ is injective on $T^\Delta(A,B)$. The image of $T^\Delta(A,B)$ under $\pi\alpha$ is equal to the subgroup of tuples $(\chi_{(\Delta(P,\phi,Q),e\tensor f^\ast)})$ in
		\begin{equation*}
			(\prod_{(\Delta(P,\phi,Q),e\tensor f^\ast)\in\BP_\OO^\Delta(A\tensor B^\ast)}R_\K(Y_{(\Delta(P,\phi,Q),e\tensor f^\ast)}/\Delta(P,\phi,Q),e\tensor f^\ast))^{G\times H}
		\end{equation*}
		that satisfy: for all $(\Delta(P,\phi,Q),e\tensor f^\ast)\in\BP_\OO^\Delta(A\tensor B^\ast)$, $((u,v),\epsilon)\in\BE_\OO(Y_{(\Delta(P,\phi,Q),e\tensor f^\ast)},e\tensor f^\ast)$, $(s,t)\in C_{Y_{(\Delta(P,\phi,Q),e\tensor f^\ast)}}(u,v)_{p'}$, and all $e'\in\bli(\OO C_G(P\gp{u}))$ and $f'\in\bli(\OO C_H(Q\gp{v}))$ such that $\epsilon\cdot(e'\tensor f'^\ast)\neq 0$, one has
		\begin{align*}
			&\chi_{(\Delta(P,\phi,Q),e\tensor f^\ast)}((us,vt)\epsilon)=\\
			&\Ind_{C\cap Y_{(\Delta(P,\phi,Q)\gp{(u,v)},e'\tensor f'^\ast)}}^{C}(\Res_{C\cap Y_{(\Delta(P,\phi,Q)\gp{(u,v)},e'\tensor f'^\ast)}}^{Y_{(\Delta(P,\phi,Q)\gp{(u,v)},e'\tensor f'^\ast)}}(\chi_{(\Delta(P,\phi,Q)\gp{(u,v)},e'\tensor f'^\ast)}))(s,t)
		\end{align*}
		if $\Delta(P,\phi,Q)\gp{(u,v)}\in S_p^\Delta(G\times H)$, where for ease we write $C=C_{Y_{(\Delta(P,\phi,Q),e\tensor f^\ast)}}(u,v)$; and
		\begin{equation*}
			\chi_{(\Delta(P,\phi,Q),e\tensor f^\ast)}((us,vt)\epsilon)=0
		\end{equation*}
		if $\Delta(P,\phi,Q)\gp{(u,v)}\notin S_p^\Delta(G\times H)$.
	\end{corollary}

	\begin{corollary}
		Let $G$ and $H$ be finite groups and let $A\in\Bl(\OO G)$, $B\in\Bl(\OO H)$. The composite $\pi\alpha$ is injective on $T^\Delta(A,B)$. The image of $T^\Delta(A,B)$ under $\pi\alpha$ is equal to the subgroup of tuples $(\chi_{(\Delta(P,\phi,Q),e\tensor f^\ast)})$ in
		\begin{equation*}
			(\prod_{(\Delta(P,\phi,Q),e\tensor f^\ast)\in\BP_\OO^\Delta(A\tensor B^\ast)}R_\K(Y_{(\Delta(P,\phi,Q),e\tensor f^\ast)}/\Delta(P,\phi,Q),e\tensor f^\ast))^{G\times H}
		\end{equation*}
		that satisfy: for all $(\Delta(P,\phi,Q),e\tensor f^\ast)\in\BP_\OO^\Delta(A\tensor B^\ast)$, all $p$-elements $(u,v)\in Y_{(\Delta(P,\phi,Q),e\tensor f^\ast)}$, and all $(s,t)\in C_{Y_{(\Delta(P,\phi,Q),e\tensor f^\ast)}}(u,v)_{p'}$ one has
		\begin{equation*}
			\chi_{(\Delta(P,\phi,Q),e\tensor f^\ast)}(us,vt)=\sum_{\substack{e'\in\bli(\OO C_G(P\gp{u}))\\(P,e)\nor(P\gp{u},e')\\s\in N_G(P\gp{u},e')}}\sum_{\substack{f'\in\bli(\OO C_H(Q\gp{v}))\\(Q,f)\nor(Q\gp{v},f')\\t\in N_H(Q\gp{v},f')}}\chi_{(\Delta(P,\phi,Q)\gp{(u,v)},e'\tensor f'^\ast)}(s,t)
		\end{equation*}
		if $\Delta(P,\phi,Q)\gp{(u,v)}\in S_p^\Delta(G\times H)$, and
		\begin{equation*}
			\chi_{(\Delta(P,\phi,Q),e\tensor f^\ast)}(us,vt)=0
		\end{equation*}
		if $\Delta(P,\phi,Q)\gp{(u,v)}\notin S_p^\Delta(G\times H)$.
	\end{corollary}
	
	Let $G$ and $H$ be finite groups, let $A\in\Bl(\OO G)$ and let $B\in\Bl(\OO H)$. Let $(D,e_D)\in\BP_\OO(A)$ be a maximal $A$-Brauer pair and let $(E,f_E)\in\BP_\OO(B)$ be a maximal $B$-Brauer pair. If $P\subgp D$ write $e_P$ for the unique block idempotent of $\OO C_G(P)$ such that $(P,e_P)\leq(D,e_D)$. Likewise if $Q\subgp E$ write $(Q,f_Q)\leq(E,f_E)$. Recall that $(E,f_E^\ast)$ is a maximal $B^\ast$-Brauer pair and if $Q\subgp E$ then $(Q,f_Q^\ast)\leq(E,f_E^\ast)$. By Lemma \ref{lem:maxlABpair} $(D\times E,e_D\tensor f_E^\ast)$ is a maximal $A\tensor_\OO B^\ast$-Brauer pair. If $R\subgp D\times E$ then
	\begin{equation*}
		(R,e_{p_1(R)}\tensor f_{p_2(R)}^\ast)\leq(D\times E,e_D\tensor f_E^\ast)
	\end{equation*}
	is a containment of $A\tensor B^\ast$-Brauer pairs by Lemma \ref{lem:tensorfusion}. Let $\mathcal{A}=\mathcal{F}_{(D,e_D)}(A,G)$ and let $\mathcal{B}^\ast=\mathcal{F}_{(E,f_E^\ast)}(B^\ast,H)$. Note that $\mathcal{B}^\ast=\mathcal{F}_{(E,f_E)}(B,H)$ --- we prefer to write $\mathcal{B}^\ast$ for consistency of notation. By Lemma \ref{lem:tensorfusion} we have
	\begin{equation*}
		\mathcal{F}_{(D\times E),e_D\tensor f_E^\ast)}(A\tensor_\OO B^\ast,G\times H)=\mathcal{A}\times\mathcal{B}^\ast.
	\end{equation*}

	For each $R\subgp D\times E$ set
	\begin{equation*}
		Y_R:=Y_{(R,e_{p_1(R)}\tensor f_{p_2(R)}^\ast)}=N_{G\times H}(R,e_{p_1(R)}\tensor f_{p_2(R)}^\ast).
	\end{equation*}
	Recall from Section \ref{sec:coherenceconditions} that we have an injective group homomorphism
	\begin{equation*}
		\delta_{A\tensor B^\ast}=\delta:T(A,B)\into\left(\prod_{R\subgp D\times E}R_\K(Y_R/R,e_{p_1(R)}\tensor f_{p_2(R)}^\ast)\right)^{\mathcal{A}\times\mathcal{B}^\ast}.
	\end{equation*}
	If $M\in\ltriv{A\tensor B^\ast}$ then 
	\begin{equation*}
		\delta([M])=(\chi_{M(R,e_{p_1(R)}\tensor f_{p_2(R)}^\ast)})_{R\subgp D\times E}.
	\end{equation*}
	
	\begin{theorem}\label{thm:diagcoherence3}
		Let $G$ and $H$ be finite groups, let $A\in\Bl(\OO G)$ and $B\in\Bl(\OO H)$ and choose maximal Brauer pairs $(D,e_D)\in\BP_\OO(A)$, $(E,f_E)\in\BP_\OO(B)$. With the notation set above,
		\begin{equation*}
			\delta(T^\Delta(A,B))=\set{(\chi_R)_{R\subgp D\times E}\in\delta(T(A,B))|\chi_R=0\text{ if }R\notin S_p^\Delta(G\times H)}.
		\end{equation*}
	\end{theorem}
	
	\begin{proof}
		Recall that we have a commutative diagram
		% https://q.uiver.app/?q=WzAsMyxbMCwxLCJUKEEsQikiXSxbMCwwLCJcXGxlZnQoXFxwcm9kX3soUixlXFx0ZW5zb3IgZl5cXGFzdClcXGluXFxCUF9cXE9PKEFcXHRlbnNvciBCXlxcYXN0KX1SX1xcSyhZX3soUixlXFx0ZW5zb3IgZl5cXGFzdCl9L1IsZVxcdGVuc29yIGZeXFxhc3QpXFxyaWdodClee0dcXHRpbWVzIEh9Il0sWzEsMCwiXFxsZWZ0KFxccHJvZF97Ulxcc3ViZ3AgRFxcdGltZXMgRX1SX1xcSyhZX1IvUixlX3twXzEoUil9XFx0ZW5zb3IgZl97cF8yKFIpfV5cXGFzdClcXHJpZ2h0KV57XFxtYXRoY2Fse0F9XFx0aW1lc1xcbWF0aGNhbHtCfV5cXGFzdH0iXSxbMCwxLCJcXGFscGhhIiwwLHsic3R5bGUiOnsidGFpbCI6eyJuYW1lIjoiaG9vayIsInNpZGUiOiJ0b3AifX19XSxbMSwyLCJcXHBpIiwwLHsic3R5bGUiOnsidGFpbCI6eyJuYW1lIjoiaG9vayIsInNpZGUiOiJ0b3AifSwiaGVhZCI6eyJuYW1lIjoiZXBpIn19fV0sWzAsMiwiXFxkZWx0YSIsMl1d
		\[\begin{tikzcd}
			{\left(\prod_{(R,e\tensor f^\ast)\in\BP_\OO(A\tensor B^\ast)}R_\K(Y_{(R,e\tensor f^\ast)}/R,e\tensor f^\ast)\right)^{G\times H}} & {\left(\prod_{R\subgp D\times E}R_\K(Y_R/R,e_{p_1(R)}\tensor f_{p_2(R)}^\ast)\right)^{\mathcal{A}\times\mathcal{B}^\ast}} \\
			{T(A,B)}
			\arrow["\alpha", hook, from=2-1, to=1-1]
			\arrow["\pi", hook, two heads, from=1-1, to=1-2]
			\arrow["\delta"', from=2-1, to=1-2]
		\end{tikzcd}\]
		In particular, $\delta(T^\Delta(A,B))=\pi(\alpha(T^\Delta(A,B)))$. By Theorem \ref{thm:diagcoherence2} we have that
		\begin{equation*}
			\alpha(T^\Delta(A,B))=\set{(\chi_{(R,e\tensor f^\ast)})\in\alpha(T(A,B))|\chi_{(R,e\tensor f^\ast)}=0\text{ if }R\notin S_p^\Delta(G\times H)}.
		\end{equation*}
	
		Now let $(\chi_{(R,e\tensor f^\ast)})\in\alpha(T^\Delta(A,B))$. Then
		\begin{equation*}
			\pi((\chi_{(R,e\tensor f^\ast)}))=(\chi_{(R,e\tensor f^\ast)})_{R\subgp D\times E}\in\delta(T(A,B))
		\end{equation*}
		and if $R\subgp D\times E$ is such that $R\notin S_p^\Delta(G\times H)$ then $\chi_{(R,e\tensor f^\ast)}=0$. It follows that
		\begin{equation*}
			\delta(T^\Delta(A,B))\subseteq\set{(\chi_R)_{R\subgp D\times E}\in\delta(T(A,B))|\chi_R=0\text{ if }R\notin S_p^\Delta(G\times H)}.
		\end{equation*} 
	
		Conversely, let $(\chi_R)_{R\subgp D\times E}\in\delta(T(A,B))$ be such that $\chi_R=0$ if $R\notin S_p^\Delta(G\times H)$. Set $(\psi_{(R,e\tensor f^\ast)})=\pi^{-1}((\chi_R))$. Recall from the proof of Proposition \ref{prop:pi} that if $(R,e\tensor f^\ast)\in\BP_{\OO}(A\tensor B^\ast)$ then
		\begin{equation*}
			\psi_{(R,e\tensor f^\ast)}={}^{(g,h)^{-1}}\chi_{{}^{(g,h)}R}
		\end{equation*}
		where $(g,h)\in G\times H$ is an element satisfying ${}^{(g,h)}(R,e\tensor f^\ast)\leq(D\times E,e_D\tensor f_E^\ast)$. Since $(\chi_R)\in\delta(T(A,B))=\pi(\alpha(T(A,B)))$ we have that $(\psi_{(R,e\tensor f^\ast)})\in\alpha(T(A,B))$. If $(R,e\tensor f^\ast)$ is an $A\tensor B^\ast$-Brauer pair such that $R\notin S_p^\Delta(G\times H)$ then $\psi_{(R,e\tensor f^\ast)}={}^{(g,h)^{-1}}\chi_{{}^{(g,h)}R}=0$ since $\chi_{{}^{(g,h)}R}=0$. Thus we have that $(\psi_{(R,e\tensor f^\ast)})\in\alpha(T^\Delta(A,B))$, and so $(\chi_R)\in \pi(\alpha(T^\Delta(A,B)))=\delta(T^\Delta(A,B))$. The proof is complete.
	\end{proof}

	The corollary below follows immediately from Theorems \ref{thm:coherence4} and \ref{thm:diagcoherence3}.
	
	\begin{corollary}\label{cor:diagcoherencecor3}
		Let $G$ and $H$ be finite groups, $A\in\Bl(\OO G)$, $B\in\Bl(\OO H)$. Let $(D,e_D)$ and $(E,f_E)$ be maximal Brauer pairs (over $\OO$) for $A$ and $B$, respectively. For each $P\subgp D$ write $e_P$ for the unique block idempotent of $\OO C_G(P)$ such that $(P,e_P)\leq(D,e_D)$. Similarly, for each $Q\subgp E$ write $(Q,f_Q)\leq(E,f_E)$. Set $\mathcal{A}=\mathcal{F}_{(D,e_D)}(A,G)$ and $\mathcal{B}^\ast=\mathcal{F}_{(E,f_E^\ast)}(B^\ast,H)$. For each $R\subgp D\times E$ set $Y_R=Y_{(R,e_{p_1(R)}\tensor f_{p_2(R)}^\ast)}=N_{G\times H}(R,e_{p_1(R)}\tensor f_{p_2(R)}^\ast)$. Then the image of $T^\Delta(A,B)$ under the map
		\begin{equation*}
			\delta:T(A,B)\into\left(\prod_{R\subgp D\times E}R_\K(Y_R/R,e_{p_1(R)}\tensor f_{p_2(R)}^\ast)\right)^{\mathcal{A}\times\mathcal{B}^\ast}.
		\end{equation*}
		is equal to the subgroup of tuples $(\chi_R)\in \left(\prod_{R\subgp D\times E}R_\K(Y_R/R,e_{p_1(R)}\tensor f_{p_2(R)}^\ast)\right)^{\mathcal{A}\times\mathcal{B}^\ast}$ that satisfy $\chi_R=0$ if $R\notin S_p^\Delta(G\times H)$ and
		\begin{equation*}
			\chi_R((us,vt)\epsilon_{\gp{(u,v)}})=\Ind_{C_{Y_R\cap Y_{R\gp{(u,v)}}}(u,v)}^{C_{Y_R}(u,v)}(\Res_{C_{Y_R\cap Y_{R\gp{(u,v)}}}(u,v)}^{Y_{R\gp{(u,v)}}}(\chi_{R\gp{(u,v)}}))(s,t)
		\end{equation*}
		for all $R\subgp D\times E$, $(u,v)\in N_{D\times E}(R)$, and $(s,t)\in C_{Y_R}(u,v)_{p'}$, where
		\begin{equation*}
			\epsilon_{\gp{(u,v)}}=\tr_{C_{Y_R\cap Y_{R\gp{(u,v)}}}(u,v)}^{C_{Y_R}(u,v)}(e_{p_1(R)\gp{u}}\tensor f_{p_2(R)\gp{v}}^\ast).
		\end{equation*}
	\end{corollary}

\section{Strong isotypies and $p$-permutation equivalences}\label{sec:strongisoandpperm}

	Throughout this section $G$ and $H$ are finite groups, $(\K,\OO,F)$ is a $p$-modular system large enough for $G\times H$, $A$ is a block of $\OO G$, and $B$ is a block of $\OO H$. The purpose of this section is to introduce a new type of block equivalence that we call a \textit{strong isotypy} and to compare this definition with that of a $p$-permutation equivalence.

	\begin{hypotheses}\label{hyp}
		Let $A\in\Bl(\OO G)$ and $B\in\Bl(\OO H)$. Let $(D,e_D)\in\BP_\OO(A)$ and $(E,f_E)\in\BP_\OO(B)$ be maximal Brauer pairs for $A$ and $B$, respectively. For each subgroup $P\subgp D$ (respectively, $Q\subgp E$) write $e_P$ (resp. $f_Q$) for the unique block idempotent of $\OO C_G(P)$ (resp. $\OO C_H(Q)$) such that $(P,e_P)\leq(D,e_D)$ (resp. $(Q,f_Q)\leq (E,f_E)$). Set $\mathcal{A}=\mathcal{F}_{(D,e_D)}(G,A)$, set $\mathcal{B}=\mathcal{F}_{(E,f_E)}(H,B)$, and let $\phi:E\isoto D$ be a group isomorphism that induces an isomorphism of fusion systems $\phi:\mathcal{B}\isoto\mathcal{A}$. For each $Q\subgp E$ set
		\begin{align*}
			Y_Q&=N_{G\times H}(\Delta(\phi(Q),\phi,Q),e_{\phi(Q)}\tensor f_Q^\ast),\\
			J_Q&=N_H(Q,f_Q)
		\end{align*}
		and for each $P\subgp D$ set
		\begin{equation*}
			I_P=N_G(P,e_P).
		\end{equation*}
	\end{hypotheses}
	
	The definition below makes use of the notion of the ``extended tensor product'' of two characters. See \cite[Section 6]{Boltje_2020} for more information.
	
	%Change previous coherence theorems to be stated for Brauer elements contained in the "good" Brauer element, as in the statement of (2) below, so that the proof of this lemma is easier to follow, and previous coherence conditions are less opaque.
	
	%Do I need I_P and J_Q notation in the following lemma?
	
	%In separate corollary, give (2) in terms of generalized decomposition maps.
	
	%expand on part 3 of defn below by giving character formula.
	
	%double check the equivlance part of 2b in the defn of strong isotypy.
	
	\begin{definition}\label{defn:strongisotypy}
		Assume Hypotheses \ref{hyp}. Then a \textit{strong isotypy between} $A$ \textit{and} $B$ (relative to $(D,e_D)$, $(E,f_E)$, and $\phi:E\isoto D$) is a family of virtual characters
		\begin{equation*}
			\chi_Q\in R_\K(Y_Q/\Delta(\phi(Q),\phi,Q),e_{\phi(Q)}\tensor f_Q^\ast),\qquad Q\subgp E
		\end{equation*}
		that satisfies each of the following conditions:
		\begin{itemize}
			\item[(1)] Let $Q_1,Q_2\subgp E$ and set $P_i:=\phi(Q_i)$, $i=1,2$. If $g\in G$, $h\in H$ are such that ${}^g(P_1,e_{P_1})=(P_2,e_{P_2})$, ${}^h(Q_1,f_{Q_1})=(Q_2,f_{Q_2})$, and $c_g\phi=\phi c_h:Q_1\isoto P_2$, then ${}^{(g,h)}\chi_{Q_1}=\chi_{Q_2}$.
			\item[(2)] Let $Q\subgp E$ and set $P:=\phi(Q)$. Let $((u,v),\epsilon)\in\BE_\OO(Y_Q,e_P\tensor f_Q^\ast)$ be such that
			\begin{equation*}
				(\gp{(u,v)},\epsilon)\leq(N_{D\times E}(\Delta(P,\phi,Q)),e_{N_D(P)}\tensor f_{N_E(Q)}^\ast)
			\end{equation*}
			and let $(s,t)\in C_{Y_Q}(u,v)_{p'}$.
			\begin{itemize}
				\item[(a)] If $u=\phi(v)$ then
				\begin{equation*}
					\chi_Q((us,vt)\epsilon)=\Ind_{C_{Y_Q\cap Y_{Q\gp{v}}}(u,v)}^{C_{Y_Q}(u,v)}(\Res_{C_{Y_Q\cap Y_{Q\gp{v}}}(u,v)}^{Y_{Q\gp{v}}}(\chi_{Q\gp{v}}))(s,t).
				\end{equation*}
				\item[(b)] If $\chi_Q((us,vt)\epsilon)\neq 0$ then
				\begin{equation*}
					(\Delta(P,\phi,Q)\gp{(u,v)},e_{P\gp{u}}\tensor f_{Q\gp{v}}^\ast)\leq_{G\times H}(\Delta(D,\phi,E),e_D\tensor f_E^\ast)
				\end{equation*}
				or equivalently, there exists an $\mathcal{A}$-isomorphism $\alpha:P\gp{u}\isoto\alpha(P\gp{u})$ and a $\mathcal{B}$-isomorphism $\beta:Q\gp{v}\isoto\beta(Q\gp{v})$ such that $\alpha=\phi\beta\phi^{-1}:P\isoto\alpha(P)$ and $\alpha(u)=\phi(\beta(v))$.
			\end{itemize}
			\item[(3)] Let $Q\subgp E$ and set $P:=\phi(Q)$. Then
			\begin{equation*}
				\chi_Q\extens{H}{Y_Q,Y_Q^\circ}\chi_Q^\circ=[\K C_G(P)e_P]\in R_\K(N_{G\times G}(\Delta(P),e_P\tensor e_P^\ast)/\Delta(P),e_P\tensor e_P^\ast)
			\end{equation*}
			and
			\begin{equation*}
				\chi_Q^\circ\extens{G}{Y_Q^\circ,Y_Q}\chi_Q=[\K C_H(Q)f_Q]\in R_\K(N_{H\times H}(\Delta(Q),f_Q\tensor f_Q^\ast)/\Delta(Q),f_Q\tensor f_Q^\ast).
			\end{equation*}
		\end{itemize}
	\end{definition}
	
	Before continuing we analyze further the conclusion of condition (2b) in Definition \ref{defn:strongisotypy}. 
	
	%need to double check these results and probably restate them since I changed the first lemma to include the strong isotypy defn.
	
	%Can I make lemma below more general? Can it be phrased with just one fusion system somehow? Do I need cyclic subgroups or would it work for $U\subgp N_D(P)$?
	
	\begin{lemma}\label{lem:moreon2b}
		Let $\mathcal{A}$ and $\mathcal{B}$ be saturated fusion systems over $D$ and $E$, respectively, and let $\phi:E\isoto D$ be a group isomorphism that induces an isomorphism of fusion systems $\phi:\mathcal{B}\isoto\mathcal{A}$. Let $Q\subgp E$ and set $P:=\phi(Q)$. Let $u\in N_D(P)$, $v\in N_E(Q)$ and suppose that there exists an $\mathcal{A}$-isomorphism $\alpha:P\gp{u}\isoto\alpha(P\gp{u})$ and a $\mathcal{B}$-isomorphism $\beta:Q\gp{v}\isoto\beta(Q\gp{v})$ such that $\alpha=\phi\beta\phi^{-1}:P\isoto\alpha(P)$ and $\alpha(u)=\phi(\beta(v))$. Then
		\begin{itemize}
			\item[(1)] There exists an $\mathcal{A}$-isomorphism $\psi:P\gp{u}\isoto P\gp{\phi(v)}$ such that $\psi|_P=\id_P$ and $\psi(u)=\phi(v)$. In particular, $u$ and $\phi(v)$ are $\mathcal{N}_{\mathcal{A}}(P)$-conjugate.
			\item[(2)] There exists a $\mathcal{B}$-isomorphism $\omega:Q\gp{v}\isoto Q\gp{\phi^{-1}(u)}$ such that $\omega|_Q=\id_Q$ and $\omega(v)=\phi^{-1}(u)$. In particular, $v$ and $\phi^{-1}(u)$ are $\mathcal{N}_{\mathcal{B}}(Q)$-conjugate.
		\end{itemize}
	\end{lemma}
	
	\begin{proof}
		The assumptions imply that $\phi\beta^{-1}\phi^{-1}:\alpha(P\gp{u})\isoto P\gp{\phi(v)}$ is an $\mathcal{A}$-isomorphism. Then the $\mathcal{A}$-isomorphism $\psi=\phi\beta^{-1}\phi^{-1}\alpha:P\gp{u}\isoto P\gp{\phi(v)}$ satisfies $\psi|_P=\id_P$ and $\psi(u)=\phi(v)$. In particular, $\psi|_{\gp{u}}:\gp{u}\isoto\gp{\phi(v)}$ is an $\mathcal{N}_{\mathcal{A}}(P)$-isomorphism that maps $u$ to $\phi(v)$. This shows that (1) holds. Setting $\omega=\phi^{-1}\psi^{-1}\phi:Q\gp{v}\isoto Q\gp{\phi^{-1}(u)}$ gives (2).
	\end{proof}
	
	\begin{corollary}\label{cor:2bcor}
		Assume Hypotheses \ref{hyp} and let $\chi_Q$, $Q\subgp E$, be a strong isotypy between $A$ and $B$. Let $Q\subgp E$, $P:=\phi(Q)$, and let $((u,v),\epsilon)\in\BE_\OO(Y_Q,e_P\tensor f_Q^\ast)$ be such that
		\begin{equation*}
			(\gp{(u,v)},\epsilon)\leq(N_{D\times E}(\Delta(P,\phi,Q)),e_{N_D(P)}\tensor f_{N_E(Q)}^\ast).
		\end{equation*}
		If $u$ and $\phi(v)$ are not $\mathcal{N}_{\mathcal{A}}(P)$-conjugate then $d_{Y_Q}^{(u,v),\epsilon}(\chi_Q)=0$.
	\end{corollary}
	
	\begin{proof}
		If $u$ and $\phi(v)$ are not $\mathcal{N}_{\mathcal{A}}(P)$-conjugate then Lemma \ref{lem:moreon2b} and (2b) of Definition \ref{defn:strongisotypy} imply that $\chi_Q((us,vt)\epsilon)=0$ for all $(s,t)\in C_{Y_Q}(u,v)_{p'}$, hence $d_{Y_Q}^{(u,v),\epsilon}(\chi_Q)=0$.
	\end{proof}
	
	\begin{lemma}\label{lem:charsofppermequiv}
		Let $A\in\Bl(\OO G)$, $B\in\Bl(\OO H)$, and let $\gamma\in T^\Delta(A,B)$ be a $p$-permutation equivalence with maximal $\gamma$-Brauer pair $(\Delta(D,\phi,E),e_D\tensor f_E^\ast)$. Then $(D,e_D)\in\BP_\OO(G,A)$ and $(E,f_E)\in\BP_\OO(H,B)$ are maximal Brauer pairs. For each $P\subgp D$ (respectively, $Q\subgp E$) write $e_P$ (resp. $f_Q$) for the unique block idempotent of $\OO C_G(P)$ (resp. $\OO C_H(Q)$) such that $(P,e_P)\leq(D,e_D)$ (resp. $(Q,f_Q)\leq(E,f_E)$). Set $\mathcal{A}=\mathcal{F}_{(D,e_D)}(G,A)$ and set $\mathcal{B}=\mathcal{F}_{(E,f_E)}(H,B)$. Then $\phi:E\isoto D$ induces a fusion system isomorphism $\phi:\mathcal{B}\isoto\mathcal{A}$. For each $Q\subgp E$ set
		\begin{equation*}
			Y_Q:=N_{G\times H}(\Delta(\phi(Q),\phi,Q),e_{\phi(Q)}\tensor f_Q^\ast)
		\end{equation*}
		and set
		\begin{equation*}
			\chi_Q:=\chi_{\gamma(\Delta(\phi(Q),\phi,Q),e_{\phi(Q)}\tensor f_Q^\ast)}\in R_\K(Y_Q/\Delta(\phi(Q),\phi,Q),e_{\phi(Q)}\tensor f_Q^\ast).
		\end{equation*}
		Then the characters $\chi_Q$, $Q\subgp E$, form a strong isotypy between $A$ and $B$.
	\end{lemma}
	
	\begin{proof}
		That $(D,e_D)$ and $(E,f_E)$ are maximal Brauer pairs for $A$ and $B$, respectively, follows from \cite[Theorem 10.11(c)]{Boltje_2020}, and $\phi:E\isoto D$ induces an isomorphism of fusion systems $\phi:\mathcal{B}\isoto\mathcal{A}$ by \cite[Theorem 11.2]{Boltje_2020}. Recall that if $\mathcal{B}^\ast:=\mathcal{F}_{(E,f_E^\ast)}(H,B^\ast)$ then $\mathcal{B}=\mathcal{B}^\ast$ and also $\mathcal{A}\times\mathcal{B}^\ast=\mathcal{F}_{(D\times E,e_D\tensor f_E^\ast)}(G\times H,A\tensor B^\ast)$ by Lemma \ref{lem:tensorfusion}. Write $(\theta_R)_{R\subgp D\times E}$ for the image of $\gamma$ under the map $\delta$ of Theorem \ref{thm:diagcoherence3} and Corollary \ref{cor:diagcoherencecor3}. Note that if $Q\subgp E$ then $\chi_Q=\theta_{\Delta(\phi(Q),\phi,Q)}$.
		
		We first verify condition (1) of Definition \ref{defn:strongisotypy}. Let $Q_1,Q_2\subgp E$ and set $P_i=\phi(Q_i)$ for $i=1,2$. Let $g\in G$, $h\in H$ be such that ${}^g(P_1,e_{P_1})=(P_2,e_{P_2})$, ${}^h(Q_1,f_{Q_1})=(Q_2,f_{Q_2})$, and $c_g\phi=\phi c_h:Q_1\isoto P_2$. Note that for $i=1,2$
		\begin{equation*}
			(\Delta(P_i,\phi,Q_i),e_{P_i}\tensor f_{Q_i}^\ast)\leq(D\times E,e_D\tensor f_E^\ast)
		\end{equation*}
		is a containment of $A\tensor B^\ast$-Brauer pairs and
		\begin{equation*}
			{}^{(g,h)}(\Delta(P_1,\phi,Q_1),e_{P_1}\tensor f_{Q_1}^\ast)=(\Delta(P_2,\phi,Q_2),e_{P_2}\tensor f_{Q_2}^\ast),
		\end{equation*}
		so $c_{(g,h)}:\Delta(P_1,\phi,Q_1)\isoto\Delta(P_2,\phi,Q_2)$ is an $\mathcal{A}\times\mathcal{B}^\ast$-isomorphism. Since the character tuple $(\theta_R)_{R\subgp D\times E}$ is $\mathcal{A}\times\mathcal{B}^\ast$-fixed it follows that ${}^{(g,h)}\theta_{\Delta(P_1,\phi,Q_1)}=\theta_{\Delta(P_2,\phi,Q_2)}$. In other words ${}^{(g,h)}\chi_{Q_1}=\chi_{Q_2}$, as needed.
		
		Next we verify (2). Let $Q\subgp E$ and set $P=\phi(Q)$. Let $((u,v),\epsilon)\in\BE_\OO(Y_Q,e_P\tensor f_Q^\ast)$ be such that
		\begin{equation*}
			(\gp{(u,v)},\epsilon)\leq(N_{D\times E}(\Delta(P,\phi,Q)),e_{N_D(P)}\tensor f_{N_E(Q)}^\ast)
		\end{equation*}
		and let $(s,t)\in C_{Y_Q}(u,v)_{p'}$. First we prove (a). Assume that $u=\phi(v)$. Lemma \ref{lem:normalizerbrauerpairs} implies that
		\begin{equation*}
			(\gp{(u,v)},\tr_{C_{Y_Q\cap Y_{Q\gp{v}}}(u,v)}^{C_{Y_Q}(u,v)}(e_{P\gp{u}}\tensor f_{Q\gp{v}}^\ast))
		\end{equation*}
		is an $\OO Y_Q(e_P\tensor f_Q^\ast)$-Brauer pair contained in $(N_{D\times E}(\Delta(P,\phi,Q)),e_{N_D(P)}\tensor f_{N_E(Q)}^\ast)$, hence
		\begin{equation*}
			\epsilon=\tr_{C_{Y_Q\cap Y_{Q\gp{v}}}(u,v)}^{C_{Y_Q}(u,v)}(e_{P\gp{u}}\tensor f_{Q\gp{v}}^\ast)
		\end{equation*}
		by the uniqueness of Brauer pairs contained in a fixed Brauer pair. The equality
		\begin{equation*}
			\chi_Q((us,vt)\epsilon)=\Ind_{C_{Y_Q\cap Y_{Q\gp{v}}}(u,v)}^{C_{Y_Q}(u,v)}(\Res_{C_{Y_Q\cap Y_{Q\gp{v}}}(u,v)}^{Y_{Q\gp{v}}}(\chi_{Q\gp{v}}))(s,t).
		\end{equation*}
		then follows from the coherence condition of Corollary \ref{cor:diagcoherencecor3} applied at $\Delta(P,\phi,Q)$, $(u,v)$, and $(s,t)$. This proves (a) and we move on to (b). Assume that $\chi_Q((us,vt)\epsilon)\neq 0$. Then the coherence condition of Corollary \ref{cor:diagcoherencecor3} implies that $\theta_{\Delta(P,\phi,Q)\gp{(u,v)}}\neq 0$, hence $\gamma(\Delta(P,\phi,Q)\gp{(u,v)},e_{P\gp{u}}\tensor f_{Q\gp{v}}^\ast)\neq 0$. Thus $(\Delta(P,\phi,Q)\gp{(u,v)},e_{P\gp{u}}\tensor f_{Q\gp{v}}^\ast)$ is a $\gamma$-Brauer pair. Since any two maximal $\gamma$-Brauer pairs are $G\times H$-conjugate (see \cite[Theorem 10.11(b)]{Boltje_2020}) we have
		\begin{equation*}
			(\Delta(P,\phi,Q)\gp{(u,v)},e_{P\gp{u}}\tensor f_{Q\gp{v}}^\ast)\leq_{G\times H}(\Delta(D,\phi,E),e_D\tensor f_E^\ast).
		\end{equation*}
		This proves (b) and completes the verification of (2).
		
		Finally, the equalities in condition (3) of Definition \ref{defn:strongisotypy} follow from \cite[Propositions 11.8, 11.9]{Boltje_2020}, noting that
		\begin{equation*}
			Y_Q\ast Y_Q^\circ=N_{G\times G}(\Delta(P),e_P\tensor e_P^\ast)
		\end{equation*}
		and
		\begin{equation*}
			Y_Q^\circ\ast Y_Q=N_{H\times H}(\Delta(Q),f_Q\tensor f_Q^\ast).
		\end{equation*}
	\end{proof}
	
	%For remark later on (describing the block \epsilon as a "twisted diagonal sum":
	%\item[(2')] Let $Q\subgp E$ and set $P=\phi(Q)$. Let $(u,v)\in N_{D\times E}(\Delta(P,\phi,Q))$ and let $(s,t)\in C_{Y_Q}(u,v)_{p'}$. Set
	%typo
	%\begin{equation*}
	%	\epsilon_{\gp{(u,v)}}=\tr_{C_{Y_Q\cap Y_{Q\gp{v}}}(u,v)}^{C_{Y_Q}(u,v)}(e_{P\gp{u}}\tensor f_{Q\gp{v}}^\ast).
	%\end{equation*}
	%Then
	%\begin{itemize}
	%	\item[(a)] if $\chi_Q((us,vt)\epsilon_{(u,v)})\neq 0$ then there exist $g\in G$ and $h\in H$ such that ${}^g(P\gp{u})\subgp D$, $c_g:P\gp{u}\isoto {}^g(P\gp{u})$ is an $\mathcal{A}$-isomorphism, ${}^h(Q\gp{v})\subgp E$, $c_h:Q\gp{v}\isoto{}^h(Q\gp{v})$ is a $\mathcal{B}$-isomorphism, $\phi({}^hv)={}^gu$, and $c_g\phi=\phi c_h:Q\to {}^gP$.
	%	\item[(b)] if $u=\phi(v)$ then
	%	\begin{equation*}
	%		\chi_Q((us,vt)\epsilon_{\gp{(u,v)}})=\Ind_{C_{Y_Q\cap Y_{Q\gp{v}}}(u,v)}^{C_{Y_Q}(u,v)}(\Res_{C_{Y_Q\cap Y_{Q\gp{v}}}(u,v)}^{Y_{Q\gp{v}}}(\chi_{Q\gp{v}}))(s,t).
	%	\end{equation*}
	%	In this case, if we set $C_i=p_i(C_{Y_Q}(u,v))$ for $i=1,2$, then
	%	\begin{equation*}
	%		p_1(\Stab_{C_{Y_Q}(u,v)}(e_{P\gp{u}}\tensor f_{Q\gp{v}}^\ast))=\Stab_{C_1}(e_{P\gp{u}})
	%	\end{equation*}
	%	and
	%	\begin{equation*}
	%		p_2(\Stab_{C_{Y_Q}(u,v)}(e_{P\gp{u}}\tensor f_{Q\gp{v}}^\ast))=\Stab_{C_2}(f_{Q\gp{v}}).
	%	\end{equation*}
	%	In particular, there is a bijection
	%	\begin{equation*}
	%		\beta:\Orb_{C_2}(f_{Q\gp{v}})\isoto\Orb_{C_1}(e_{P\gp{u}})
	%	\end{equation*}
	%	such that
	%	\begin{equation*}
	%		\epsilon_{\gp{(u,v)}}=\sum_{f\in\Orb_{C_2}(f_{Q\gp{v}})}\beta(f)\tensor f^\ast.
	%	\end{equation*}
	%\end{itemize}
	
	\begin{lemma}\label{lem:maximalBrauerpairs}
		Let $B\in\Bl(\OO G)$ and let $0\neq\gamma\in T_\OO(B)$. Set
		\begin{align*}
			X(\gamma)	&:=\set{(P,e)\in\BP_\OO(B)|\chi_{\gamma(P,e)}\neq 0}\\
			Y(\gamma)	&:=\set{(P,e)\in\BP_\OO(B)|\gamma(P,e)\neq 0}\\
			Z(\gamma)	&:=\set{(P,e)\in\BP_\OO(B)|M(P,e)\neq 0\text{ for some indecomposable }M\in\operatorname{supp}(\gamma)}.
		\end{align*}
		(Here $\operatorname{supp}(\gamma)$ denotes the \textit{support} of $\gamma$, which is the set of isomorphism classes of indecomposable trivial source $B$-modules appearing in $\gamma$ with nonzero coefficient when $\gamma$ is written as a linear combination of the standard basis elements of $T_\OO(B)$.) Then
		\begin{equation*}
			\max X(\gamma)=\max Y(\gamma) =\max Z(\gamma).
		\end{equation*}
	\end{lemma}
	
	%need to rewrite proof below (made a mistake with "projective indecomposable". 
	
	\begin{proof}
		Recall that if $(Q,f)\in\BP_\OO(B)$ then the Brauer construction $-(Q,f):T_\OO(B)\to T_\OO(N_G(Q,f),f)$ restricts to an injective map from the span of the isomorphism classes $[M]$ of indecomposable trivial source $B$-modules $M$ such that $(Q,f)$ is a maximal $M$-Brauer pair to the span of the isomorphism classes of projective indecomposable $\OO N_G(Q,f)f$-modules in $T_\OO(N_G(Q,f),f)$.
		
		It is clear that $X(\gamma)\subseteq Y(\gamma)\subseteq Z(\gamma)$. Let $(Q,f)\in\max Z(\gamma)$. Let $M$ be any indecomposable module in the support of $\gamma$ satisfying $M(Q,f)\neq 0$. Then $(Q,f)$ is a maximal $M$-Brauer pair, for if $(Q,f)\leq(Q',f')$ with $(Q',f')$ a maximal $M$-Brauer pair then $(Q',f')\in Z(\gamma)$ and hence $(Q,f)=(Q',f')$. Now write $\gamma=\sum a_{[M]}[M]$ where $[M]$ runs over the standard basis elements of $T_\OO(B)$. Then
		\begin{align*}
			\gamma(Q,f)	&=\sum_{\substack{[M]\in\operatorname{supp}(\gamma)\\ M(Q,f)\neq 0}}a_{[M]}[M(Q,f)]=\sum_{\substack{[M]\in\operatorname{supp}(\gamma)\\ (Q,f)\in\max \BP_\OO(M)}}a_{[M]}[M(Q,f)]\\
			&=\left(\sum_{\substack{[M]\in\operatorname{supp}(\gamma)\\ (Q,f)\in\max \BP_\OO(M)}}a_{[M]}[M]\right)(Q,f).
		\end{align*}
		By the remark in the first paragraph of this proof it follows that $\gamma(Q,f)\neq 0$ and that $\chi_{\gamma(Q,f)}\neq 0$. Thus $(Q,f)\in Y(\gamma)$ and $(Q,f)\in X(\gamma)$. Since $X(\gamma)$ and $Y(\gamma)$ are both subsets of $Z(\gamma)$ it follows that $(Q,f)\in\max Y(\gamma)$ and $(Q,f)\in\max X(\gamma)$. Since $(Q,f)$ was an arbitrary maximal element of $Z(\gamma)$ this gives
		\begin{equation*}
			\max Z(\gamma)\subseteq\max Y(\gamma)\qquad\text{and}\qquad \max Z(\gamma)\subseteq\max X(\gamma).
		\end{equation*}
		The reverse containments follow easily from the fact that $X(\gamma)\subseteq Y(\gamma)\subseteq Z(\gamma)$.
	\end{proof}
	
	%double check lemma below, since it has been altered after including defn of strong isotypy.
	
	\begin{lemma}\label{lem:strongisocomesfrompperm}
		Assume Hypotheses \ref{hyp} and let $\chi_Q$, $Q\subgp E$, be a strong isotypy between $A$ and $B$. Then there exists a $p$-permutation equivalence $\gamma\in T^\Delta(A,B)$ such that $(\Delta(D,\phi,E),e_D\tensor f_E^\ast)$ is a maximal $\gamma$-Brauer pair and
		\begin{equation*}
			\chi_Q=\chi_{\gamma(\Delta(\phi(Q),\phi,Q),e_{\phi(Q)}\tensor f_Q^\ast)}
		\end{equation*}
		for each $Q\subgp E$.
	\end{lemma}

	\begin{proof}
		Recall that if $\mathcal{B}^\ast:=\mathcal{F}_{(E,f_E^\ast)}(H,B^\ast)$ then $\mathcal{B}=\mathcal{B}^\ast$ and also $\mathcal{A}\times\mathcal{B}^\ast=\mathcal{F}_{(D\times E,e_D\tensor f_E^\ast)}(G\times H,A\tensor B^\ast)$ by Lemma \ref{lem:tensorfusion}. Recall also that if $R\subgp D\times E$ then $$(R,e_{p_1(R)}\tensor f_{p_2(R)}^\ast)\leq(D\times E,e_D\tensor f_E^\ast)$$ is a containment of $A\tensor B^\ast$-Brauer pairs, again by Lemma \ref{lem:tensorfusion}. For each $R\subgp D\times E$ set $X_R=N_{G\times H}(R,e_{p_1(R)}\tensor f_{p_2(R)}^\ast)$. Note that if $Q\subgp E$ then $X_{\Delta(\phi(Q),\phi,Q)}=Y_Q$. 
		
		For each $R\subgp D\times E$ we define a (virtual) character
		\begin{equation*}
			\theta_R\in R_\K(X_R/R,e_{p_1(R)}\tensor f_{p_2(R)}^\ast)
		\end{equation*}
		as follows: let $R\subgp D\times E$. If there exists a subgroup $Q\subgp E$ and an $\mathcal{A}\times\mathcal{B}^\ast$-isomorphism $\psi:\Delta(\phi(Q),\phi,Q)\isoto R$ set $\theta_R:={}^\psi\chi_Q$; otherwise, set $\theta_R=0$. Note that in the former case $\theta_R$ does not depend on the subgroup $Q$ or the isomorphism $\psi$: indeed, suppose that $Q'\subgp E$ and $\psi':\Delta(\phi(Q'),\phi,Q')\isoto R$ is another isomorphism in $\mathcal{A}\times\mathcal{B}^\ast$. Then $\psi^{-1}\psi':\Delta(\phi(Q'),\phi,Q')\isoto\Delta(\phi(Q),\phi,Q)$ is an $\mathcal{A}\times\mathcal{B}^\ast$-isomorphism. Let $(g,h)\in G\times H$ be such that $\psi^{-1}\psi'=c_{(g,h)}$ and 
		\begin{equation*}
			{}^{(g,h)}(\Delta(\phi(Q'),\phi,Q'),e_{\phi(Q')}\tensor f_{Q'}^\ast)=(\Delta(\phi(Q),\phi,Q),e_{\phi(Q)}\tensor f_Q^\ast).
		\end{equation*}
		By definition we have ${}^{\psi^{-1}\psi'}\chi_{Q'}={}^{(g,h)}\chi_{Q'}$. Note that ${}^g(\phi(Q'),e_{\phi(Q')})=(\phi(Q),e_{\phi(Q)})$ and ${}^h(Q',f_{Q'})=(Q,f_Q)$ by Lemma \ref{lem:pi1pi2}. Note also that $c_g\phi=\phi c_h:Q'\isoto\phi(Q)$. So by condition (1) of Definition \ref{defn:strongisotypy} we have ${}^{(g,h)}\chi_{Q'}=\chi_Q$, i.e., ${}^{\psi^{-1}\psi'}\chi_{Q'}=\chi_Q$. Thus ${}^{\psi'}\chi_{Q'}={}^\psi\chi_Q$ and the definition of $\theta_R$ does not depend on the choice of $Q$ or $\psi$.
		
		We now have a tuple $(\theta_R)\in\prod_{R\subgp D\times E}R_\K(X_R/R,e_{p_1(R)}\tensor f_{p_2(R)}^\ast)$. Note that if $Q\subgp E$ then $\theta_{\Delta(\phi(Q),\phi,Q)}=\chi_Q$. Recall that we have an injective homomorphism
		\begin{equation*}
			\delta:T(A,B)\into\left(\prod_{R\subgp D\times E}R_\K(X_R/R,e_{p_1(R)}\tensor f_{p_2(R)}^\ast)\right)^{\mathcal{A}\times\mathcal{B}^\ast}
		\end{equation*}
		which maps the isomorphism class of a trivial source $A\tensor B^\ast$-module $M$ to $(\chi_{M(R,e_{p_1(R)}\tensor f_{p_2(R)}^\ast)})$, and that by Corollary \ref{cor:diagcoherencecor3} the image of $T^\Delta(A,B)$ under $\delta$ is the subgroup of tuples $(\theta_R)\in \left(\prod_{R\subgp D\times E}R_\K(X_R/R,e_{p_1(R)}\tensor f_{p_2(R)}^\ast)\right)^{\mathcal{A}\times\mathcal{B}^\ast}$ satisfying $\theta_R=0$ if $R\notin S_p^\Delta(G\times H)$ and the ``coherence condition''
		\begin{equation}\label{eqn:cc}
			\theta_R((us,vt)\epsilon_{\gp{(u,v)}})=\Ind_{C_{X_R\cap X_{R\gp{(u,v)}}}(u,v)}^{C_{X_R}(u,v)}(\Res_{C_{X_R\cap X_{R\gp{(u,v)}}}(u,v)}^{X_{R\gp{(u,v)}}}(\theta_{R\gp{(u,v)}}))(s,t)
		\end{equation}
		for all $R\subgp D\times E$, $(u,v)\in N_{D\times E}(R)$, and $(s,t)\in C_{X_R}(u,v)_{p'}$, where
		\begin{equation*}
			\epsilon_{\gp{(u,v)}}=\tr_{C_{X_R\cap X_{R\gp{(u,v)}}}(u,v)}^{C_{X_R}(u,v)}(e_{p_1(R)\gp{u}}\tensor f_{p_2(R)\gp{v}}^\ast).
		\end{equation*}
		I claim that the tuple $(\theta_R)$ defined above belongs to $\delta(T^\Delta(A,B))$. It is clear that $\theta_R=0$ if $R$ is not a twisted diagonal subgroup of $G\times H$ and it is not hard to see that the tuple $(\theta_R)$ is $\mathcal{A}\times\mathcal{B}^\ast$-fixed: to check this, one needs to show that if $\psi:R\isoto S$ is an $\mathcal{A}\times\mathcal{B}^\ast$-isomorphism then ${}^\psi\theta_R=\theta_S$. But if $R$ is not $\mathcal{A}\times\mathcal{B}^\ast$-isomorphic to a subgroup of the form $\Delta(\phi(Q),\phi,Q)$ for any $Q\subgp E$ then neither is $S$, and hence both $\theta_R$ and $\theta_S$ are 0, and if there exists a subgroup $Q\subgp E$ and an $\mathcal{A}\times\mathcal{B}^\ast$-isomorphism $\psi':\Delta(\phi(Q),\phi,Q)\isoto R$ then $\psi\psi':\Delta(\phi(Q),\phi,Q)\isoto S$ is an $\mathcal{A}\times\mathcal{B}^\ast$-isomorphism and we have
		\begin{equation*}
			{}^\psi\theta_R={}^\psi({}^{\psi'}\chi_Q)={}^{\psi\psi'}\chi_Q=\theta_S.
		\end{equation*}
		So to verify the claim it remains to show that the tuple $(\theta_R)$ satisfies the ``coherence condition'' (\ref{eqn:cc}) given above.
		
		Let $R\subgp D\times E$, $(u,v)\in N_{D\times E}(R)$, and $(s,t)\in C_{X_R}(u,v)_{p'}$. Set
		\begin{equation*}
			\epsilon_{\gp{(u,v)}}=\tr_{C_{X_R\cap X_{R\gp{(u,v)}}}(u,v)}^{C_{X_R}(u,v)}(e_{p_1(R)\gp{u}}\tensor f_{p_2(R)\gp{v}}^\ast).
		\end{equation*}
		By Lemma \ref{lem:normalizerbrauerpairs} $(\gp{(u,v)},\epsilon_{\gp{(u,v)}})$ is an $\OO X_R(e_{p_1(R)}\tensor f_{p_2(R)}^\ast)$-Brauer pair. 
		
		Suppose first that $R$ is not $\mathcal{A}\times\mathcal{B}^\ast$-isomorphic to any subgroup of $\Delta(D,\phi,E)$. Then neither is $R\gp{(u,v)}$, so $\theta_R=0$ and $\theta_{R\gp{(u,v)}}=0$. In this case, Equation (\ref{eqn:cc}) clearly holds.
		
		Now suppose that $R\gp{(u,v)}$ is $\mathcal{A}\times\mathcal{B}^\ast$-isomorphic to a subgroup of $\Delta(D,\phi,E)$. Let $\psi':\Delta(P',\phi,Q')\isoto R\gp{(u,v)}$ be an $\mathcal{A}\times\mathcal{B}^\ast$-isomorphism where $Q'\subgp E$ and $P'=\phi(Q')$. Write $\Delta(P,\phi,Q)$ for the unique subgroup of $\Delta(P',\phi,Q')$ such that $\psi'(\Delta(P,\phi,Q))=R$ and write $\psi$ for the restriction of $\psi'$ to $\Delta(P,\phi,Q)$. Then $\psi:\Delta(P,\phi,Q)\isoto R$ is an $\mathcal{A}\times\mathcal{B}^\ast$-isomorphism. Let us also write $(u',v')$ for the unique element of $\Delta(P',\phi,Q')$ satisfying $\psi'((u',v'))=(u,v)$. Note that $u'=\phi(v')$, that $(u',v')\in N_{D\times E}(\Delta(P,\phi,Q))$, and that $\Delta(P',\phi,Q')=\Delta(P\gp{u'},\phi,Q\gp{v'})$. Now let $(g,h)\in G\times H$ be such that $(\psi')^{-1}=c_{(g,h)}$ and
		\begin{align*}
			{}^{(g,h)}(R\gp{(u,v)},e_{p_1(R)\gp{u}}\tensor f_{p_2(R)\gp{v}}^\ast)	&=(\Delta(P',\phi,Q'),e_{P'}\tensor f_{Q'}^\ast)\\
			&=(\Delta(P\gp{u'},\phi,Q\gp{v'}),e_{P\gp{u'}}\tensor f_{Q\gp{v'}}^\ast).
		\end{align*}
		Note that $(u',v')=({}^gu,{}^hv)$. Also, $\psi^{-1}=c_{(g,h)}$ and
		\begin{equation*}
			{}^{(g,h)}(R,e_{p_1(R)}\tensor f_{p_2(R)}^\ast)=(\Delta(P,\phi,Q),e_P\tensor f_Q^\ast).
		\end{equation*}
		It follows that ${}^{(g,h)}(\gp{(u,v)},\epsilon_{\gp{(u,v)}})\in\BP_\OO(Y_Q,e_P\tensor f_Q^\ast)$. In other words, if we set $\epsilon'={}^{(g,h)}\epsilon_{\gp{(u,v)}}$ then $(\gp{(u',v')},\epsilon')$ is an $\OO Y_Q(e_P\tensor f_Q^\ast)$-Brauer pair. Note that
		\begin{align*}
			\epsilon'	&={}^{(g,h)}\epsilon_{\gp{(u,v)}}\\
						&={}^{(g,h)}\tr_{C_{X_R\cap X_{R\gp{(u,v)}}}(u,v)}^{C_{X_R}(u,v)}(e_{p_1(R)\gp{u}}\tensor f_{p_2(R)\gp{v}}^\ast)\\
						&=\tr_{C_{Y_Q\cap Y_{Q\gp{v'}}}(u',v')}^{C_{Y_Q}(u',v')}(e_{P\gp{u'}}\tensor f_{Q\gp{v'}}^\ast),
		\end{align*}
		so by Lemma \ref{lem:normalizerbrauerpairs}
		\begin{equation*}
			(\gp{(u',v')},\epsilon')\leq(N_{D\times E}(\Delta(P,\phi,Q)),e_{N_D(P)}\tensor f_{N_E(Q)}^\ast)
		\end{equation*}
		is a containment of $\OO Y_Q(e_P\tensor f_Q^\ast)$-Brauer pairs. Set $s'={}^gs$ and $t'={}^ht$. Then $(s',t')\in C_{Y_Q}(u',v')_{p'}$. Condition (2a) of Definition \ref{defn:strongisotypy} then implies that
		\begin{equation*}
		\chi_Q((u's',v't')\epsilon')=\Ind_{C_{Y_Q\cap Y_{Q\gp{v'}}}(u',v')}^{C_{Y_Q}(u',v')}(\Res_{C_{Y_Q\cap Y_{Q\gp{v'}}}(u',v')}^{Y_{Q\gp{v'}}}(\chi_{Q\gp{v'}}))(s',t').
		\end{equation*}
		Now $\theta_R={}^\psi\chi_Q={}^{(g,h)^{-1}}\chi_Q$ and $\theta_{R\gp{(u,v)}}={}^{\psi'}\chi_{Q'}={}^{(g,h)^{-1}}\chi_{Q\gp{v'}}$, so we can compute:
		\begin{align*}
			\theta_R((us,vt)\epsilon_{\gp{(u,v)}})	&={}^{(g,h)^{-1}}\chi_Q((us,vt)\epsilon_{\gp{(u,v)}})\\
				&=\chi_Q((u's',v't')\epsilon')\\
				&=\Ind_{C_{Y_Q\cap Y_{Q\gp{v'}}}(u',v')}^{C_{Y_Q}(u',v')}(\Res_{C_{Y_Q\cap Y_{Q\gp{v'}}}(u',v')}^{Y_{Q\gp{v'}}}(\chi_{Q\gp{v'}}))(s',t')\\
				&=\Ind_{{}^{(g,h)}C_{X_R\cap X_{R\gp{(u,v)}}}(u,v)}^{{}^{(g,h)}C_{X_R}(u,v)}(\Res_{{}^{(g,h)}C_{X_R\cap X_{R\gp{(u,v)}}}(u,v)}^{{}^{(g,h)}X_{R\gp{(u,v)}}}({}^{(g,h)}\theta_{R\gp{(u,v)}}))({}^gs,{}^ht)\\
				&={}^{(g,h)}[\Ind_{C_{X_R\cap X_{R\gp{(u,v)}}}(u,v)}^{C_{X_R}(u,v)}(\Res_{C_{X_R\cap X_{R\gp{(u,v)}}}(u,v)}^{X_{R\gp{(u,v)}}}(\theta_{R\gp{(u,v)}}))]({}^{(g,h)}(s,t))\\
				&=\Ind_{C_{X_R\cap X_{R\gp{(u,v)}}}(u,v)}^{C_{X_R}(u,v)}(\Res_{C_{X_R\cap X_{R\gp{(u,v)}}}(u,v)}^{X_{R\gp{(u,v)}}}(\theta_{R\gp{(u,v)}}))(s,t)
		\end{align*}
		So Equation (\ref{eqn:cc}) holds in this case as well.
		
		The final case to consider is the case in which $R$ is $\mathcal{A}\times\mathcal{B}^\ast$-isomorphic to a subgroup of $\Delta(D,\phi,E)$ but $R\gp{(u,v)}$ is not. In this case there exists a subgroup $\Delta(P,\phi,Q)\subgp\Delta(D,\phi,E)$ and an $\mathcal{A}\times\mathcal{B}^\ast$-isomorphism $\psi:\Delta(P,\phi,Q)\isoto R$. We may assume without loss of generality that $Q$ is a fully $\mathcal{B}$-normalized subgroup of $E$, for if $Q'\in Q^{\mathcal{B}}$ is a fully $\mathcal{B}$-normalized subgroup of $E$ and $\beta:Q'\isoto Q$ is a $\mathcal{B}$-isomorphism then, setting $P'=\phi(Q')$, we have that $\alpha:=\phi\beta\phi^{-1}:P'\isoto P$ is an $\mathcal{A}$-isomorphism, hence that $\psi':=(\alpha,\beta):\Delta(P',\phi,Q')\isoto\Delta(P,\phi,Q)$ is an $\mathcal{A}\times\mathcal{B}^\ast$-isomorphism, and so $\psi\psi':\Delta(P',\phi,Q')\isoto R$ is an $\mathcal{A}\times\mathcal{B}^\ast$-isomorphism. 
		
		Now since $Q$ is fully $\mathcal{B}$-normalized and $\phi:\mathcal{B}\isoto\mathcal{A}$ is an isomorphism of fusion systems, Lemma \ref{lem:twisteddiagfullynormalized} implies that $\Delta(P,\phi,Q)$ is a fully $\mathcal{A}\times\mathcal{B}^\ast$-normalized subgroup of $D\times E$. In particular, we have that $(N_{D\times E}(\Delta(P,\phi,Q)),e_{N_D(P)}\tensor f_{N_E(Q)}^\ast)$ is a maximal $\OO Y_Q(e_P\tensor f_Q^\ast)$-Brauer pair by Lemma \ref{lem:normalizerbrauerpairs}.
		
		Now let $(g,h)\in G\times H$ be such that $\psi^{-1}=c_{(g,h)}$ and
		\begin{equation*}
			{}^{(g,h)}(R,e_{p_1(R)}\tensor f_{p_2(R)}^\ast)=(\Delta(P,\phi,Q),e_P\tensor f_Q^\ast).
		\end{equation*}
		Since $(\gp{(u,v)},\epsilon_{\gp{(u,v)}})\in\BP_\OO(X_R,e_{p_1(R)}\tensor f_{p_2(R)}^\ast)$ we have that ${}^{(g,h)}(\gp{(u,v)},\epsilon_{\gp{(u,v)}})\in\BP_\OO(Y_Q,e_P\tensor f_Q^\ast)$. Because $(N_{D\times E}(\Delta(P,\phi,Q)),e_{N_D(P)}\tensor f_{N_E(Q)}^\ast)$ is a maximal $\OO Y_Q(e_P\tensor f_Q^\ast)$-Brauer pair there exists an element $(x,y)\in Y_Q$ such that
		\begin{equation*}
			{}^{(x,y)}({}^{(g,h)}(\gp{(u,v)},\epsilon_{\gp{(u,v)}}))\leq (N_{D\times E}(\Delta(P,\phi,Q)),e_{N_D(P)}\tensor f_{N_E(Q)}^\ast).
		\end{equation*}
		Set $(g',h')=(xg,yh)$. Note that $c_{(x,y)}\in\Aut_{\mathcal{A}\times\mathcal{B}^\ast}(\Delta(P,\phi,Q))$. Set $\psi'=\psi\circ c_{(x,y)}^{-1}:\Delta(P,\phi,Q)\isoto R$. Note that $\psi'$ is an $\mathcal{A}\times\mathcal{B}^\ast$-isomorphism. Note also that $(\psi')^{-1}=c_{(g',h')}$ and that
		\begin{equation*}
			{}^{(g',h')}(R,e_{p_1(R)}\tensor f_{p_2(R)}^\ast)=(\Delta(P,\phi,Q),e_P\tensor f_Q^\ast).
		\end{equation*}
		Thus, after replacing $\psi$ with $\psi'$ and $(g,h)$ with $(g',h')$, we may assume without loss of generality that
		\begin{equation*}
			{}^{(g,h)}(\gp{(u,v)},\epsilon_{\gp{(u,v)}})\leq(N_{D\times E}(\Delta(P,\phi,Q)),e_{N_D(P)}\tensor f_{N_E(Q)}^\ast).
		\end{equation*}
	
		Now since $\Delta(P,\phi,Q)$ is fully $\mathcal{A}\times\mathcal{B}^\ast$-normalized also $\Delta(P,\phi,Q)$ is receptive in $\mathcal{A}\times\mathcal{B}^\ast$. Since ${}^{(g,h)}(u,v)\in N_{D\times E}(\Delta(P,\phi,Q))$ we have that --- in the notation of \cite[Part I, Definition 2.2]{Aschbacher_2011} --- $(u,v)\in N_{\psi^{-1}}$. Therefore $\psi^{-1}$ extends to an $\mathcal{A}\times\mathcal{B}^\ast$-homomorphism defined on $R\gp{(u,v)}$. In other words, there exists an element $(g',h')\in G\times H$ such that
		\begin{equation*}
			{}^{(g',h')}(R\gp{(u,v)},e_{p_1(R)\gp{u}}\tensor f_{p_2(R)\gp{v}}^\ast)\leq(D\times E,e_D\tensor f_E^\ast)
		\end{equation*}
		and $\psi^{-1}=c_{(g',h')}:R\isoto\Delta(P,\phi,Q)$. Note that this implies
		\begin{equation*}
			{}^{(g',h')}(R,e_{p_1(R)}\tensor f_{p_2(R)}^\ast)=(\Delta(P,\phi,Q),e_P\tensor f_Q^\ast)
		\end{equation*}
		and ${}^{(g',h')}(u,v)\in N_{D\times E}(\Delta(P,\phi,Q))$. Furthermore
		\begin{align*}
			{}^{(g',h')}\epsilon_{\gp{(u,v)}}&={}^{(g',h')}\tr_{C_{X_R\cap X_{R\gp{(u,v)}}}(u,v)}^{C_{X_R}(u,v)}(e_{p_1(R)\gp{u}}\tensor f_{p_2(R)\gp{v}}^\ast)\\
			&=\tr_{C_{Y_Q\cap X_{\Delta(P,\phi,Q)\gp{({}^{g'}u,{}^{h'}v)}}}({}^{g'}u,{}^{h'}v)}^{C_{Y_Q}({}^{g'}u,{}^{h'}v)}(e_{P\gp{{}^{g'}u}}\tensor f_{Q\gp{{}^{h'}v}}^\ast).
		\end{align*}
		So by Lemma \ref{lem:normalizerbrauerpairs}
		\begin{equation*}
			{}^{(g',h')}(\gp{(u,v)},\epsilon_{\gp{(u,v)}})\leq(N_{D\times E}(\Delta(P,\phi,Q)),e_{N_D(P)}\tensor f_{N_E(Q)}^\ast)
		\end{equation*}
		is a containment of $\OO Y_Q(e_P\tensor f_Q^\ast)$-Brauer pairs. Thus we may assume without loss of generality that
		\begin{equation*}
			{}^{(g,h)}(R\gp{(u,v)},e_{p_1(R)\gp{u}}\tensor f_{p_2(R)\gp{v}}^\ast)\leq(D\times E,e_D\tensor f_E^\ast).
		\end{equation*}
	
		Set $(u',v')={}^{(g,h)}(u,v)$, $(s',t')={}^{(g,h)}(s,t)$, and $\epsilon'={}^{(g,h)}\epsilon_{\gp{(u,v)}}$. Then $((u',v'),\epsilon')\in\BE_\OO(Y_Q,e_P\tensor f_Q^\ast)$,
		\begin{equation*}
			(\gp{(u',v')},\epsilon')\leq(N_{D\times E}(\Delta(P,\phi,Q)),e_{N_D(P)}\tensor f_{N_E(Q)}^\ast),
		\end{equation*}
		and $(s',t')\in C_{Y_Q}(u',v')_{p'}$. Note that
		\begin{equation*}
			{}^{(g,h)}(R\gp{(u,v)},e_{p_1(R)\gp{u}}\tensor f_{p_2(R)\gp{v}}^\ast)=(\Delta(P,\phi,Q)\gp{(u',v')},e_{P\gp{u'}}\tensor f_{Q\gp{v'}}^\ast).
		\end{equation*}
		Now, since we are assuming that $R\gp{(u,v)}$ is not $\mathcal{A}\times\mathcal{B}^\ast$-isomorphic to a subgroup of $\Delta(D,\phi,E)$ we have $\theta_{R\gp{(u,v)}}=0$ by definition. So to verify Equation (\ref{eqn:cc}) in this final case we need to show that $\theta_R((us,vt)\epsilon_{\gp{(u,v)}})=0$. We have $\theta_R={}^\psi\chi_Q={}^{(g,h)^{-1}}\chi_Q$, so $\theta_R((us,vt)\epsilon_{\gp{(u,v)}})=\chi_Q((u's',v't')\epsilon')$. Suppose $\chi_Q((u's',v't')\epsilon')\neq 0$. Then by condition (2b) of Definition \ref{defn:strongisotypy} we have
		\begin{equation*}
			(\Delta(P,\phi,Q)\gp{(u',v')},e_{P\gp{u'}}\tensor f_{Q\gp{v'}}^\ast)\leq_{G\times H}(\Delta(D,\phi,E),e_D\tensor f_E^\ast).
		\end{equation*}
		But this implies that
		\begin{equation*}
			(R\gp{(u,v)},e_{p_1(R)\gp{u}}\tensor f_{p_2(R)\gp{v}}^\ast)\leq_{G\times H}(\Delta(D,\phi,E),e_D\tensor f_E^\ast),
		\end{equation*}
		or in other words, that $R\gp{(u,v)}$ is $\mathcal{A}\times\mathcal{B}^\ast$-isomorphic to a subgroup of $\Delta(D,\phi,E)$. So we must have $\chi_Q((u's',v't')\epsilon')=0$, hence $\theta_R((us,vt)\epsilon_{\gp{(u,v)}})=0$ and Equation (\ref{eqn:cc}) holds. This completes the proof of the claim that $(\theta_R)$ belongs to $\delta(T^\Delta(A,B))$.
		
		Let $\gamma\in T^\Delta(A,B)$ be such that $\delta(\gamma)=(\theta_R)$. Note that for each $Q\subgp E$ we have
		\begin{equation*}
			\chi_Q=\theta_{\Delta(\phi(Q),\phi,Q)}=\chi_{\gamma(\Delta(\phi(Q),\phi,Q),e_{\phi(Q)}\tensor f_Q^\ast)}
		\end{equation*}
		by definition of the map $\delta$. Now Condition (3) of Definition \ref{defn:strongisotypy} implies in particular that $\chi_Q\neq 0$ for all $Q\subgp E$. Therefore $\gamma(\Delta(D,\phi,E),e_D\tensor f_E^\ast)\neq 0$, i.e., $(\Delta(D,\phi,E),e_D\tensor f_E^\ast)$ is a $\gamma$-Brauer pair. In fact, $(\Delta(D,\phi,E),e_D\tensor f_E^\ast)$ is a maximal $\gamma$-Brauer pair since any $\gamma$-Brauer pair is an $A\tensor B^\ast$-Brauer pair and must have a twisted diagonal subgroup in its first entry. 
		
		To complete the proof it remains to show that $\gamma$ is a $p$-permutation equivalence. By \cite[Theorems 12.2, 12.3]{Boltje_2020} it suffices to show that any maximal $\gamma$-Brauer pair is $G\times H$-conjugate to $(\Delta(D,\phi,E),e_D\tensor f_E^\ast)$. Let $(\Delta(S,\psi,T),e\tensor f^\ast)\in\BP_\OO(\gamma)$ be maximal. Then
		\begin{equation*}
			(\Delta(S,\psi,T),e\tensor f^\ast)\leq_{G\times H}(D\times E,e_D\tensor f_E^\ast)
		\end{equation*}
		Since $\BP_\OO(\gamma)$ is stable under conjugation we may assume without loss of generality that $(\Delta(S,\psi,T),e\tensor f^\ast)$ is contained in $(D\times E,e_D\tensor f_E^\ast)$. Then $S\subgp D$, $e=e_S$, $T\subgp E$, and $f=f_T$. By Lemma \ref{lem:maximalBrauerpairs} we have
		\begin{equation*}
			\theta_{\Delta(S,\psi,T)}=\chi_{\gamma(\Delta(S,\psi,T),e_S\tensor f_T^\ast)}\neq 0.
		\end{equation*}
		It follows that $\Delta(S,\psi,T)$ is $\mathcal{A}\times\mathcal{B}^\ast$-isomorphic to a subgroup of $\Delta(D,\phi,E)$. In other words, there exists an element $(g,h)\in G\times H$ such that
		\begin{equation*}
			{}^{(g,h)}(\Delta(S,\psi,T),e_S\tensor f_T^\ast)\leq(\Delta(D,\phi,E),e_D\tensor f_E^\ast).
		\end{equation*}
		Since $(\Delta(S,\psi,T),e_S\tensor f_T^\ast)$ is a maximal $\gamma$-Brauer pair we must have $${}^{(g,h)}(\Delta(S,\psi,T),e_S\tensor f_T^\ast)=(\Delta(D,\phi,E),e_D\tensor f_E^\ast).$$ The proof is complete.
	\end{proof}

	%double check this...

	\begin{theorem}\label{thm:ppermequivtosi}
		Assume Hypotheses \ref{hyp}. Then the construction of Lemma \ref{lem:charsofppermequiv} defines a bijection from the set of $p$-permutation equivalences $\gamma$ such that $(\Delta(D,\phi,E),e_D\tensor f_E^\ast)$ is a maximal $\gamma$-Brauer pair and the set of strong isotypies between $A$ and $B$ relative to $(D,e_D)$, $(E,f_E)$, and $\phi:E\isoto D$.
	\end{theorem}

	\begin{proof}
		This follows from Lemmas \ref{lem:charsofppermequiv} and \ref{lem:strongisocomesfrompperm}.
	\end{proof}

\section{Connection with isotypies}

	Let $G$ and $H$ be finite groups, $(\K,\OO,F)$ a $p$-modular system large enough for $G\times H$, $A$ a block of $\OO G$, and let $B$ be a block of $\OO H$. Assume Hypotheses \ref{hyp}, and the notation set there. In this section we show that a strong isotypy between $A$ and $B$ restricts to an isotypy between $A$ and $B$. The definition of isotypy we will use is the one given in \cite[Definition 15.3]{Boltje_2020}.
	
	\begin{definition}\label{defn:isotypy}
		Let $A\in\Bl(\OO G)$, $B\in\Bl(\OO H)$, and assume Hypotheses \ref{hyp}. An \textit{isotypy} between $A$ and $B$ is a family of perfect isometries
		\begin{equation*}
			\mu_Q\in R_\K(C_G(\phi(Q))\times C_H(Q),e_{\phi(Q)}\tensor f_Q^\ast),\qquad Q\subgp E
		\end{equation*}
		such that the following conditions are satisfied:
		\begin{itemize}
			\item[(1)] (Equivariance) For every $Q\subgp E$ and $(g,h)\in G\times H$ such that $${}^{(g,h)}(\Delta(\phi(Q),\phi,Q),e_{\phi(Q)}\tensor f_Q^\ast)\leq(\Delta(D,\phi,E),e_D\tensor f_E^\ast)$$ one has ${}^{(g,h)}\mu_Q=\mu_{{}^hQ}$.
			\item[(2)] (Compatibility) For every $Q\subgp E$ and $v\in C_E(Q)$ the diagram below commutes
				\begin{equation}\label{ntscommutes}
				% https://q.uiver.app/#q=WzAsNCxbMCwwLCJDRihDX0goUSksZl9RO1xcSykiXSxbMywwLCJDRihDX0coUCksZV9QO1xcSykiXSxbMCwxLCJDRl97cCd9KENfSChRXFxncHt2fSksZl97UVxcZ3B7dn19O1xcSykiXSxbMywxLCJDRl97cCd9KENfRyhQXFxncHt1fSksZV97UFxcZ3B7dX19O1xcSykiXSxbMCwxLCJcXG11X1FcXHRlbnNvcl97Q19IKFEpfS0iXSxbMCwyLCJkX3tDX0goUSl9Xnt2LGZfe1FcXGdwe3Z9fX0iLDJdLFsxLDMsImRfe0NfRyhQKX1ee3UsZV97UFxcZ3B7dX19fSJdLFsyLDMsIlxcbXVfe1FcXGdwe3Z9fVxcdGVuc29yX3tDX0goUVxcZ3B7dn0pfS0iLDJdXQ==
				\begin{tikzcd}
					{CF(C_H(Q),f_Q;\K)} &&& {CF(C_G(P),e_P;\K)} \\
					{CF_{p'}(C_H(Q\gp{v}),f_{Q\gp{v}};\K)} &&& {CF_{p'}(C_G(P\gp{u}),e_{P\gp{u}};\K)}
					\arrow["{\mu_Q\tensor_{C_H(Q)}-}", from=1-1, to=1-4]
					\arrow["{d_{C_H(Q)}^{v,f_{Q\gp{v}}}}"', from=1-1, to=2-1]
					\arrow["{d_{C_G(P)}^{u,e_{P\gp{u}}}}", from=1-4, to=2-4]
					\arrow["{\mu_{Q\gp{v}}\tensor_{C_H(Q\gp{v})}-}"', from=2-1, to=2-4]
				\end{tikzcd}
			\end{equation}
			where $P=\phi(Q)$ and $u=\phi(v)$.
		\end{itemize}
	\end{definition}

	\begin{theorem}\label{thm:strongisotimpliesisot}
		Let $A\in\Bl(\OO G)$, $B\in\Bl(\OO H)$, and assume Hypotheses \ref{hyp}. Let $\chi_Q$, $Q\subgp E$, be a strong isotypy between $A$ and $B$. For each $Q\subgp E$ set
		\begin{equation*}
			\mu_Q=\Res_{C_G(\phi(Q))\times C_H(Q)}^{Y_Q}(\chi_Q)\in R_\K(C_G(\phi(Q))\times C_H(Q),e_{\phi(Q)}\tensor f_Q^\ast).
		\end{equation*}
		Then the characters $\mu_Q$, $Q\subgp E$, form an isotypy between $A$ and $B$.
	\end{theorem}

	%in proof below, prove \mu_Q is perfect isometry without using p-perm equivalence?
	
	%double check this proof, including making sure that x and y are used instead of u and v in the correct places.

	\begin{proof}
		Let $Q\subgp E$ and set $P=\phi(Q)$. Then $\mu_Q$ is a perfect isometry by Lemma \ref{lem:strongisocomesfrompperm} and \cite[Proposition 11.9]{Boltje_2020}. If $(g,h)\in G\times H$ is such that
		\begin{equation*}
			{}^{(g,h)}(\Delta(P,\phi,Q),e_P\tensor f_Q^\ast)\leq(\Delta(D,\phi,E),e_D\tensor f_E^\ast)
		\end{equation*}
		then $\phi({}^hQ)={}^gP$, ${}^g(P,e_P)=({}^gP,e_{{}^gP})$, ${}^h(Q,f_Q)=({}^hQ,f_{{}^hQ})$, and $c_g\phi=\phi c_h:Q\isoto {}^gP$, so by (1) of Definition \ref{defn:strongisotypy} we have ${}^{(g,h)}\chi_Q=\chi_{{}^hQ}$. It follows that ${}^{(g,h)}\mu_Q=\mu_{{}^hQ}$. We have thus shown that the characters $\mu_Q$, $Q\subgp E$, satisfy the equivariance condition of Definition \ref{defn:isotypy}, and all that remains to check is that the compatiblity condition also holds.
		
		Continue to let $Q\subgp E$ and $P=\phi(Q)$. Let $v\in C_E(Q)$ and set $u=\phi(v)\in C_D(P)$. It suffices to check that the compatiblity condition holds in the case where $Q$ is fully $\mathcal{B}$-centralized. Indeed, let $Q'$ be a fully $\mathcal{B}$-centralized subgroup that is isomorphic in $\mathcal{B}$ to $Q$. Then $Q'$ is receptive, so there exists an element $h\in H$ such that ${}^h(Q,f_Q)=(Q',f_{Q'})$ and ${}^h(Q\gp{v},f_{Q\gp{v}})\leq(E,f_E)$. Notice that $v':={}^hv\in C_E(Q')$. If we set $P'=\phi(Q')$ then since $\phi:\mathcal{B}\isoto\mathcal{A}$ is an isomorphism there exists an element $g\in G$ such that ${}^g(P,e_P)=(P',e_{P'})$, ${}^g(P\gp{u},e_{P\gp{u}})\leq(D,e_D)$, and $c_g=\phi c_h\phi^{-1}:P\gp{u}\isoto P'\gp{{}^gu}$. Note that $u':={}^gu=\phi({}^hv)\in C_D(P')$. We have
		\begin{equation*}
			{}^{(g,h)}\mu_Q=\mu_{Q'}\qquad\text{and}\qquad {}^{(g,h)}\mu_{Q\gp{v}}=\mu_{Q'\gp{v'}}
		\end{equation*}
		by equivariance. Now if
		\begin{equation*}
			d_{C_G(P')}^{u',e_{P'\gp{u'}}}\circ(\mu_{Q'}\tensor_{C_H(Q')}-)=(\mu_{Q'\gp{v'}}\tensor_{C_H(Q'\gp{v'})}-)\circ d_{C_H(Q')}^{v',f_{Q'\gp{v'}}}
		\end{equation*}
		then the formulas given in Lemma \ref{lem:conjandtensor} and Corollary \ref{cor:conjandgendecomp} show that the diagram (\ref{ntscommutes}) commutes. Thus we are reduced to the case where $Q$ is fully $\mathcal{B}$-centralized. In this case $P$ is fully $\mathcal{A}$-centralized, $(C_D(P),e_{PC_D(P)})$ is a maximal $\OO C_G(P)e_P$-Brauer pair, and $(C_E(Q),f_{QC_E(Q)})$ is a maximal $\OO C_H(Q)f_Q$-Brauer pair. If $U\subgp C_D(P)$ then $e_{PU}$ is the unique block idempotent of $\OO C_{C_G(P)}(U)=\OO C_G(PU)$ for which $(U,e_{PU})\leq(C_D(P),e_{PC_D(P)})$ is a containment of $\OO C_G(P)e_P$-Brauer pairs. Likewise if $V\subgp C_E(Q)$ then $(V,f_{QV})\leq(C_E(Q),f_{QC_E(Q)})$ is a containment of $\OO C_H(Q)f_Q$-Brauer pairs. By \cite[Theorem 3.19(b)]{Aschbacher_2011} we have
		\begin{equation*}
			\mathcal{C}_{\mathcal{A}}(P)=\mathcal{F}_{(C_D(P),e_{PC_D(P)})}(C_G(P),e_P)
		\end{equation*}
		and
		\begin{equation*}
			\mathcal{C}_{\mathcal{B}}(Q)=\mathcal{F}_{(C_E(Q),f_{QC_E(Q)})}(C_H(Q),f_Q).
		\end{equation*}
		In particular, the isomorphism $\phi:\mathcal{B}\isoto\mathcal{A}$ restricts to an isomorphism $\phi:\mathcal{F}_{(C_E(Q),f_{QC_E(Q)})}(C_H(Q),f_Q)\to \mathcal{F}_{(C_D(P),e_{PC_D(P)})}(C_G(P),e_P)$. Let $\mathcal{V}$ denote a set of representatives for the $\mathcal{C}_{\mathcal{B}}(Q)$-conjugacy classes of elements of $C_E(Q)$ such that $v\in\mathcal{V}$. Set $\mathcal{U}=\phi(\mathcal{V})$. Then $\mathcal{U}$ is a set of representatives for the $\mathcal{C}_{\mathcal{A}}(P)$-conjugacy classes of $C_D(P)$ and $u\in\mathcal{U}$. By Proposition \ref{prop:Brouediagram} the diagram below commutes:
		\begin{equation}\label{ntscommutes3}
			% https://q.uiver.app/#q=WzAsNCxbMCwwLCJDRihDX0goUSksZl9RO1xcSykiXSxbMiwwLCJDRihDX0coUCksZV9QO1xcSykiXSxbMCwxLCJcXGJpZ29wbHVzX3t2XFxpblxcbWF0aGNhbHtWfX1DRl97cCd9KENfSChRXFxncHt2fSksZl97UVxcZ3B7dn19O1xcSykiXSxbMiwxLCJcXGJpZ29wbHVzX3t1XFxpblxcbWF0aGNhbHtVfX1DRl97cCd9KENfRyhQXFxncHt1fSksZV97UFxcZ3B7dX19O1xcSykiXSxbMCwxLCJcXG11X1FcXHRlbnNvcl97Q19IKFEpfS0iXSxbMiwzXSxbMCwyLCJcXGRzdW1fe3ZcXGluXFxtYXRoY2Fse1Z9fWRfe0NfSChRKX1ee3YsZl97UVxcZ3B7dn19fSIsMl0sWzEsMywiXFxkc3VtX3t1XFxpblxcbWF0aGNhbHtVfX1kX3tDX0coUCl9Xnt1LGVfe1BcXGdwe3V9fX0iXV0=
			\begin{tikzcd}
				{CF(C_H(Q),f_Q;\K)} && {CF(C_G(P),e_P;\K)} \\
				{\bigoplus_{y\in\mathcal{V}}CF_{p'}(C_H(Q\gp{y}),f_{Q\gp{y}};\K)} && {\bigoplus_{x\in\mathcal{U}}CF_{p'}(C_G(P\gp{x}),e_{P\gp{x}};\K)}
				\arrow["{\mu_Q\tensor_{C_H(Q)}-}", from=1-1, to=1-3]
				\arrow[from=2-1, to=2-3]
				\arrow["{\dsum_{y\in\mathcal{V}}d_{C_H(Q)}^{y,f_{Q\gp{y}}}}"', from=1-1, to=2-1]
				\arrow["{\dsum_{x\in\mathcal{U}}d_{C_G(P)}^{x,e_{P\gp{x}}}}", from=1-3, to=2-3]
			\end{tikzcd}
		\end{equation}
		where the lower horizontal map is defined by
		\begin{equation*}
			\sum_{y\in\mathcal{V}}\psi_y\mapsto\sum_{x\in\mathcal{U}}\sum_{y\in\mathcal{V}}d_{C_G(P)\times C_H(Q)}^{(x,y),e_{P\gp{x}}\tensor f_{Q\gp{y}}^\ast}(\mu_Q)\tensor_{C_H(Q\gp{y})}\psi_y.
		\end{equation*}
		Now let $x\in\mathcal{U}$ and let $y\in\mathcal{V}$. Let $\epsilon_{\gp{(x,y)}}$ denote the $C_{Y_Q}(x,y)$-orbit sum of $e_{P\gp{x}}\tensor f_{Q\gp{y}}^\ast$. Then by Lemma \ref{lem:normalizerbrauerpairs} $((x,y),\epsilon_{\gp{(x,y)}})\in\BE_\OO(Y_Q,e_P\tensor f_Q^\ast)$ and
		\begin{equation*}
			(\gp{(x,y)},\epsilon_{\gp{(x,y)}})\leq(N_{D\times E}(\Delta(P,\phi,Q)),e_{N_D(P)}\tensor f_{N_E(Q)}^\ast).
		\end{equation*}
		Therefore, by condition (2a) of Definition \ref{defn:strongisotypy}, if $x=\phi(y)$ then
		\begin{equation*}
			d_{Y_Q}^{(x,y),\epsilon_{\gp{(x,y)}}}(\chi_Q)=d_{C_{Y_Q}(x,y)}(\Ind_{C_{Y_Q\cap Y_{Q\gp{y}}}(x,y)}^{C_{Y_Q}(x,y)}(\Res_{C_{Y_Q\cap Y_{Q\gp{y}}}(x,y)}^{Y_{Q\gp{y}}}(\chi_{Q\gp{y}}))).
		\end{equation*}
		If $x\neq\phi(y)$ then by condition (2b) of Definition \ref{defn:strongisotypy} and Lemma \ref{lem:moreon2b} we have
		\begin{equation*}
			d_{Y_Q}^{(x,y),\epsilon_{\gp{(x,y)}}}(\chi_Q)=0.
		\end{equation*}
		Next we observe that
		\begin{equation*}
			\Res_{C_G(P\gp{x})\times C_H(Q\gp{y})}^{C_{Y_Q}(x,y)}(d_{Y_Q}^{(x,y),\epsilon_{\gp{(x,y)}}}(\chi_Q))
		\end{equation*}
		is a class function in $CF_{p'}(C_G(P\gp{x})\times C_H(Q\gp{y}),\epsilon_{\gp{(x,y)}};\K)$ whose evaluation at a $p'$-element $(s,t)\in[C_G(P\gp{x})\times C_H(Q\gp{y})]_{p'}$ is equal to
		\begin{align*}
			\chi_Q((xs,yt)\epsilon_{\gp{(x,y)}})	&=\sum_{e\tensor f^\ast\in\Orb_{C_{Y_Q}(x,y)}(e_{P\gp{x}}\tensor f_{Q\gp{y}}^\ast)}\chi_Q((xs,yt)(e\tensor f^\ast))\\
			&=\sum_{e\tensor f^\ast\in\Orb_{C_{Y_Q}(x,y)}(e_{P\gp{x}}\tensor f_{Q\gp{y}}^\ast)}d_{C_G(P)\times C_H(Q)}^{(x,y),e\tensor f^\ast}(\Res_{C_G(P)\times C_H(Q)}^{Y_Q}(\chi_Q))(s,t)\\
			&=\sum_{e\tensor f^\ast\in\Orb_{C_{Y_Q}(x,y)}(e_{P\gp{x}}\tensor f_{Q\gp{y}}^\ast)}d_{C_G(P)\times C_H(Q)}^{(x,y),e\tensor f^\ast}(\mu_Q)(s,t).
		\end{align*}
		It follows that
		\begin{equation*}
			(e_{P\gp{x}}\tensor f_{Q\gp{y}}^\ast)\Res_{C_G(P\gp{x})\times C_H(Q\gp{y})}^{C_{Y_Q}(x,y)}(d_{Y_Q}^{(x,y),\epsilon_{\gp{(x,y)}}}(\chi_Q))=d_{C_G(P)\times C_H(Q)}^{(x,y),e_{P\gp{x}}\tensor f_{Q\gp{y}}^\ast}(\mu_Q).
		\end{equation*}
		In particular, if $x\neq\phi(y)$ then
		\begin{equation*}
			d_{C_G(P)\times C_H(Q)}^{(x,y),e_{P\gp{x}}\tensor f_{Q\gp{y}}^\ast}(\mu_Q)=0
		\end{equation*}
		and if $x=\phi(y)$ then one computes via the Mackey formula that
		\begin{equation*}
			d_{C_G(P)\times C_H(Q)}^{(x,y),e_{P\gp{x}}\tensor f_{Q\gp{y}}^\ast}(\mu_Q)=d_{C_G(P\gp{x})\times C_H(Q\gp{y})}(\mu_{Q\gp{y}})
		\end{equation*}
		It follows that the lower horizontal map in the diagram (\ref{ntscommutes3}) maps a class function $\psi_y\in CF_{p'}(C_H(Q\gp{y}),f_{Q\gp{y}};\K)$ to
		\begin{align*}
			\psi_y	&\mapsto \sum_{x\in\mathcal{U}}d_{C_G(P)\times C_H(Q)}^{(x,y),e_{P\gp{x}}\tensor f_{Q\gp{y}}^\ast}(\mu_Q)\tensor_{C_H(Q\gp{y})}\psi_y\\
			&=d_{C_G(P\gp{\phi(y)})\times C_H(Q\gp{y})}(\mu_{Q\gp{y}})\tensor\psi_y\\
			&=\mu_{Q\gp{y}}\tensor_{C_H(Q\gp{y})}\psi_y.
		\end{align*}
		Letting $x=u$ and $y=v$, the above implies that the diagram (\ref{ntscommutes}) of the compatibility condition commutes. The proof is complete.
	\end{proof}

\section{Commutative diagrams from a strong isotypy}

	Continue to let $G$ and $H$ be finite groups, $(\K,\OO,F)$ a $p$-modular system large enough for $G\times H$, $A$ a block of $\OO G$, and let $B$ be a block of $\OO H$. Assume Hypotheses \ref{hyp} and the notation set there. In this section we describe a commutative diagram that is induced from a strong isotypy and which extends the commutative diagram of the compatibility condition in the definition of ``isotypy.''
	
	\begin{lemma}\label{lem:commdiagramattrivlevel}
		Let $A\in\Bl(\OO G)$ and $B\in\Bl(\OO H)$. Assume Hypotheses \ref{hyp} and let $\chi_Q$, $Q\subgp E$, be a strong isotypy between $A$ and $B$. Let $v\in E$ and set $u=\phi(v)$. Then the diagram below commutes:
		\begin{equation}\label{commdia}
			% https://q.uiver.app/#q=WzAsNCxbMCwwLCJDRihILEI7XFxLKSJdLFs0LDAsIkNGKEcsQTtcXEspIl0sWzAsMSwiQ0Zfe3AnfShDX0godiksZl97XFxncHt2fX07XFxLKSJdLFs0LDEsIkNGX3twJ30oQ19HKHUpLGVfe1xcZ3B7dX19O1xcSykiXSxbMCwxLCJcXGNoaV9cXGdhbW1hXFx0ZW5zb3JfSC0iXSxbMCwyLCJkX0hee3YsZl97XFxncHt2fX19IiwyXSxbMSwzLCJkX0dee3UsZV97XFxncHt1fX19Il0sWzIsMywiXFxSZXNfe0NfRyh1KVxcdGltZXMgQ19IKHYpfV57WV97XFxncHt2fX19KFxcY2hpX3tcXGdwe3Z9fSlcXHRlbnNvcl97Q19IKHYpfS0iLDJdXQ==
			\begin{tikzcd}
				{CF(H,B;\K)} &&&& {CF(G,A;\K)} \\
				{CF_{p'}(C_H(v),f_{\gp{v}};\K)} &&&& {CF_{p'}(C_G(u),e_{\gp{u}};\K)}
				\arrow["{\chi_{\set{1}}\tensor_H-}", from=1-1, to=1-5]
				\arrow["{d_H^{v,f_{\gp{v}}}}"', from=1-1, to=2-1]
				\arrow["{d_G^{u,e_{\gp{u}}}}", from=1-5, to=2-5]
				\arrow["{\Res_{C_G(u)\times C_H(v)}^{Y_{\gp{v}}}(\chi_{\gp{v}})\tensor_{C_H(v)}-}"', from=2-1, to=2-5]
			\end{tikzcd}
		\end{equation}
		Moreover, $\chi_{\set{1}}$ and $\Res_{C_G(u)\times C_H(v)}^{Y_{\gp{v}}}(\chi_{\gp{v}})$ are perfect isometries.
	\end{lemma}

	\begin{proof}
		This is a consequence of Theorem \ref{thm:strongisotimpliesisot} and an application of the compatibility condition in Definition \ref{defn:isotypy} with $Q=\set{1}$.
	\end{proof}

	\begin{theorem}\label{thm:mydiagrams}
		Let $A\in\Bl(\OO G)$ and $B\in\Bl(\OO H)$. Assume Hypotheses \ref{hyp} and let $\chi_Q$, $Q\subgp E$, be a strong isotypy between $A$ and $B$. Fix a subgroup $Q\subgp E$ and set $P=\phi(Q)$. For each subgroup $U\subgp N_D(P)$ write $\epsilon_U$ for the unique block idempotent of $\OO C_{I_P}(U)$ such that $(U,\epsilon_U)\leq(N_D(P),e_{N_D(P)})$ is a containment of $\OO I_Pe_P$-Brauer pairs and for each $V\subgp N_E(Q)$ write $\varphi_V$ for the unique block idempotent of $\OO C_{J_Q}(V)$ such that $(V,\varphi_V)\leq(N_E(Q),f_{N_E(Q)})$ is a containment of $\OO J_Qf_Q$-Brauer pairs (see Lemma \ref{lem:normalizerbrauerpairs}). Let $v\in N_E(Q)$ and set $u=\phi(v)\in N_D(P)$. Let $\chi_{Q\gp{v}}'$ denote the character
		\begin{equation*}
			\chi_{Q\gp{v}}'=\Ind_{C_{Y_Q\cap Y_{Q\gp{v}}}(u,v)}^{C_{I_P}(u)\times C_{J_Q}(v)}\Res_{C_{Y_Q\cap Y_{Q\gp{v}}}(u,v)}^{Y_{Q\gp{v}}}\chi_{Q\gp{v}}.
		\end{equation*}
		Then $\Ind_{Y_Q}^{I_P\times J_Q}(\chi_Q)$ and $\chi_{Q\gp{v}}'$ are perfect isometries and the diagram below commutes:
		% https://q.uiver.app/#q=WzAsNCxbMCwwLCJDRihKX1EsZl9RO1xcSykiXSxbNCwwLCJDRihJX1AsZV9QO1xcSykiXSxbMCwyLCJDRl97cCd9KENfe0pfUX0odiksXFx2YXJwaGlfe1xcZ3B7dn19O1xcSykiXSxbNCwyLCJDRl97cCd9KENfe0lfUH0odSksXFxlcHNpbG9uX3tcXGdwe3V9fTtcXEspIl0sWzAsMSwiXFxJbmRfe1lfUX1ee0lfUFxcdGltZXMgSl9RfShcXGNoaV9RKVxcdGVuc29yX3tKX1F9LSJdLFswLDIsImRfe0pfUX1ee3YsXFx2YXJwaGlfe1xcZ3B7dn19fSIsMl0sWzEsMywiZF97SV9QfV57dSxcXGVwc2lsb25fe1xcZ3B7dX19fSJdLFsyLDMsIlxcY2hpX3tRXFxncHt2fX0nXFx0ZW5zb3Jfe0Nfe0pfUX0odil9LSIsMl1d
		\begin{equation}\label{mydiagram}\begin{tikzcd}
				{CF(J_Q,f_Q;\K)} &&&& {CF(I_P,e_P;\K)} \\
				\\
				{CF_{p'}(C_{J_Q}(v),\varphi_{\gp{v}};\K)} &&&& {CF_{p'}(C_{I_P}(u),\epsilon_{\gp{u}};\K)}
				\arrow["{\Ind_{Y_Q}^{I_P\times J_Q}(\chi_Q)\tensor_{J_Q}-}", from=1-1, to=1-5]
				\arrow["{d_{J_Q}^{v,\varphi_{\gp{v}}}}"', from=1-1, to=3-1]
				\arrow["{d_{I_P}^{u,\epsilon_{\gp{u}}}}", from=1-5, to=3-5]
				\arrow["{\chi_{Q\gp{v}}'\tensor_{C_{J_Q}(v)}-}"', from=3-1, to=3-5]
		\end{tikzcd}\end{equation}
		Moreover, if $v\in C_E(Q)$ then the diagrams (\ref{ntscommutes}) and (\ref{mydiagram}) form a commutative cube:
		\begin{equation*}
			\begin{tikzcd}[row sep=tiny, column sep=tiny, font=\tiny]
				& CF(C_H(Q),f_Q;\K) \arrow[rr] \arrow[dd] & & CF(C_G(P),e_P;\K) \arrow[dd] \\
				CF(J_Q,f_Q;\K) \arrow[ur, "\Res"] \arrow[rr, crossing over] \arrow[dd] & & CF(I_P,e_P;\K) \arrow[ur, "\Res"'] \\
				& CF_{p'}(C_H(Q\gp{v},f_{Q\gp{v}};\K) \arrow[rr] & & CF_{p'}(C_G(P\gp{u}),e_{P\gp{u}};\K) \\
				CF_{p'}(C_{J_Q}(v),\varphi_{\gp{v}};\K) \arrow[ur, "f_{Q\gp{v}}\Res" near end] \arrow[rr] & & CF_{p'}(C_{I_P}(u),\epsilon_{\gp{u}};\K) \arrow[ur, "e_{P\gp{u}}\Res"' near end] \arrow[from=uu, crossing over]\\
			\end{tikzcd}
		\end{equation*}
	\end{theorem}

	%missing details in proof below, including computation that \Res\theta_{\gp{v}}=\chi_{Q\gp{v}}'. Also, why are the maps perfect isometries?

	\begin{proof}
		We may assume without loss of generality that $Q$ is fully $\mathcal{B}$-normalized. Then $P$ is fully $\mathcal{A}$-normalized. By Lemma \ref{lem:normalizerbrauerpairs} $(N_D(P),e_{N_D(P)})$ is a maximal $\OO I_Pe_P$-Brauer pair and $(N_E(Q),f_{N_E(Q)})$ is a maximal $\OO J_Qf_Q$-Brauer pair. Furthermore,
		\begin{align*}
			\mathcal{N}_{\mathcal{A}}(P)	&=\mathcal{F}_{(N_D(P),e_{N_D(P)})}(I_P,e_P)\\
			\mathcal{N}_{\mathcal{B}}(Q)	&=\mathcal{F}_{(N_E(Q),f_{N_E(Q)})}(J_Q,f_Q),
		\end{align*}
		and $\phi:\mathcal{B}\isoto\mathcal{A}$ restricts to an isomorphism $\phi:\mathcal{N}_{\mathcal{B}}(Q)\isoto\mathcal{N}_{\mathcal{A}}(P)$.
		
		Let $\gamma\in T^\Delta(A,B)$ be the unique $p$-permutation equivalence with maximal $\gamma$-Brauer pair $(\Delta(D,\phi,E),e_D\tensor f_E^\ast)$ mapping to $(\chi_Q)_{Q\subgp E}$ under the bijection in Theorem \ref{thm:ppermequivtosi}. Set
		\begin{equation*}
			\gamma_Q=\Ind_{Y_Q}^{I_P\times J_Q}(\gamma(\Delta(P,\phi,Q),e_P\tensor f_Q^\ast))\in T^\Delta(\OO I_Pe_P,\OO J_Qf_Q).
		\end{equation*}
		Then $\gamma_Q$ is a $p$-permutation equivalence between $\OO I_Pe_P$ and $\OO J_Qf_Q$ by \cite[Theorem 11.4]{Boltje_2020}, and $(\Delta(N_D(P),\phi,N_E(Q)),e_{N_D(P)}\tensor f_{N_E(Q)}^\ast)$ is a maximal $\gamma_Q$-Brauer pair by \cite[Proposition 11.5(b)]{Boltje_2020}. For each $V\subgp N_E(Q)$ set
		\begin{equation*}
			Z_V=N_{I_P\times J_Q}(\Delta(\phi(V),\phi,V),\epsilon_{\phi(V)}\tensor\varphi_V^\ast)
		\end{equation*}
		and set
		\begin{equation*}
			\theta_V=\chi_{\gamma_Q(\Delta(\phi(V),\phi,V),\epsilon_{\phi(V)}\tensor\varphi_V^\ast)}.
		\end{equation*}
		Then by Lemma \ref{lem:charsofppermequiv} the characters $\theta_V$, $V\subgp N_E(Q)$ form a strong isotypy between $\OO I_Pe_P$ and $\OO J_Qf_Q$. Therefore by Lemma \ref{lem:commdiagramattrivlevel} the diagram below commutes:
		% https://q.uiver.app/#q=WzAsNCxbMCwwLCJDRihKX1EsZl9RO1xcSykiXSxbNCwwLCJDRihJX1AsZV9QO1xcSykiXSxbMCwxLCJDRl97cCd9KENfe0pfUX0odiksXFx2YXJwaGlfe1xcZ3B7dn19O1xcSykiXSxbNCwxLCJDRl97cCd9KENfe0lfUH0odSksXFxlcHNpbG9uX3tcXGdwe3V9fTtcXEspIl0sWzAsMSwiXFx0aGV0YV97XFxzZXR7MX19XFx0ZW5zb3Jfe0pfUX0tIl0sWzIsMywiXFxSZXNfe0Nfe0lfUH0odSlcXHRpbWVzIENfe0pfUX0odil9XntaX3tcXGdwe3Z9fX0oXFx0aGV0YV97XFxncHt2fX0pXFx0ZW5zb3Jfe0Nfe0pfUX0odil9LSIsMl0sWzAsMiwiZF97Sl9RfV57dixcXHZhcnBoaV97XFxncHt2fX19IiwyXSxbMSwzLCJkX3tJX1B9Xnt1LFxcZXBzaWxvbl97XFxncHt1fX19Il1d
		\[\begin{tikzcd}
			{CF(J_Q,f_Q;\K)} &&&& {CF(I_P,e_P;\K)} \\
			{CF_{p'}(C_{J_Q}(v),\varphi_{\gp{v}};\K)} &&&& {CF_{p'}(C_{I_P}(u),\epsilon_{\gp{u}};\K)}
			\arrow["{\theta_{\set{1}}\tensor_{J_Q}-}", from=1-1, to=1-5]
			\arrow["{\Res_{C_{I_P}(u)\times C_{J_Q}(v)}^{Z_{\gp{v}}}(\theta_{\gp{v}})\tensor_{C_{J_Q}(v)}-}"', from=2-1, to=2-5]
			\arrow["{d_{J_Q}^{v,\varphi_{\gp{v}}}}"', from=1-1, to=2-1]
			\arrow["{d_{I_P}^{u,\epsilon_{\gp{u}}}}", from=1-5, to=2-5]
		\end{tikzcd}\]
		The commutativity of the diagram (\ref{mydiagram}) follows after noting that $\theta_{\set{1}}=\Ind_{Y_Q}^{I_P\times J_Q}(\chi_Q)$ and $\Res_{C_{I_P}(u)\times C_{J_Q}(v)}^{Z_{\gp{v}}}(\theta_{\gp{v}})=\chi_{Q\gp{v}}'$. The final statement regarding the ``commutative cube'' is easy to verify.
	\end{proof}
	
	%Bibliography
	
	\bibliographystyle{plain}
	\bibliography{../../../../../Bibliography/bibliography}
	
\end{document}